\documentclass[11pt]{article}

\usepackage{fullpage}
\usepackage{authblk}
\usepackage{natbib}
\usepackage{algorithm}
\usepackage{algpseudocode}
\usepackage{amsmath,amsthm,amsfonts}
\usepackage{amsmath}
\usepackage[subnum]{cases}
\usepackage[most]{tcolorbox}
\usepackage{mathrsfs}
\usepackage{amssymb}
\usepackage{color}
\usepackage{mathrsfs}
\usepackage{empheq}
\usepackage{enumitem}
\usepackage{bm}
\usepackage{multirow}
\usepackage{booktabs}
\usepackage{makecell}
\usepackage{graphicx}
\usepackage{subcaption}
\usepackage{comment}
\usepackage{cases}
\usepackage{appendix}
\usepackage{tikz}
\usetikzlibrary{arrows,shapes}
\usepackage[colorlinks= true, linkcolor=brown, citecolor=blue, urlcolor=black]{hyperref}
\usepackage{cleveref}
\usepackage{marginnote}
\usepackage{diagbox} 

\numberwithin{equation}{section}

\theoremstyle{definition}

\newtheorem{theorem}{Theorem}[section]
\newtheorem{lemma}[theorem]{Lemma}

\newtheorem{remark}[theorem]{Remark}

\newtheorem{condition}[theorem]{Condition}
\newtheorem{prop}[theorem]{Proposition}
\newtheorem{defn}[theorem]{Definition}

\newcommand{\defi}{\overset{\triangle}{=}}

\usepackage{booktabs}
\usepackage{pgfplots}
\usetikzlibrary{calc,intersections}
\usepackage{tikz}
\usetikzlibrary{patterns}

\usetikzlibrary{shapes.geometric, arrows, positioning}

\tikzset{
	mybox/.style  = {draw, rectangle, minimum width=2.0cm, minimum height=0.8cm, text centered, text width=2.0cm,   
		font=\normalsize,  align=left},
	box/.style  = {draw, rectangle, minimum width=2.0cm, minimum height=0.6cm, text centered, text width=7.4cm,   
		font=\normalsize,  align=left},		
	box1/.style  = {draw, rectangle, minimum width=2.0cm, minimum height=0.6cm, text centered, text width=5.6cm,   
		font=\normalsize},	
	box2/.style  = {draw, rectangle, minimum width=2.0cm, minimum height=0.6cm, text centered, text width=7.4cm,   
		font=\normalsize,  align=left},	
	myarrow/.style = {line width=0.2pt, draw=black, -triangle 60, postaction={draw, line width=0.2pt, shorten >=10pt,-}}	
}

\tikzstyle{arrow} = [->, >=stealth, -triangle 60]

\allowdisplaybreaks

\makeatletter
\newcommand{\leqnomode}{\tagsleft@true}
\newcommand{\reqnomode}{\tagsleft@false}
\makeatother

\begin{document}

\title{Numerical Construction of Quasi-Periodic Solutions Beyond Symplectic Integrators}
\author[1]{Mingwei Fu}
\author[2,3]{Bin Shi\thanks{Corresponding author: \url{binshi@fudan.edu.cn} } }
\affil[1]{School of Mathematical Sciences, University of Chinese Academy of Sciences, Beijing 100049, China}
\affil[2]{Center for Mathematics and Interdisciplinary Sciences, Fudan University, Shanghai 200433, China}
\affil[3]{Shanghai Institute for Mathematics and Interdisciplinary Sciences, Shanghai 200433, China}

\date\today

\maketitle

\begin{abstract}

Symplectic integrators are the established standard for long-term simulations of nearly-integrable Hamiltonian systems due to their preservation of geometric structures. However, they suffer from an inherent limitation: secular phase-shift errors. While the qualitative ``shape'' of invariant tori is preserved, the numerical solution gradually drifts along the torus, leading to a phase-lag accumulation that degrades long-term positional accuracy. Inspired by the Craig-Wayne-Bourgain (CWB) scheme, originally developed as an analytical tool for infinite-dimensional systems, we introduce a numerical operator that incorporates frequency updates into a dimension-enlarged Newton iteration to compute quasi-periodic solutions. Unlike conventional time-stepping integrators, our alternating numerical procedure eliminates phase-lag accumulation by directly solving for instantaneous positions and phase angles. Theoretically, provided sufficient computational resources, the phase error can be reduced arbitrarily, remaining independent of the total integration time. Our algorithm translates the Nash-Moser iteration into a practical numerical framework, marking a significant departure from traditional Kolmogorov-Arnold-Moser (KAM) theory. While KAM provides rigorous existence proofs, its requirement for global Diophantine conditions and the total exclusion of resonant sets render it numerically inaccessible. By employing a ``step-by-step'' exclusion process and incrementally enlarging the dimension, our algorithm resolves irrationality conditions locally. This approach demonstrates that the ``numerical irrationality problem'' is not an intrinsic barrier to computation, offering a constructive, executable alternative to the non-executable nature of global KAM-based methods.

\end{abstract}

\section{Introduction}
\label{sec: intro}

The Hamiltonian formulation provides the modern geometric framework for classical mechanics, describing the state of a conservative system by canonical coordinates $\pmb{x} = (x_1, \ldots, x_n)$ and $\pmb{y} = (y_1, \ldots, y_n)$ defined in a $2n$-dimensional phase space. The time evolution is governed by Hamilton's canonical equations
\begin{subequations}
\label{eqn: hamilton-eqn}
\begin{empheq}[left=\empheqlbrace]{align} 
         &  \dot{\pmb{x}} =  -  \frac{\partial H}{\partial \pmb{y}},                                                  \label{eqn: h-q}            \\   
         &  \dot{\pmb{y}} =     \frac{\partial H}{\partial \pmb{x}},                                                   \label{eqn:h-p}            
\end{empheq}    
\end{subequations}
where the Hamiltonian function $H = H(\pmb{x} ; \pmb{y} ) = H(x_1, \ldots, x_n; y_1, \ldots, y_n)$ typically represents the total energy of the system and acts as the generator of the dynamics. Beyond a reformulation of Newton’s laws, the Hamiltonian framework is inherently geometric. The phase space is naturally endowed with a canonical symplectic structure $\omega = d\pmb{x} \wedge d\pmb{y} = \sum_{j = 1}^{n} dx_j \wedge dy_j$ and the resulting flow $\varphi_t: (\pmb{x}(0), \pmb{y}(0)) \mapsto  (\pmb{x}(t), \pmb{y}(t))$ is a symplectomorphism that preserves the symplectic two-form exactly for all times~\citep{arnol2013mathematical}. This geometric structure underpins the characteristic long-time behavior of conservative systems, including phase‑space volume preservation and the existence of invariant tori in integrable systems.

For long-time numerical simulations, it is therefore of fundamental importance to preserve this symplectic structure under time discretization.  Symplectic integrators are numerical schemes specifically designed for this purpose: their one-step update maps are symplectic transformations of phase space, ensuring that the discrete flow inherits the essential geometric structure of the continuous Hamiltonian system. The formal development of symplectic integrators began with~\citet{feng1986difference}, who provided the first rigorous proof that such schemes can be systematically constructed as symplectic maps. From an algorithmic viewpoint, early developments in structure-preserving discretizations can be traced back to~\citet{devogelaere1956methods}, with practical higher-order methods later introduced by~\citet{ruth1983canonical} and~\citet{feng1985difference}.   By construction, symplectic integrators exhibit remarkable long-time stability and qualitative accuracy, most notably the absence of secular energy drift and the faithful reproduction of phase-space topology. The theoretical foundation is further clarified by backward error analysis, which interprets the numerical solution as the exact flow of a nearby modified Hamiltonian system. The systematic study of backward error analysis for structure-preserving methods was initiated by~\citet{feng1991formal}. Sharp characterizations of truncation errors were subsequently developed by~\citet{yoshida1993recent}, while long-time backward error analysis and lifespan estimates were established by~\citet{hairer1997life}; see also the comprehensive review by~\citet{hairer2003geometric}. More recently,~\citet{jordan2018dynamical} advocated symplectic integrators as a geometric framework for analyzing gradient-based optimization algorithms.

This geometric philosophy extends naturally to other structure-preserving problems.  For general divergence-free vector fields, such as those arising in incompressible fluid dynamics, the corresponding goal is to construct volume-preserving integrators. Building upon insights from the classical ABC flow,~\citet{feng1995volume} established a general framework for constructing such schemes, accompanied by rigorous proofs of convergence and structure-preservation.    Moreover,~\citet{hairer2000long} introduced a class of highly oscillatory systems motivated by variants of the Fermi-Pasta-Ulam (FPU) model, and investigated their long-time numerical energy behavior. Although such modified systems do not coincide exactly with the original physical models, they capture essential stiffness and multiscale features relevant to numerical analysis. The stiffness properties of ordinary differential equations were analyzed in detail by~\citet{hairer1996solving2}. Furthermore,~\citet{cohen2003modulated}  introduced the framework of modulated Fourier expansions, demonstrating that the harmonic energy associated with the highly oscillatory components is nearly conserved over time intervals that are exponentially long with respect to the dominant frequency.

A class of systems of paramount importance is that of nearly integrable Hamiltonian systems. In action–angle variables $(\pmb{I}, \pmb{\theta}) = (I_1, \ldots, I_n; \theta_1, \ldots, \theta_n) \in \mathbb{R}_+^n \times \mathbb{T}^n$, the Hamiltonian takes the form
\begin{equation}
\label{eqn: nearly-integrable-func}
H = H(\pmb{I} ; \pmb{\theta} ) = H_0(\pmb{I})+ \varepsilon H_1(\pmb{I}, \pmb{\theta}), \qquad 0 < \varepsilon \ll 1,
\end{equation}
where $I_{j} = (x_j^2 + y_j^2)/2$ and $ \theta_{j} = \arctan(y_j/x_j)$. The unperturbed Hamiltonian $H_0(\pmb{I})$ describes an integrable system with frequencies $\pmb{\omega}(\pmb{I}) := \nabla_{\pmb{I}}H_0(\pmb{I})$.  One of the most profound achievements in twentieth‑century dynamical systems theory is the Kolmogorov-Arnold-Moser (KAM) theorem~\citep{arnold1963small, moser1962invariant}, which asserts that, under suitable nondegeneracy and smoothness conditions, a large measure of invariant tori for the unperturbed system persists under small Hamiltonian perturbations, albeit slightly deformed. These surviving KAM tori confine the motion and provide a profound mechanism for long‑time stability, thereby explaining the absence of global chaos in weakly perturbed Hamiltonian systems. From a numerical perspective, the persistence of KAM tori can also be observed for symplectic integrators. Under appropriate step-size conditions,~\citet{shang1999kam, shang2000resonant} established rigorous KAM-type theorems for symplectic discretizations.  Nevertheless, as emphasized by~\citet{gauckler2018dynamics}, the design of numerical algorithms with favorable long-time dynamical properties remains a central challenge.

\subsection{Intrinsic limitation of symplectic integrators:  secular drift of angle variables}
\label{subsec: limit-symplectic}

We here consider the simple one-dimensional harmonic oscillator: $ \ddot{x} =  - x$. Introducing the auxiliary variable $y = - \dot{x}$, the system can be written in Hamiltonian form as
\begin{subequations}
\label{eqn: ham-os}
\begin{empheq}[left=\empheqlbrace]{align} 
         &  \dot{x} =  - y,                                                   \label{eqn: ham-os-x}            \\   
         &  \dot{y} =  x.                                                     \label{eqn: ham-os-y}            
\end{empheq}    
\end{subequations}
Equivalently, the dynamics are generated by the Hamiltonian $H = (x^2 + y^2) / 2$, equipped with the canonical symplectic two-form $\omega = dx \wedge dy$. The contours of the Hamiltonian are circles in phase space, corresponding to periodic orbits with constant energy. Starting from the initial condition $x(0) = 1$ and $ y(0) = \dot{x}(0) = 0$, the exact solution traces a circular trajectory  in the $(x,y)$-plane, rotating counterclockwise with constant angular frequency $\omega = 1$. Equivalently, the phase angle evolves as $\theta(t) = t$; see~\Cref{fig: symp-orbit}. 
\begin{figure}[htb!]
    \centering
    \includegraphics[scale=0.25]{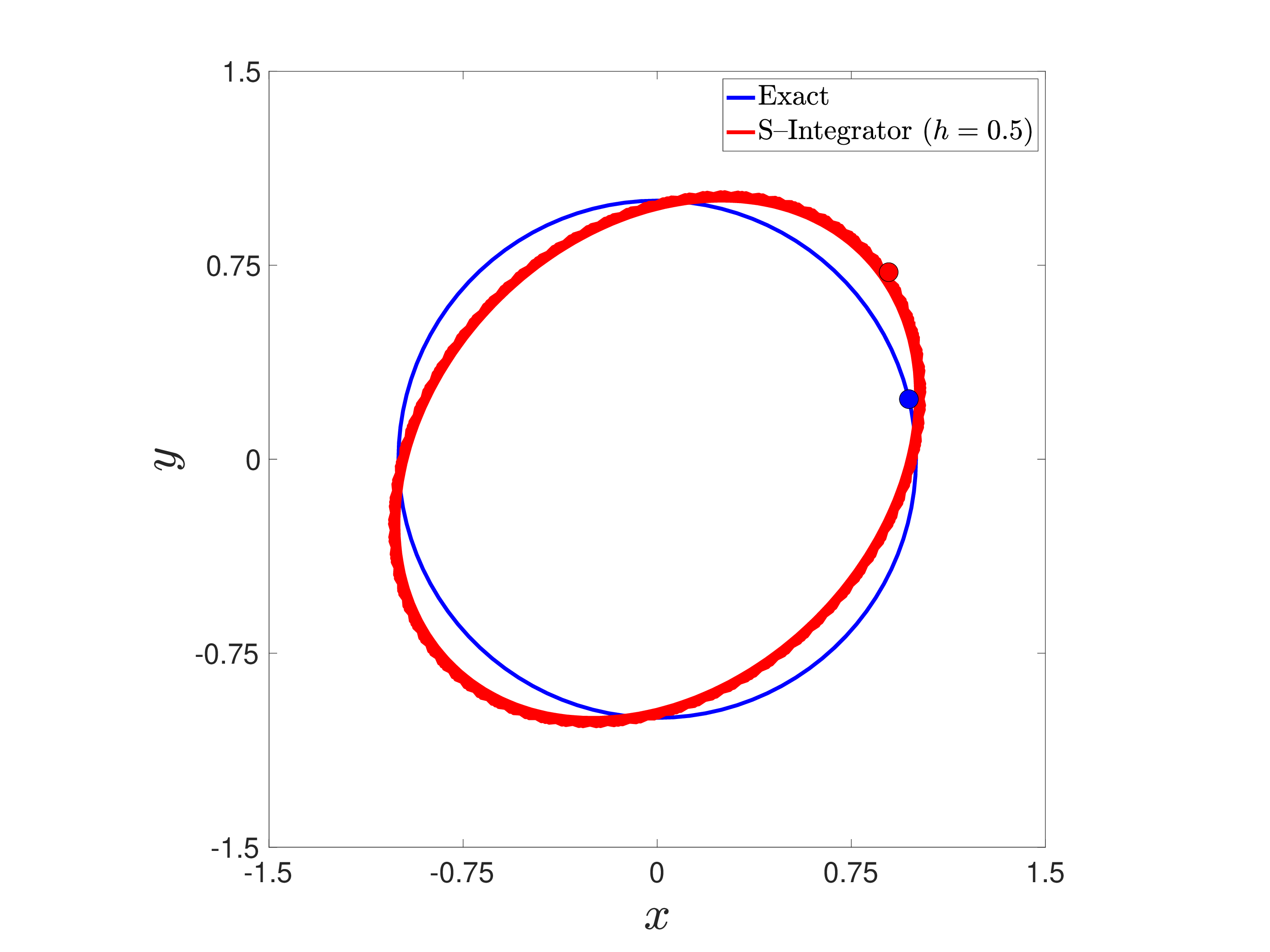}
    \caption{Phase space trajectories for both the exact solution of the Hamiltonian system~\eqref{eqn: ham-os} and the first-order symplectic integrator~\eqref{eqn: ham-os-sym}, starting from the initial $(x(0),y(0)) = (1,0)$; markers indicate the states at $t_n = nh$ with $n=101$.}
    \label{fig: symp-orbit}
\end{figure}
Given a step size $h>0$, we apply the simple first-order symplectic (forward-backward) integrator to the Hamilton system~\eqref{eqn: ham-os}:
\begin{subequations}
\label{eqn: ham-os-sym}
\begin{empheq}[left=\empheqlbrace]{align} 
         &  x_{n+1} - x_{n} =  - hy_{n},                                                   \label{eqn: ham-os-x-sym}            \\   
         &  y_{n+1} - y_{n} =    hx_{n+1},                                               \label{eqn: ham-os-y-sym}            
\end{empheq}    
\end{subequations}
which can be written in matrix form as
\[
    \begin{pmatrix} 
        x_{n+1} \\ 
        y_{n+1} 
    \end{pmatrix} 
    = 
    \begin{pmatrix} 
        1 & -h \\
        h & 1 - h^2 
    \end{pmatrix} 
    \begin{pmatrix} 
        x_n \\ 
        y_n 
    \end{pmatrix} 
    \defi M_h 
    \begin{pmatrix} 
        x_n \\ 
        y_n 
    \end{pmatrix}.
\]
The eigenvalues of the symplectic map $M_h$ are given by
\begin{equation}
    \label{eqn: eig-sym-integrator}
    \lambda_\pm = 1 - \frac{h^2}{2} \pm ih\sqrt{1 - \frac{h^2}{4}}.
\end{equation}
For any $ h \in [0, 2) $, these eigenvalues form a complex conjugate pair with unit modulus, $ |\lambda_\pm| = 1$,  implying that $M_h$ is area-preserving; see~\Cref{fig: symp-orbit}.


As illustrated in~\Cref{fig: symp-orbit}, although both the exact solution and the symplectic integrator generate closed circular orbits that remain close in phase space, their positions along the orbits, equivalently, their phase angles, do not coincide. To quantify the phase mismatch, we compute the numerical rotation angle per step, denoted by $\theta_h$. Since the eigenvalues of $M_h$ have unit modulus, there exists an invertible matrix $S_h$ such that
\begin{equation}
\label{eqn: m-h-new}
M_h = S_h^{-1}\begin{pmatrix}
           \cos \theta_h     &     -   \sin \theta_h   \\
           \sin \theta_h      &         \cos \theta_h
           \end{pmatrix} S_h.
\end{equation} 
Combining~\eqref{eqn: eig-sym-integrator} and~\eqref{eqn: m-h-new}, it is straightforward to verify that
\begin{equation}
\label{eqn: angle-step-size}
 \cos \theta_h =  1 - \frac{h^2}{2}, \qquad \sin \theta_h = h\sqrt{1 - \frac{h^2}{4}}.
\end{equation}
Since the exact solution advances the phase by $h$ per time step, we define the numerical phase error as $ \Delta \theta_h  := \theta_h - h$. For any $h \in [0, 2)$, a direct estimate from~\eqref{eqn: angle-step-size} yields:
\[
    \Delta \theta_h  = \theta_h - h = \arccos\left(1 - \frac{h^2}{2}\right) - h \geq \frac{1}{24}h^3.
\]
Consequently, after $n$ time steps, the accumulated phase error satisfies
\[
    n \Delta \theta_h \geq \frac{1}{24} n h^3 = \frac{1}{24} t_n h^2,
\]
where $t_n=nh$ denotes the physical time. This results in a secular drift in the numerical phase, as illustrated in~\Cref{fig: phase-difference}.

\begin{figure}[htb!]
\centering
\begin{subfigure}[t]{0.46\linewidth}
\centering
 \includegraphics[scale=0.16]{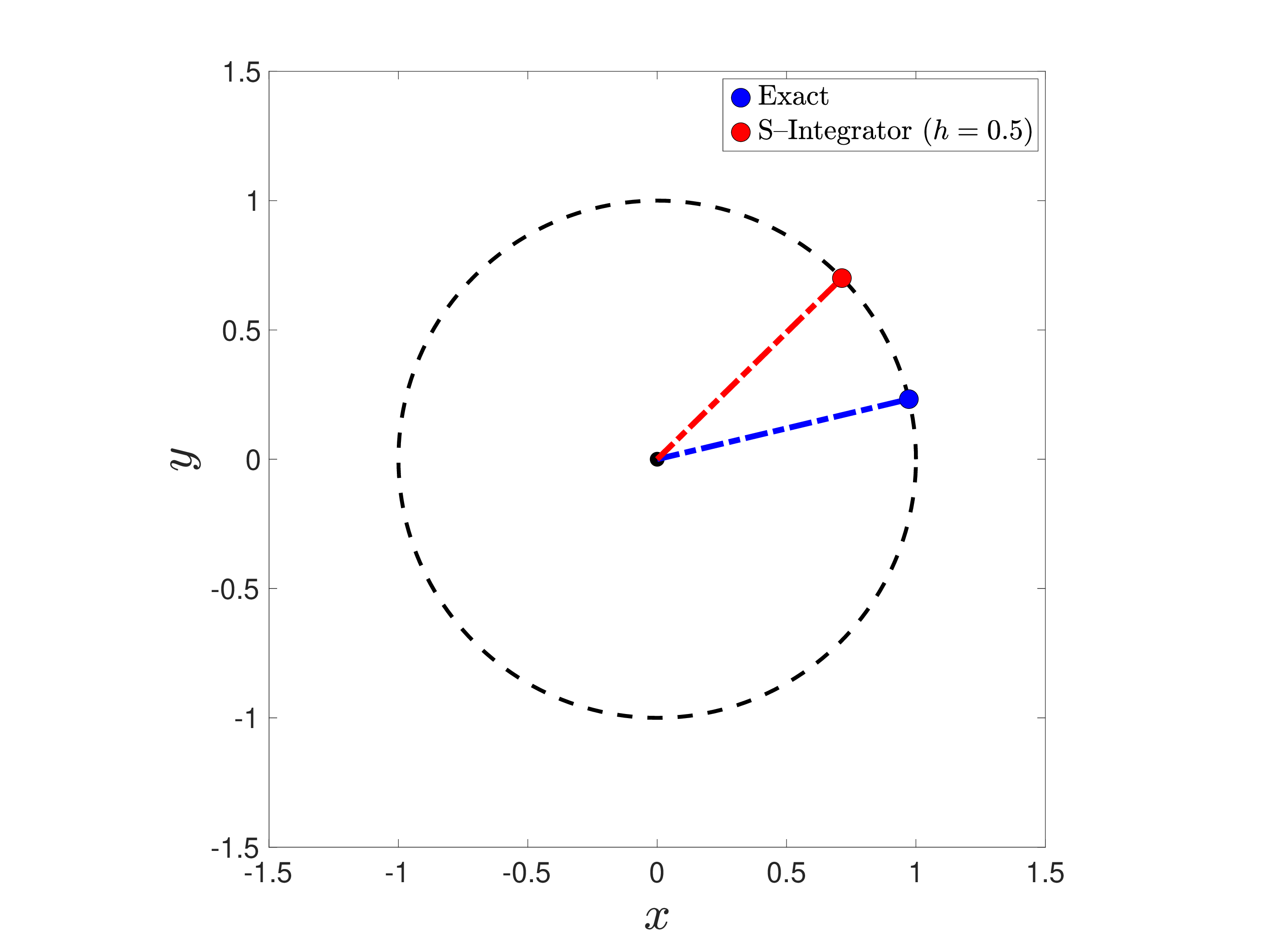}
\caption{Phase mismatch $n\Delta \theta_h$ at $n = 101$}
\end{subfigure}
\begin{subfigure}[t]{0.46\linewidth}
\centering
\begin{tikzpicture}[scale=1.8, >=stealth, thick]
    \coordinate (O) at (0,0);
    \def\R{3.5} 
    \def\angReal{15} 
    \def\angNum{30}  

    \draw[->] (-0.5,0) -- (\R+0.5, 0) node[right] {}; 
    \draw[->] (0,-0.5) -- (0, \R-1.0) node[above] {}; 
    \node at (O) [below left] {$0$};
    \node at (3.6,0) [below right] {$x$};
    \node at (0,2.0) [above left] {$y$};

    \draw[->, blue, line width=1.2pt] (O) -- (\angReal:\R+0.6) node[left, black] {Exact $\mathrel{\phantom{aaa}}$};
    \draw[->, red, line width=1.2pt] (O) -- (\angNum:\R+1.1) node[left, black] {S-Integrator $\mathrel{\phantom{a}}$};

    \draw[->, thin] (0.8,0) arc (0:\angReal:0.8);
    \node at (\angReal/2:1.0) {$h$};

    \draw[->, thin] (2.5,0) arc (0:\angNum:2.5);
    \node at (\angNum-40/2:2.7) {$\theta_h$};

    \draw[<->, thin] (\angReal:1.3) arc (\angReal:\angNum:1.3);
    \node at ({(\angReal+\angNum)/2}:1.6) {$\Delta\theta_h$};


\end{tikzpicture}
\caption{Schematic illustration: $\Delta \theta_h$}
\end{subfigure}
\caption{Accumulated phase (angle) error $n\Delta\theta_h$ at time $t_n = nh$ between the exact solution of the Hamiltonian system~\eqref{eqn: ham-os} and the first-order symplectic integrator~\eqref{eqn: ham-os-sym}.}
\label{fig: phase-difference}
\end{figure}
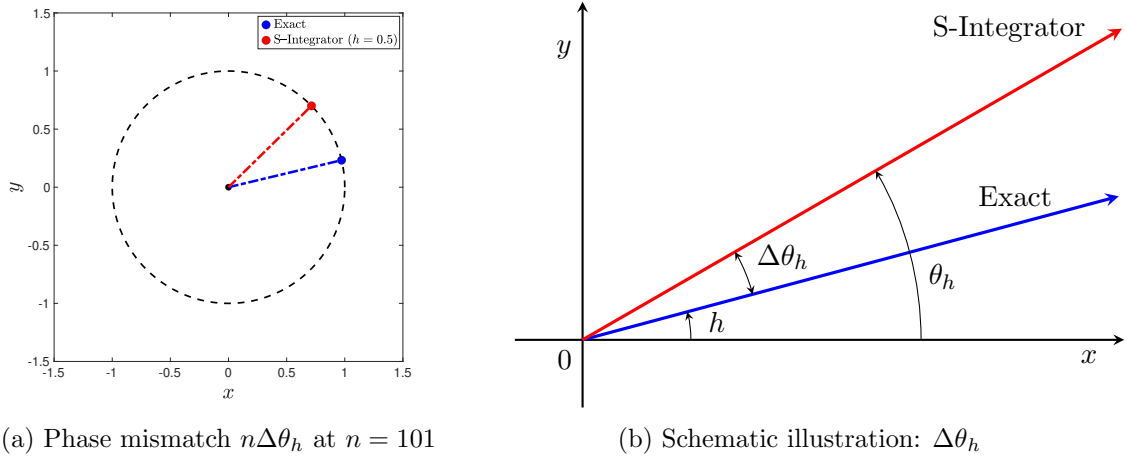

\Cref{fig: symp-orbit} and~\Cref{fig: phase-difference} demonstrate that, although symplectic integrators preserve the global geometric and topological structure of phase space, they do not, in general, guarantee uniform accuracy of the angle variables. Over long times, small local phase errors accumulate linearly in time, leading to an inevitable drift in the phase.  In the long-time regime, while the numerical and exact trajectories remain close in shape, the precise position along the orbit becomes increasingly unpredictable, reflecting an intrinsic loss of phase fidelity.

\section{A dimension-enlarged Newton scheme}
\label{sec: newton-scheme}

The Craig-Wayne-Bourgain (CWB) scheme, originally developed by~\citet{craig1993newton, bourgain1994construction, bourgain1998quasi}, is a fundamentally analytical framework for establishing the existence of quasi-periodic solutions. While extraordinarily powerful from a theoretical perspective, it has so far remained largely confined to existence proofs, and its potential as a practical numerical method has not been fully realized. At the core of the CWB approach lies a dimension-enlarged Newton (or Nash-Moser) iteration, which naturally suggests a viable computational strategy. In this work, we bridge this gap by transforming the analytic CWB framework into a concrete and implementable numerical algorithm. The proposed dimension-enlarged Newton scheme computes quasi-periodic solutions while preserving the underlying geometric structure of the system. Crucially, it avoids the secular phase-error drift inherent in standard symplectic integrators. In addition, the method is equipped with a rigorous a priori error bound, providing a quantitative certificate of accuracy and ensuring the reliability of the computed solutions.

\subsection{Lattice representation}
\label{subsec: lattice}

To begin our analysis, we introduce multi-index notation. for any $\pmb{\alpha} = (\alpha_1, \ldots, \alpha_n) \in \mathbb{N}^n$ and $\pmb{z} = (z_1, \ldots, z_n) \in \mathbb{C}^n$,  we define the length $|\pmb{\alpha}| := \sum_{i=1}^n |\alpha_i|$ and the associated monomial $\pmb{z}^{\pmb{\alpha}} := \prod_{j=1}^{n}z_j^{\alpha_j} $. Consider the nearly integrable Hamiltonian system~\eqref{eqn: nearly-integrable-func} in action-angle variables $(\pmb{I}, \pmb{\theta})$. We apply the complex canonical transformation $z_j = \sqrt{I_j} e^{i\theta_j} $ (and its conjugate $\overline{z}_j = \sqrt{I_j} e^{-i\theta_j}$), under which  the symplectic form becomes $d\pmb{z} \wedge d \overline{\pmb{z}} = -  i d \pmb{I} \wedge d\pmb{\theta}$. In these coordinates, the Hamiltonian~\eqref{eqn: nearly-integrable-func} is expressed as:
\begin{equation}
    \label{eqn: nearly-integrable-ham-complex}
    H(\pmb{z}, \overline{\pmb{z}}) = H_0\left( |z_1|^2, \ldots, |z_n|^2 \right) + \varepsilon H_1(\pmb{z}, \overline{\pmb{z}}),
\end{equation}
where the unperturbed part $H_0$ depends only on the actions $I_j = |z_j|^2 $, and the perturbation $H_1$ is a real-valued polynomial given by
\begin{equation}
    \label{eqn: H1}
    H_1(\pmb{z}, \overline{\pmb{z}}) = \sum_{\pmb{\alpha},  \pmb{\beta}} h_{\pmb{\alpha} \pmb{\beta}}  \left( \pmb{z}^{\pmb{\alpha}} \overline{\pmb{z}}^{\pmb{\beta}}  + \pmb{z}^{\pmb{\beta}} \pmb{\overline{z}}^{\pmb{\alpha}}  \right), 
\end{equation}
where the coefficients $h_{\pmb{\alpha} \pmb{\beta}} $ are real, and the total degree satisfies $ \max\{ |\pmb{\alpha}| + |\pmb{\beta}|\} \geq 3$. The resulting Hamiltonian equations of motion are
\begin{equation}
\label{eqn: nearly-complex}
\dot{\pmb{z}} =  i \pmb{\omega} \odot \pmb{z} +   i\varepsilon \frac{\partial H_1}{\partial \overline{\pmb{z}}},
\end{equation}
where $\odot $ denotes componentwise multiplication.
Before proceeding, we recall the definition of a time-quasi-periodic function to set the stage for excluding resonant interactions.
\begin{defn}[Quasi-Periodic Function]
\label{defn: quasi-periodic}
The frequency $\pmb{\omega} \in \mathbb{R}^n$ is said to be \textit{rationally independent} if it satisfies the non-resonance condition $\langle \pmb{k}, \pmb{\omega} \rangle \neq 0$ for any  $\pmb{k} \in \mathbb{Z}^n \setminus \{0\}$.  A function $f = f(t)$ is called \emph{time quasi-periodic} if there exists a periodic function $g = g(\pmb{\theta})$ on the $n$-dimensional torus $\mathbb{T}^{n}$ such that $f(t) = g(\pmb{\omega} t)$. 
\end{defn}

For a vector-valued quasi-periodic function $\pmb{z}(\pmb{\omega}t) = \left( z_1(\pmb{\omega}t), \ldots, z_n(\pmb{\omega}t) \right)^{\top}$, we consider its Fourier decomposition: 
\begin{equation}
    \label{eqn: unpert-sol-1}
    \pmb{z}(\pmb{\omega}t) = \sum_{\pmb{k} \in \mathbb{Z}^n} \widehat{\pmb{z}}(\pmb{k}) e^{i \langle \pmb{k}, \;\pmb{\omega} \rangle t}, 
\end{equation}
where the coefficients $\{\widehat{\pmb{z}}(\pmb{k})\}_{\pmb{k} \in \mathbb{Z}^n }$ are indexed by the integer lattice $\mathbb{Z}^n$.  In the unperturbated case ($\varepsilon = 0$), the  system~\eqref{eqn: nearly-complex} is linear and completely integrable, admitting the explicit solution
\begin{equation}
    \label{eqn: unpert-sol}
    z_{j}(\pmb{\omega}t)= a_j e^{i \omega_j t}, \quad j = 1, \dots, n,
\end{equation}
Given that the symplectic structure is invariant under the scaling $z_j \mapsto ea_jz_j$, we may assume without loss of generality that the amplitudes $a_j = e^{-1}$ for $j=1, \ldots, n$. The linear solution~\eqref{eqn: unpert-sol} corresponds to the Fourier representation~\eqref{eqn: unpert-sol-1} with coefficients localized on the standard basis vectors $ \pmb{e}_j \in \mathbb{Z}^n$ for $j=1,\ldots,n$: 
\begin{equation}
\label{eqn: init}
\widehat{z}_j(\pmb{k}) = \left\{ \begin{aligned}
                                                & a_j,        && \pmb{k} = \pmb{e}_j, \\
                                                & 0,           && \pmb{k} \neq \pmb{e}_ j.
                                                \end{aligned} \right.  
\end{equation}                                                
The solution to the nearly integrable Hamilton system~\eqref{eqn: nearly-complex} is treated as a perturbation of the linear solution~\eqref{eqn: init}, which admits the Fourier representation:
\begin{equation}
\label{eqn: pert-sol}  
    \pmb{z}(\pmb{\omega}'t) = \sum_{\pmb{k} \in \mathbb{Z}^n} \widehat{\pmb{z}}(\pmb{k}) e^{i \langle \pmb{k}, \;\pmb{\omega}' \rangle t}.
\end{equation}
where $\pmb{\omega}'$ represents a drifted frequency.   This perturbed solution is characterized by the following properties:   
\begin{itemize}    
\item \textbf{Frequency drift}: The drifted frequency  $\pmb{\omega}'$ remains close to the original frequency $ \pmb{\omega} $.

\item \textbf{Fixed dominant modes}: The  dominant Fourier coefficients remain fixed, with $\widehat{z}_j(\pmb{e}_j) = a_j$ for $j =1 , \ldots, n$.

\item \textbf{Perturbation-induced coefficients}: All remaining coefficients $\widehat{z}_j(\pmb{k}) \neq 0$ for $\pmb{k} \neq \pmb{e}_j $ are generated by the perturbation. \end{itemize}    
These resulting ``excited'' Fourier modes are expected to exhibit suitable decay properties, which will be established in the subsequent analysis. Our goal is to compute, numerically, a time quasi-periodic solution~\eqref{eqn: pert-sol} to the perturbed system, starting  from the linear solution~\eqref{eqn: unpert-sol}. The overall numerical procedure, mapping the initial Fourier coefficients and frequency to their perturbed counterparts, is illustrated schematically in~\Cref{fig: chart}. 
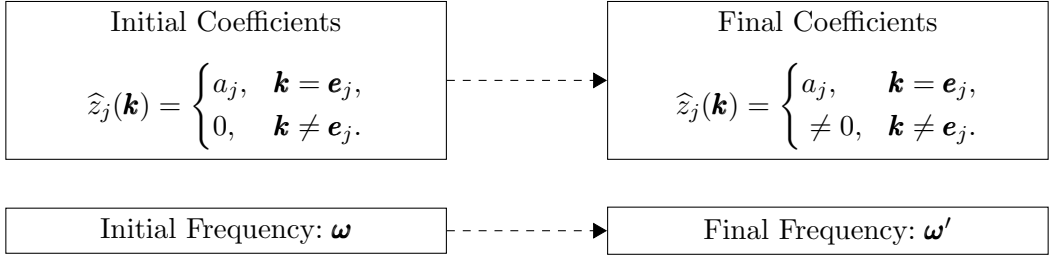
\begin{figure}[htb!]
\begin{center}
\scalebox{1.0}{
\begin{tikzpicture}[node distance=4cm]
    \node (n00) [box1, draw=black] {Initial Coefficients \[ \widehat{z}_j(\pmb{k}) = \left\{ \begin{aligned}
                                                & a_j,        && \pmb{k} = \pmb{e}_j, \\
                                                &  0,           && \pmb{k} \neq \pmb{e}_ j.
                                                \end{aligned} \right.  \]};
    \node (n01) [box1, draw=black,  below of=n00, yshift=+2cm] {Initial Frequency:~$\pmb{\omega}$};
    \node (n10) [box1, draw=black,  right of=n00,   xshift=+4cm] {Final Coefficients\[ \widehat{z}_j(\pmb{k}) = \left\{ \begin{aligned}
                                                & a_j,        && \pmb{k} = \pmb{e}_j, \\
                                                & \neq 0,           && \pmb{k} \neq \pmb{e}_ j.
                                                \end{aligned} \right.  \]};
    \node (n11) [box1, draw=black, below of=n10, yshift=+2cm] {Final Frequency:~$\pmb{\omega}'$};

    \draw [arrow] [dashed] (n00) --  (n10);    
    \draw [arrow] [dashed] (n01) --  (n11);    
\end{tikzpicture}
}
\end{center}
\caption{Schematic diagram of the numerical procedure.} 
\label{fig: chart}
\end{figure}

\subsection{Lyapunov-Schmidt decomposition}
\label{subsec: lyapunov-schmidt}

To facilitate the analysis, we denote the perturbation vector field by $\pmb{X} = \partial H_1 / \partial \pmb{\overline{z}}$, with its Fourier representation written as $\widehat{\pmb{X}} =\{ \widehat{\pmb{X}}(\pmb{k})\}_{\pmb{k} \in \mathbb{Z}^n}$. Substituting the formal time quasi-periodic solution~\eqref{eqn: pert-sol} into the complex Hamilton system~\eqref{eqn: nearly-complex}, we derive the following algebraic system over the lattice $\mathbb{Z}^n$:
\begin{equation}
\label{eqn: lattice}
-\langle \pmb{k}, \pmb{\omega}' \rangle \widehat{\pmb{z}} + \pmb{\omega} \odot \widehat{\pmb{z}} + \varepsilon \widehat{\pmb{X}}\big(\widehat{\pmb{z}}, \widehat{\overline{\pmb{z}}}  \big) = 0.
\end{equation}
where the Fourier coefficients satisfy the symmetry condition $ \widehat{\overline{\pmb{z}}} (\pmb{k}) = \overline{\widehat{\pmb{z}}(-\pmb{k})}$.  A key feature of the Fourier representation in~\eqref{eqn: lattice} is that the initial state~\eqref{eqn: init} is real.  Given that the perturbation Hamiltonian $H_1$ is a multivariate polynomial with real coefficients, both the full vector field  and its associated tangent linear operator are inherently real. The Newton scheme described in~\Cref{subsec: newton} preserves this reality,  ensuring that all subsequent iterates remain real. Consequently, we may assume without loss of generality that the Fourier vector $\widehat{\pmb{z}}$ is real.  Under this realness assumption, the symmetry condition simplifies to $ \widehat{\pmb{z}} (\pmb{k}) = \widehat{\pmb{z}}(-\pmb{k})$, which implies that the perturbation vector field $\widehat{\pmb{X}}$ depends only on $\widehat{\pmb{z}}$, rather than depending on $\widehat{\pmb{z}}$ and $\widehat{\overline{\pmb{z}}}$ independently. Therefore, the algebraic system~\eqref{eqn: lattice} can be rewritten as: 
\begin{equation}
\label{eqn: lattice-real}
-\langle \pmb{k}, \pmb{\omega}' \rangle \widehat{\pmb{z}} + \pmb{\omega} \odot \widehat{\pmb{z}} + \varepsilon \widehat{\pmb{X}}(\widehat{\pmb{z}}  ) = 0.
\end{equation}
Notice that when the perturbed frequency vector $\pmb{\omega}'$ coincides with the unperturbed vector $\pmb{\omega}$, then the coefficient of $\widehat{z}_{j}(\pmb{k})$ in~\eqref{eqn: lattice-real} vanishes for the mode $\pmb{k} = \pmb{e}_j$. This  degeneracy motivates the definition of the resonant set:
\begin{equation}
\label{eqn: resonant-set}
\mathcal{S}= \left\{ \left( j, \pmb{e}_j \right) | j = 1, \dots, n \right\}.
\end{equation} 
To handle the resulting singularities, we employ the Lyapunov-Schmidt decomposition, partitioning the algebraic system~\eqref{eqn: lattice-real} into two subsystems  according to whether the index pair $(j, \pmb{k})$ belongs to the resonant set $\mathcal{S}$ or not. 
\begin{tcolorbox}[breakable]
\begin{itemize}
\item \textbf{$Q$-equations} (\textbf{Resonant Modes}):   For the resonant indices $ (j, \pmb{k}) \in \mathcal{S}$, specifically for $j = 1, \dots, n$ and $\pmb{k} = \pmb{e}_j$,  the algebraic system~\eqref{eqn: lattice-real} reduces to
\begin{equation}
    \label{eqn: q-eqn}
    \left( -\omega'_j + \omega_j \right) a_j +  \varepsilon \widehat{X}_j(\pmb{e}_j) = 0, 
\end{equation}
which are referred to as the~\textit{$Q$-equations}. For convenience, we collect the resonant Fourier coefficients into the \textit{resonant vector} $\widehat{z}_q := \pmb{a} = (a_1, \ldots, a_n)^{\top}$, and define the \textit{resonant vector field} $\widehat{\pmb{X}}_q: = \big(\widehat{X}_1(\pmb{e}_1), \ldots, \widehat{X}_n(\pmb{e}_n) \big)^{\top}$. Consequently, the $Q$-equations~\eqref{eqn: q-eqn} can be written in the compact vector form as
\begin{equation}
    \label{eqn: q-eqn-vector}
\left( - \pmb{\omega}'+ \pmb{\omega} \right) \odot \widehat{\pmb{z}}_q + \widehat{\pmb{X}}_q  = 0.
\end{equation}

\item \textbf{$P$-equations} (\textbf{Non-Resonant Modes}):  For non-resonant indices $ (j, \pmb{k}) \notin \mathcal{S}$, the corresponding  components are collected into \textit{non-resonant vector} $\widehat{\pmb{z}}_p $ and  \textit{non-resonant vector field} $\widehat{\pmb{X}}_p$.  Accordingly, the full Fourier vector and vector field admit the following decomposition:  
\begin{equation}
\label{eqn: vector-field-decompostion}
 \widehat{\pmb{z}} = \begin{pmatrix} \widehat{\pmb{z}}_q \\ \widehat{\pmb{z}}_p \end{pmatrix} \quad  \mathrm{and} \quad \widehat{\pmb{X}} = \begin{pmatrix} \widehat{\pmb{X}}_q \\ \widehat{\pmb{X}}_p \end{pmatrix}.
 \end{equation}
By substituting the resonant components $\widehat{\pmb{z}}_p$ and $ \widehat{\pmb{X}}_p $ into the algebraic system~\eqref{eqn: lattice-real}, we derive the~\textit{$P$-equations} as:
\begin{equation}
\label{eqn: p-eqn-vector}
-\langle \pmb{k}, \pmb{\omega}' \rangle \widehat{\pmb{z}}_p + \pmb{\omega} \odot \widehat{\pmb{z}}_p + \varepsilon \widehat{\pmb{X}}_p= 0.
\end{equation}
\end{itemize}
\end{tcolorbox}

To implement the Newton scheme for numerically computing the solution, we derive the tangent linear operator associated with the $P$-equations~\eqref{eqn: p-eqn-vector}:
\begin{equation}
    \label{eqn: lin-op}
    T = D + \varepsilon S,
\end{equation}
where $D$ denotes the diagonal operator with entries as $D_{j, \pmb{k}} = -\langle \pmb{k}, \pmb{\omega}' \rangle + \omega_j$ and $S = \partial \widehat{\pmb{X}}_p/ \partial \widehat{\pmb{z}}_p$ arises from the nonlinear perturbation. The properties of $S$ is discussed in~\Cref{subsec: decay-prop}.

\subsection{Numerical scheme}
\label{subsec: newton}

Following the derivation of the $P$-equations~\eqref{eqn: p-eqn-vector}, the central computational challenge lies in the treatment of the infinite-dimensional vector $\widehat{\pmb{z}}_p$. Traditional Newton methods typically require truncating the system to a fixed finite dimension. Once the problem is restricted to a finite lattice, however, truncation errors are inevitably introduced. In strongly nonlinear regimes, such errors do not vanish; instead, they may persist and accumulate, potentially growing as the iteration proceeds. To maintain computational tractability while ensuring that the truncation error converges rigorously to zero, we employ a dimension-enlarged Newton scheme. This strategy draws inspiration from the Nash-Moser iteration, originally developed to overcome ``small-divisor'' difficulties  in nonlinear analysis and to establish convergence within infinite-dimensional function spaces~\citep{nash1956imbedding, moser1966rapidlyI, moser1966rapidlyII}. 

To formalize the numerical procedure, we first define the truncated Fourier lattice box $\Lambda_N $ for any positive integer $N \in \mathbb{N}$:
\begin{equation}
\label{eqn: box-fourier-truncated}
\Lambda_N := \left\{ \pmb{k} \in \mathbb{Z}^n: \; \max_{1 \leq i \leq n}|\pmb{k}_i| \leq N \right\}.
\end{equation}
We then introduce the projection operator $P_N$, which restricts functions on the full lattice $\mathbb{Z}^{n}$ to the truncated lattice $\Lambda_N$ box:
\begin{equation} 
\label{eqn: project-operator}
    \left( P_{N} \widehat{\pmb{z}} \right) \left(\pmb{k} \right) := \left\{
    \begin{aligned}
        & \widehat{\pmb{z}}(\pmb{k}),   \quad       &&\mathrm{if}~\pmb{k} \in \Lambda_N,  \\
        & 0,                                            \quad      &&\mathrm{if}~\pmb{k} \notin \Lambda_N.
    \end{aligned}
    \right.
\end{equation}
For the tangent linear operator $T$, we define its restriction to the truncated lattice $\Lambda_N$ as:
\begin{equation}
    \label{eqn: restriction-block}
    T_{N} :=  T\big |_{\pmb{k} \in \Lambda_N}
\end{equation}
 Equivalently, consistent with the projection operator defined in~\eqref{eqn: project-operator}, this restriction can be written as $ T_{N} = P_{N} T P_{N}$.

With these preparation, we describe our alternatiting numerical scheme for updating the frequency and the non-resonant vector.  Starting from the initial frequency  $\pmb{\omega}^{(0)} = \pmb{\omega}$ and the initial non-resonant vector $\widehat{\pmb{z}}_p^{(0)} = \pmb{0}$, we proceed iteratively. At the $r$-th step, the current approximations $\pmb{\omega}^{(r)}$ and $\widehat{\pmb{z}}_p^{(r)}$ serve as inputs for a two-stage process. First, we update the frequency by solving the $Q$-equations~\eqref{eqn: q-eqn-vector}:
\begin{equation}
\label{eqn: q-eqn-numerical}
\left( - \pmb{\omega}^{(r+1)}+ \pmb{\omega} \right) \odot \pmb{a} + \varepsilon \widehat{\pmb{X}}_q\big(\widehat{\pmb{z}}_p^{(r)}; \pmb{a}\big)  = 0
\end{equation}
Next, we update  the non-resonant vector via the $P$-equations~\eqref{eqn: p-eqn-vector}. To this end, we define the full non-resonant vector field: 
\begin{equation}
\label{eqn: non-resonant-vector}
F\big(\pmb{\omega}, \widehat{\pmb{z}}_p; \pmb{a} \big) = -\langle \pmb{k}, \pmb{\omega}' \rangle \widehat{\pmb{z}}_p + \pmb{\omega} \odot \widehat{\pmb{z}}_p + \varepsilon \widehat{\pmb{X}}_p\big(\widehat{\pmb{z}}_p; \pmb{a}\big).
\end{equation}
Departing from classical existence theories for quasi-periodic solutions (see~\citet{bourgain1994construction, bourgain1998quasi}), numerically efficiency necessitates an additional rank-one operator. This operator is derived from the $Q$-equations~\eqref{eqn: q-eqn-vector} to stablize the frequency iteration: 
\begin{equation}
\label{eqn: rank-1}
B = - \frac{1}{e}\left( \frac{\partial \langle \pmb{k}, \widehat{\pmb{X}}_q\rangle}{\partial \widehat{\pmb{z}}_p} \right)  \widehat{\pmb{z}}_p^{\top}
\end{equation}
Let $M \in \mathbb{N}$ be a fixed positive integer. We set the truncation dimension at the $r$-th step as $N_{r+1} := M^{r+1}$, so that the dimension of the truncated space increases by a factor of $M$ at each iteration. The updated non-resonant vector is then obtained via a dimension-enlarged Newton step: 
\begin{equation}
\label{eqn: p-eqn-numerical}
\widehat{\pmb{z}}_{p}^{(r+1)} = \widehat{\pmb{z}}_{p}^{(r)} - \left[ (T + \varepsilon B)_{N_{r+1}}^{-1}\big(  \widehat{\pmb{z}}_{p}^{(r)}, \pmb{\omega}^{(r+1)} ; \pmb{a} \big) \right] F\big(  \widehat{\pmb{z}}_{p}^{(r)}, \pmb{\omega}^{(r+1)} ; \pmb{a} \big), 
\end{equation}
where the non-resonant vector $  \widehat{\pmb{z}}^{(r)}_p$ is supported on the lattice box $\Lambda_{N_r}$, while the updated vector $ \widehat{\pmb{z}}^{(r+1)}_p$ resides in the enlarged lattice box $\Lambda_{N_{r+1}}$. The alternating numerical procedure is illustrated in~\Cref{fig: algorithm-nm}.
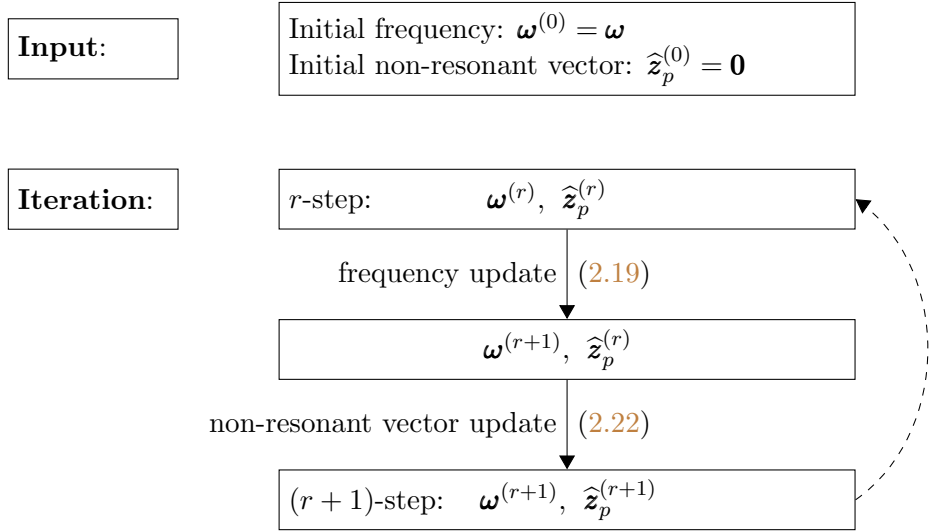
\begin{figure}[htb!]
\centering
\begin{tikzpicture}
\node (n00) [mybox, draw=black] {\textbf{Input}:};
\node (n01) [box, draw=black, right of = n00, xshift =5.3cm] {Initial frequency: $\pmb{\omega}^{(0)} = \pmb{\omega}$\\
                                                 Initial non-resonant vector: $ \widehat{\pmb{z}}^{(0)}_p =  \pmb{0}$ };
                                                 
\node (n10) [mybox, draw=black, below of = n00, yshift = -1cm] {\textbf{Iteration}:};       
\node (n11) [box2, draw=black, right of = n10, xshift =5.3cm] {$r$-step: $\mathrel{\phantom{asdef}}$ \; $\pmb{\omega}^{(r)}$,\;  $  \widehat{\pmb{z}}^{(r)}_p$};       
\node (n100) [box2, draw=black, below of = n11, yshift = -1cm] {$\mathrel{\phantom{astepaabcdef}}$ \; $\pmb{\omega}^{(r+1)}$,\;  $ \widehat{\pmb{z}}^{(r)}_{p} $};        
\node (n101) [box2, draw=black, below of = n100, yshift = -1cm] {$(r+1)$-step: \; $\pmb{\omega}^{(r+1)}$,\;  $ \widehat{\pmb{z}}^{(r+1)}_{p} $};     

  \draw [arrow]  (n11) -- node[left]{frequency update} node[right]{\eqref{eqn: q-eqn-numerical}}(n100);
  \draw [arrow]  (n100) -- node[left]{non-resonant vector update} node[right]{\eqref{eqn: p-eqn-numerical}} (n101);                              
  \draw[arrow, dashed, bend right=55] (n101.east) to (n11.east);
  
\end{tikzpicture}
\caption{A diagram illustrating the alternating numerical procedure.}
\label{fig: algorithm-nm}
\end{figure}

This progressive enlargement is a hallmark of the Nash-Moser strategy: higher-frequency modes are gradually activated as the truncation dimension increases. This allows the scheme to approximate the infinite-dimensional solution while maintaining control over the inverse of the linearized operator. Although the norm of the inverse may grow as the dimension increases, the super-exponential convergence of the Newton update compensates for this growth, thereby ensuring that the scheme converges toward the true solution.

\subsection{Statement of the main theorem}
\label{subsec: main-results}

To establish the convergence theorem, we first define the norms necessary to quantify the infinite-dimensional vectors.  For any Fourier vector $\widehat{\pmb{z}} = \{ \widehat{\pmb{z}}(\pmb{k})  \}_{\pmb{k} \in \mathbb{Z}^n}$, the $\ell_2$-norm is given by
\[
\| \widehat{\pmb{z}}\|_2 = \left( \sum_{\pmb{k} \in \mathbb{Z}^n} \|\widehat{\pmb{z}}(\pmb{k}) \|_2^2 \right)^{\frac12}
\]
Given that the continuous differentiability of the solution with respect to parameter, we know that the Fourier vector $\widehat{\pmb{z}}$ is a function of the frequency. Accordingly, we define the normed space $\mathscr{H}(\mathbb{Z}^n)$ as follows: 
\begin{equation}
\label{eqn: l2-space}
\mathscr{H}(\mathbb{Z}^n) = \left\{ \widehat{\pmb{z}} = \left\{ \widehat{\pmb{z}}(\pmb{k})\right\}_{\pmb{k} \in \mathbb{Z}^n} \bigg | \widehat{\pmb{z}}(\pmb{k}) \in \mathbb{R}^n, \;  \| \widehat{\pmb{z}} \| = \| \widehat{\pmb{z}}\|_2 + \| \partial_{\pmb{\omega}}\widehat{\pmb{z}}\|_2  < \infty  \right\}. 
\end{equation}
where the composite norm is given by $\| \cdot \| = \| \cdot \|_2 + \| \partial_{\omega} (\cdot)\|_2$. Unless otherwise specified, the composite norm $\| \cdot \|$ is used throughout this paper.  In addition, $|\cdot|$ and $|\cdot|_{\infty}$ denote $\ell^1$-norm and $\ell^{\infty}$-norm, respectively, in $\mathbb{Z}^n$ or $\mathbb{R}^n$. To characterize the regularity of the solutions, we define Gevrey decay set within the $\ell_2$ space as:
\begin{equation}
\label{eqn: gevrey-l2}
\mathcal{K}(s) = \left\{ \widehat{\pmb{z}} \in \mathscr{H}(\mathbb{Z}^n) \bigg | \sup_{\pmb{k} \in \mathbb{Z}^n} \big( \| \widehat{\pmb{z}}(\pmb{k}) \| \exp\left\{|\pmb{k}|^s \right\} \big) \leq 1 \right\},
\end{equation}
where the sub-exponential decay of the Fourier coefficients is essential for maintaining the regularity of the iterative scheme.\footnote{Technically, this bound may be any positive constant; for convenience, we set it to $1$ without loss of generality.} We now present the main theoretical statement of this study: the a priori error estimate for the alternating numerical procedure illustrated in~\Cref{fig: algorithm-nm}.

\begin{theorem}[Convergence and Error Estimation]
\label{thm: main}
 Let  $ \Omega \subseteq \mathbb{R}^n $ be a bounded domain and $\tau > n-1$ be a fixed parameter. There exists a critical threshold $\varepsilon_0 = \varepsilon_0(H_1, \Omega; \pmb{a}) > 0$ such that for any $0 < \varepsilon \leq \varepsilon_0$, the following properties hold:
 \begin{itemize}
 \item[1.] \textbf{Measure of resonance}: There exists two small Lebesgue measure $\kappa =  \kappa(n, \tau, \Omega) >0$ and $\delta = \delta(\varepsilon) > 0$, where  $\delta \rightarrow 0$ as $\varepsilon \rightarrow 0$, such that the excluded set of ``bad'' frequencies $\Omega^{\star}$ satisfies $\mathrm{mes}(\Omega^{\star}) \leq \kappa + \delta $.
 
 \item[2.] \textbf{Iterative sequence}: For any initial frequency $\pmb{\omega}^{(0)} \in \Omega \setminus \Omega^{\star}$ and initial lattice vector $\widehat{\pmb{z}}^{(0)}$ defined in~\eqref{eqn: init},  the alternating numerical procedure, comprising the frequency update~\eqref{eqn: q-eqn-numerical} and the non-resonant vector update~\eqref{eqn: p-eqn-numerical}, generates a sequence of iterates $\{ (\pmb{\omega}^{(r)}, \widehat{\pmb{z}}^{(r)}) \}_{r=0}^{\infty}$. 
 
 \item[3.] \textbf{Convergence}: There exists a Gevrey exponent $s = s(\varepsilon)$ such that this sequence remains within the product space $(\Omega \setminus \Omega^\star) \times \mathcal{K}(s)$ and converges to the exact solution pair $(\boldsymbol{\omega}^\star, \mathbf{z}^\star)$.

 \item[4.] \textbf{A priori estimates}: The convergence is characterized by the following error bounds:
 \begin{subequations}
\label{eqn: convergence-vector-frequency}
\begin{empheq}[left=\empheqlbrace]{align} 
         &  \| \widehat{\pmb{z}}^{(r)} - \widehat{\pmb{z}}^{\star} \| \leq \exp \left\{ -\frac{3}{2} (M^s)^r \right\},      \label{eqn: z-hat-converge}                 \\
         & |\pmb{\omega}^{(r)} - \pmb{\omega}^{\star}| \leq  \exp \left\{ -\frac{3}{2} (M^s)^r \right\}.                          \label{eqn: frequency-hat-converge}
\end{empheq}  
\end{subequations}
\end{itemize}
Furthermore, for any $t \in [0, N_r]$, the approximate solution $\pmb{z}^{(r)}(t)$ satisfies the time-domain error bound:
\begin{equation}
\label{eqn: convergence-soln}
    \| \pmb{z}^{(r)}(t) - \pmb{z}^{\star}(t) \| \leq \exp \left\{ -(M^s)^r \right\}.
\end{equation}
This implies that as $r \rightarrow \infty$, the numerical solution $\pmb{z}^{(r)}(t)$ converges to the exact quasi-periodic solution $\pmb{z}^{\star}(t) $ for any $t \geq 0$. 
\end{theorem}

While the main theorem is established in $\ell_2$, the result extends to  more compact Gevrey spaces. For any $s' \in [0, s)$, we define the norm 
\[
\|\widehat{\pmb{z}}\|_{s'} = \left( \sum_{\pmb{k} \in \mathbb{Z}^n }\| \widehat{\pmb{z}}(\pmb{k}) \|^2  \exp\left\{ |\pmb{k}|^{2s'} \right\} \right)^{\frac12}
\]
The proof remains essentially unchanged in this setting. Moreover, this regularity can be generalized to the analytic setting.  Let $\vartheta = \vartheta(\varepsilon)$ be an analytic strip that depends on the perturbation parameter. For any $\vartheta' \in [0, \vartheta')$, we define the analytic norm and the analytic decay condition as:
\[
\|\widehat{\pmb{z}}\|_{\vartheta'} = \left( \sum_{\pmb{k} \in \mathbb{Z}^n }\| \widehat{\pmb{z}}(\pmb{k}) \|^2  \exp\left\{ 2 \vartheta' |\pmb{k}| \right\} \right)^{\frac12} \quad  \mathrm{and}\quad \sup_{\pmb{k} \in \mathbb{Z}^n} \big( \| \widehat{\pmb{z}}(\pmb{k}) \| \exp\left\{ \vartheta |\pmb{k}| \right\} \big) \leq 1.
\]
\Cref{thm: main} remains valid in this setting as well.  The proof of this theorem is established in~\Cref{sec: completion}, supported by the technical foundations laid in~\Cref{sec: small} and~\Cref{sec: multiscale}; a comprehensive outline of the proof strategy is provided~\Cref{sec: priori-error}. 

\begin{remark}
\label{rem: main}

The convergence bounds~\eqref{eqn: z-hat-converge} and~\eqref{eqn: frequency-hat-converge} represent a significant conceptual departure from classical KAM theory~\citep{arnold1963small, moser1962invariant}. While classical theory treats action variables as regular perturbations to preserve invariant tori, the present framework explicitly accounts for frequency drift in the existence of quasi-periodic solutions. Unlike standard KAM results, these estimates here incorporate a refined decay property of the Fourier coefficients, interpreting the solution's regularity itself as a singular perturbation~\citep{Witelski:2009}.  The iteration lemma (\Cref{lem: iteration}) hinges on the super-exponential convergence of the vector field; consequently, the exponent $s$ is not arbitrary. As the perturbation parameter vanish ($\varepsilon \rightarrow 0$),  the exponent must vanish ($s \rightarrow 0$) at a rate faster than any polylogarithmic function of $1/\varepsilon$. In this regime, the admissible regularity becomes ``almost zero'',  restricting the solution's regularity nearly to $\ell_2$. This leads to a significant qualitative insight: even if the initial state $\widehat{\pmb{z}}^{(0)}$ possesses high regularity, the limiting state $\widehat{\pmb{z}}^{\star}$ is necessarily confined to a low-regularity class.~\Cref{thm: main} remains valid in the analytic setting via the same inductive estimation strategy, provided that the strip width shrinks to zero ($\vartheta \rightarrow 0$) as $\varepsilon \rightarrow 0$. Finally, the convergence threshold $\varepsilon_0 = \varepsilon_0(H_1, \Omega, \pmb{a})$ is fundamentally determined by the ``input data'' of the system,  specifically the perturbation structure $H_1$, the frequency set $\Omega$, and the initial amplitude vector $\pmb{a}$. For convenience, we assume $\varepsilon_0 \leq 1/2$ without loss of generality throughout the remainder of the paper. 
\end{remark}

\section{Basic properties and conditions}
\label{sec: priori-error}

The analysis in the following three sections is confined to  the Gevrey decay set $\mathcal{K}(s)$ as defined in~\eqref{eqn: gevrey-l2}. In this initial section, we establish the basic properties necessary for the subsequent analysis.  We then state the conditions required for the numerical implementation of the inverse of the tangent linear operator $T + \varepsilon B$, within the restricted lattice box $\Lambda_N$. Finally, we outline the overall strategy of the proof, with the technical details elaborated in~\Cref{sec: small} and~\Cref{sec: multiscale}.

\subsection{Basic properties within the Gevrey decay set}
\label{subsec: decay-prop}

Let $\mathbb{T}^n = [0, 2\pi]^n$ denote the $n$-dimensional torus. Recalling the vector field associated with the perturbation $H_1$, we write down its Fourier coefficients as
\begin{equation}
\label{eqn: vec-fourier}
\widehat{\pmb{X}}(\pmb{k}) = \frac{1}{(2\pi)^{n}} \int_{\mathbb{T}^{n}} \frac{\partial H}{\partial \overline{\pmb{z}}} e^{-i \langle \pmb{k}, \pmb{\theta} \rangle } d\pmb{\theta},
\end{equation}
for any $\pmb{k} \in \mathbb{Z}^{n}$. Since the perturbation $H_1$ is a polynomial with real coefficients, we can characterize the regularity and boundedness of the resulting vector field within the Gevrey decay set. 
\begin{prop}
\label{prop: vector-field}
Suppose that the Fourier vector belongs to the Gevrey decay set as defined in~\eqref{eqn: gevrey-l2}, i.e., $\widehat{\pmb{z}} \in \mathcal{K}(s)$. Then  there exists a constant $\gamma_1 = \gamma_1(H_1, s) > 0$  such that the components of the associated vector field satisfy
\begin{equation}
\label{eqn: vector-field-Gevrey}
\sup_{\pmb{k} \in \mathbb{Z}^n} \left( \| \widehat{\pmb{X}}(\pmb{k})\| \exp\left\{ |\pmb{k}|^s \right\} \right) \leq \gamma_1. 
\end{equation}
Furthermore, there exists a constant  $\gamma_2 = \gamma_2(H_1, s) > 0$ such that the vector fields satisfy the uniform bound:
\begin{equation}
\label{eqn: vector-field-bound}
\max \big \{ \| \widehat{\pmb{X}}_q\|, \| \widehat{\pmb{X}}_p\|  \big \} \leq \| \widehat{\pmb{X}}\| \leq \gamma_2.
\end{equation}
\end{prop}
The proof relies on the fact that the convolution operation preserves Gevrey decay properties (see~\Cref{subsec: vector-field} for technical details). We now examine the entries of its tangent linear operators. By applying the Fourier expansion to the second-order derivative, we obtain the following kernels for any $\pmb{k}, \pmb{k}' \in \mathbb{Z}^n$:
\[
\left\{ \begin{aligned}
& \mathcal{H}_1(\pmb{k}, \pmb{k}')  = \frac{1}{(2\pi)^{n}} \int_{\mathbb{T}^{n}} \frac{\partial^2 H}{\partial \pmb{z} \partial \overline{\pmb{z}}} e^{-i \langle \pmb{k} - \pmb{k}', \pmb{\theta} \rangle } d\pmb{\theta},  \\
&  \mathcal{H}_2(\pmb{k}, \pmb{k}') = \frac{1}{(2\pi)^{n}} \int_{\mathbb{T}^{n}} \frac{\partial^2 H}{ \partial \overline{\pmb{z}}^2} e^{-i \langle \pmb{k} + \pmb{k}', \pmb{\theta} \rangle } d\pmb{\theta}.  
 \end{aligned} \right.
\]
From these expressions, it follows that the kernels possess a Toeplitz-like and Hankel-like structures, respectively: $ \mathcal{H}_1(\pmb{k}, \pmb{k}') =  \mathcal{H}_1(\pmb{k} - \pmb{k}')$ and $ \mathcal{H}_2(\pmb{k}, \pmb{k}') =  \mathcal{H}_2(\pmb{k} + \pmb{k}')$. Consequently, the total tangent linear operator takes the form
\begin{equation}
 \label{eqn: tangent-linear-operator}
 \frac{\partial \widehat{\pmb{X}}(\pmb{k})}{\partial \widehat{\pmb{z}}(\pmb{k}')}  = \mathcal{H}_1(\pmb{k} - \pmb{k}') + \mathcal{H}_2(\pmb{k} + \pmb{k}') 
\end{equation}
Recall that $S = \partial \widehat{\pmb{X}}_p/ \partial \widehat{\pmb{z}}_p$, introduced in~\eqref{eqn: lin-op}, represents the non-resonant block of the tangent linear operator $\partial \widehat{\pmb{X}}(\pmb{k}) / \partial \widehat{\pmb{z}}(\pmb{k}')$. Following the decomposition $S = S_1 + S_2$, where the entries satisfy $S(\pmb{k}, \pmb{k}') = S_1(\pmb{k}-\pmb{k}') + S_2(\pmb{k} + \pmb{k}')$, its regularity and boundedness are characterized as follows.

\begin{prop}
\label{prop: convolution}
Assume $\widehat{\pmb{z}} \in \mathcal{K}(s)$. There exist constants $\gamma_3, \gamma_4, \gamma_5 > 0$, depending on $H_1$ and $s$, such that the conponents of $S$ satisfy:
\begin{subequations}
\label{eqn: hessian-gevrey}
\begin{empheq}[left=\empheqlbrace]{align} 
         & \sup_{\pmb{k} , \pmb{k}' \in \mathcal{S}} \left(\|S_1(\pmb{k}-\pmb{k}') \| \exp\left\{ |\pmb{k} - \pmb{k}'|^s \right\} \right) \leq \gamma_3 \label{eqn: hessian-gevrey-1} \\
         & \sup_{\pmb{k} , \pmb{k}' \in \mathcal{S}} \left(\|S_2(\pmb{k} + \pmb{k}') \| \exp\left\{ |\pmb{k} + \pmb{k}'|^s \right\} \right) \leq \gamma_4 \label{eqn: hessian-gevrey-2}
 \end{empheq}        
 \end{subequations}
Furthermore, the following uniform bound holds:
\begin{equation}
\label{eqn: hessian-bound}
\max \left\{ \left\| \frac{\partial \widehat{\pmb{X}}_q}{\partial \widehat{\pmb{z}}_p}  \right\|, \|S \| \right\} \leq  \left\| \frac{\partial \widehat{\pmb{X}}}{\partial \widehat{\pmb{z}}}  \right\| \leq \gamma_5. 
\end{equation}
\end{prop}
Note that the entries $S_1(\pmb{k}-\pmb{k}')$ and $S_2(\pmb{k}+\pmb{k}')$ are $n \times n$ matrices; thus, the norm in~\eqref{eqn: hessian-gevrey} refers to the standard matrix norm, i.e., the largest singular value. Since $\partial \widehat{\pmb{X}}/\partial \widehat{\pmb{z}}$ and $S$ are the real symmetric Hessian-type operators,  the norms in~\eqref{eqn: hessian-bound} are understood in the spectral sense, i.e., the maximum absolute value of the eigenvalues. These results follow from Young's convolution inequality and the property that Gevrey decay is preserved under convolution. Detailed proofs are provided in~\Cref{subsec: convolution}. We also require the regularity and boundedness properties for the frequency iteration operator $B$, as defined in~\eqref{eqn: rank-1}.
\begin{prop}
\label{prop: B-operator}
Assume $\widehat{\pmb{z}} \in \mathcal{K}(s)$. There exist constants $\gamma_6, \gamma_7 > 0$ depending on $H_1$ and $s$, such that the components of $B$ satisfy:
\begin{equation}
\label{eqn: b-gevrey}
\sup_{\pmb{k} , \pmb{k}' \in \mathcal{S}} \left( \| B(\pmb{k}, \pmb{k}') \| \exp\left\{ |\pmb{k}|^s + |\pmb{k}'|^s \right\} \right) \leq  \gamma_6.
\end{equation}
Furthermore, the following uniform bound holds:
\begin{equation}
\label{eqn: b-bound}
\| B \| \leq  \gamma_7.
\end{equation}
\end{prop}

The proof is provided in~\Cref{subsec: b-operator}. Finally, we bound the second-order derivative of the vector field generated by the perturbation $H_1$, which constitutes a third-order tensor. 
\begin{prop}
\label{prop: tensor}
Assume $\widehat{\pmb{z}} \in \mathcal{K}(s)$. There exist a constant $\gamma_8 = \gamma_8(H_1, s)$ such that 
\begin{equation}
\label{eqn: tensor}
\max \left\{ \left\| \frac{\partial^2 \widehat{\pmb{X}}_q}{\partial \widehat{\pmb{z}}_p^2}  \right\|, \left\| \frac{\partial S}{\partial \widehat{\pmb{z}}_p } \right \| \right\} \leq  \left\| \frac{\partial^2 \widehat{\pmb{X}}}{\partial \widehat{\pmb{z}}^2}  \right\| \leq \gamma_8. 
\end{equation}
\end{prop}
The proof of~\Cref{prop: tensor} is provided in~\Cref{subsec: tensor}. Let $\gamma = \max_{1 \leq j \leq 8} \gamma_j$, where each $\gamma_j$ for $j=1, \ldots, 8$ depends on the perturbation Hamiltonian $H_1$. By scaling $H_1$, or, equivalently, by choosing a sufficiently small $\varepsilon$), we assume without loss of generality that $\gamma  \leq 1/(2e) \leq 1/2$. Utilizing the estimates from~\Cref{prop: vector-field}, we establish the following result regarding the frequency drift.
\begin{prop}
\label{prop: frequency-drift}
Assume $\widehat{\pmb{z}} \in \mathcal{K}(s)$. The frequency drift determined by the $Q$-equation~\eqref{eqn: q-eqn-vector} satisfies: 
\begin{equation}
\label{eqn: frequency-drift-bound}
\left|  \pmb{\omega}' -\pmb{\omega} \right| \leq  \varepsilon.
\end{equation}
and the Jacobian of the frequency map satisfies: 
\begin{equation}
\label{eqn: frequency-drift-derivative}
1 - \varepsilon \leq \left\| \frac{\partial \pmb{\omega}'}{\partial \pmb{\omega}} \right\|_2 \leq 1 +  \varepsilon.
\end{equation}
\end{prop}

\subsection{Conditions for efficient implementations}
\label{subsec: implementation-condition}

The heart of the alternating numerical procedure, as shown in~\Cref{fig: algorithm-nm}, lies in the update of the non-resonant vector via the dimension-enlarged Newton scheme~\eqref{eqn: p-eqn-numerical}. The major challenge in this approach is the ``small-divisor'' problem, which implies that the truncated matrix $(T + \varepsilon B)_{N}$ may become increasingly singular as the dimension $N$ grows. To ensure an effective implementation, we must prevent the norm of the inverse truncation matrix from growing excessively. To quantify the allowable singularity and the required numerical precision, we define a size-dependent threshold parameter:
\begin{equation}
\label{eqn: epsilon-N}
\varepsilon_N : = \exp\left\{ - (\log N)^{15} \right\},
\end{equation} 
which serves as the error threshold for each truncation level $N$. In numerical practice, we evaluate the perturbed operator $T + \varepsilon B$ rather than $T$ alone. The validity of the implementation depends on the following two conditions.
\begin{tcolorbox}[breakable]
\begin{condition}[Implementation Conditions]
\label{cond: implementation}
The truncated linearized operator $(T + \varepsilon B)_{N}$ must satisfy the following two conditions:
\begin{itemize}
\item[(1)] \textbf{Inversion condition}: The operator norm of the inverse is bounded by the reciprocal of the size parameter:
\begin{equation}
\label{eqn: inverse-grow}
\| (T + \varepsilon B)_{N}^{-1} \|_2 \leq \frac{1}{\varepsilon_{N}}.
\end{equation}
\item[(2)] \textbf{Localization condition}: For sufficiently large spatial separations in the lattice $\Lambda_N$, specifically when $|\pmb{k} - \pmb{k}'| \geq N^{\frac12}$ and $|\pmb{k} + \pmb{k}'| \geq N^{\frac12}$ (denoted by $|\pmb{k} \pm \pmb{k}'| \geq N^{\frac12}$ for short), the entries of the inverse matrix must exhibit Gevrey-type decay:
\begin{equation}
\label{eqn: off-diagonal}
| (T + \varepsilon B)_{N}^{-1}\left( \pmb{k}, \pmb{k}' \right) | \leq \exp\left\{ - \frac{|\pmb{k} - \pmb{k}'|^{s}}{2}\right\} +  \exp\left\{ - \frac{|\pmb{k} + \pmb{k}'|^{s}}{2}\right\}.
\end{equation}
\end{itemize}
\end{condition}
\end{tcolorbox}

Regarding the inversion condition in~\eqref{eqn: inverse-grow}, we emphasize that all estimates are conducted using the $\ell_2$-norm. The localization condition~\eqref{eqn: off-diagonal} specified here refers to two-sided off-diagonal decay for lattice boxes centered at the origin $\mathbf{0}$. For ``off-side" boxes of size $N$ with centers $\pmb{k}_0 \notin \Lambda_{2N}$, the distance $|\pmb{k} + \pmb{k}'| \geq 2N$ is always satisfied; consequently, we only need to consider the off-diagonal decay in terms of $|\pmb{k} - \pmb{k}'|$. The analysis presented in \Cref{sec: small} and \Cref{sec: multiscale} focuses exclusively on the $\ell_2$-norm of the inverse and the localization of its entries for the operator $(T+\varepsilon B)$ restricted to off-side boxes. According, we postpone the detailed treatment of two-sided off-diagonal decay until~\Cref{subsec: verification}. Furthermore, as these objectives do not require estimating the norm of the derivatives, the composite norm is unnecessary for these sections. We reserve its use for derivative estimates of the linear operator $(T + \varepsilon B)_N$, which are essential only for the Iteration Lemma (\Cref{lem: iteration}). For clarity and conciseness, the formal discussion of the composite norm is deferred to \Cref{lem: inverse-derivative} in \Cref{subsec: induction}. Finally, note that the inclusion of the operator $B$ in the numerical implementation renders the operator non-self-adjoint; therefore, the $\ell_2$-norm here corresponds to the largest singular value.

\subsection{Outline of the inductive and multi-scale strategy}
\label{subsec: organization} 

The proof of the main convergence statement~(\Cref{thm: main}) follows a structural and multi-scale hierarchy. While the conceptual core of our approach remains rooted in the framework established by~\citet{bourgain1998quasi}, our approach departs significantly in both presentation and execution. Guided by~\Cref{cond: implementation}, we provide a more modular and transparent strategy as outlined below.

\paragraph{Small-scale preparation (\Cref{sec: small})}  We begin by establishing foundational results for small-scale boxes to anchor the subsequent multi-scale analysis. Departing from the high-level techniques utilized by~\citet{bourgain1998quasi}, we propose a significantly streamlined argument based on the Neumann series. Additionally, we clarify the relationship between the nearly-resonant set and the Diophantine condition common in classical KAM theory.

\paragraph{Multi-scale analysis (\Cref{sec: multiscale})} This section streamlines the multi-scale analysis of~\citet{bourgain1998quasi}, introducing a more efficient execution divided into two key components:

\begin{itemize}
\item Clustering of singular boxes (\Cref{subsec: size-reduce-cluster}).  We demonstrate how the second Melnikov condition effectively controls the diffusion of resonant clusters. Our simplified approach provides a more transparent bound on the complexity of these clusters. 
\item Multi-scale induction (\Cref{subsec: inversion-localization}). Using the small-box results as inductive bases, we employ the resolvent identity to derive a clear inductive step. We show that as long as the inversion condition~\eqref{eqn: inverse-grow} holds, the localization condition~\eqref{eqn: off-diagonal} is naturally satisfied. This refinement reconstructs the logical progression explicitly, bypassing the dense, non-linear algebraic burdens typically associated with inductive resolvent bounds~\citep{bourgain1998quasi}.
\end{itemize}
Furthermore, we demonstrate that the second Melnikov condition is not omitted rather localized. In finite dimensions, this corresponds directly to the classical the second Melnikov condition. This localized implementation is vital for PDE applications, such as the wave equation, which we explore in our forthcoming work.

\paragraph{Inductive Iteration and Completion (\Cref{sec: completion})} The formal verification for~\Cref{cond: implementation} is conducted in~\Cref{subsec: verification}, with specific focus on two-sided off-diagonal decay and the exclusion of ``bad" frequencies. We restrict the linear operator to the lattice boxes with centers at the origin and compute its inverse as the dimension is incrementally expanded toward infinity. In~\Cref{subsec: induction}, we establish~\Cref{lem: iteration} via rigorous inductive estimates, which serves as the final bridge to complete the convergence proof for our numerical algorithms.

\paragraph{Numerical experiments (\Cref{sec: numer})} Finally, we demonstrate the efficiency of our approach using the Duffing equation and the Hénon-Heiles model. We highlight two key advantages over traditional methods. Our method provides precise approximations of frequencies and Fourier coefficients where standard symplectic integrators reach their limits. While symplectic algorithms conserve the symplectic form, they often fail to track precise pointwise positions over long time scales due to phase drift. Our multi-scale inductive approach ensures both structural stability and superior accuracy in predicting the exact state of the system at any given time.

\section{Restricted operators on small boxes}
\label{sec: small}

From this section onward, we analyze the restricted operator $(T+\varepsilon B)_{N}$ as it arises within the dimension-enlarged Newton scheme~\eqref{eqn: p-eqn-numerical}. Specifically, our main objective is to establish the conditions under which this operator satisfies~\Cref{cond: implementation}. In the iterative scheme~\eqref{eqn: p-eqn-numerical}, the reference box size at the $r$-th step is set as $N_{r+1} = M^{r+1}$. To establish the benchmark for scale separation, we determine the specific value of $M$ relative to the perturbation parameter $\varepsilon$ as:
\begin{equation}
\label{eqn: M}
M = \exp\left\{ \left( \log \frac{1}{\varepsilon} \right)^{\frac{1}{20}} \right\}
\end{equation}
Based on the scaling factor $M$, we categorize the restricted lattice boxes defined in~\eqref{eqn: box-fourier-truncated} into two distinct regimes:
\begin{itemize}
\item \textbf{Small restricted boxes}. The box size $N$ falls within the range
\begin{equation}
\label{eqn: small-size}
M_0= \exp \left\{ \left( \log M \right)^{\frac{1}{20}} \right\} \leq N \leq M.
\end{equation}
In this regime, the dynamics are characterized by low-frequency oscillations or slow rotations.

\item \textbf{Large restricted boxes} The box size satisfies 
\begin{equation}
\label{eqn: large-size}
N > M,
\end{equation}
where the dynamics here involves high-frequency oscillation or rapid rotations.
\end{itemize}

This section focuses exclusively on small restricted boxes. The analysis of large restricted boxes involves a more complex multi-scale approach, which is deferred to~\Cref{sec: multiscale}.

\subsection{Nearly-resonance sets}
\label{subsec: nearly-resonance-sets}

The presence of ``bad'' frequencies typically gives rise to small-divisor problems, which may render the restricted operator non-invertiable. To characterize such frequencies explicitly, we introduce the nearly-resonant set associated with a finite lattice box $\Lambda_{2(N+1)}$.
\begin{defn}
\label{defn: nearly-resonant}
Let $\Omega \subseteq \mathbb{R}^{n}$ be a bounded domain. The \textit{nearly-resonant set} associated with the restricted box $\Lambda_{2(M+1)}$ is defined as
\begin{equation}
\label{eqn: nearly-resonance}
\Omega_{M}^{\tau} = \left\{ \pmb{\omega} \in \Omega \bigg | \left| \langle \pmb{k}, \pmb{\omega} \rangle\right| < \frac{1}{ | \pmb{k} |^{\tau}}, \; \pmb{k} \in \Lambda_{2(M+1)}\setminus\{0\} \right\}. 
\end{equation}
\end{defn}

The complement, $\Omega \setminus\Omega_{M}^{\tau}$, corresponds to the classical Diophantine condition used in KAM theory (see, e.g.,~\citet{arnold2006mathematical}). Here, it serves to exclude ``bad'' frequencies within the small boxes.  In the case of lower-dimensional quasi-periodic solutions, the nearly-resonant set naturally decomposes into two components corresponding to the first and second Melnikov conditions: the former, associated with the original small box, identifies ``bad'' frequencies that cause the operator's eigenvalues to vanish (see~\Cref{subsec: original-small}), and the latter, arising from small translated boxes, addresses ``difference resonances''. In this scenario, even if a singular eigenvalue may exist, all other eigenvalues of the translated block remain well-behaved. Consequently, if we can ensure that the singular eigenvalue remains ``good'', then the inverse and its off-digonal entries satisfy~\Cref{cond: implementation} (see~\Cref{subsec: translated-small}). Further results concerning lower-dimensional quasi-periodic solutions will be presented in forthcoming work.  The following lemma summarizes the measure estimates for these excluded (small-divisor)  frequency sets.  

\begin{lemma}
\label{lem: mes-dig}
Let $\tau > n - 1$ be a fixed exponent. Then there exist a constant $\kappa = \kappa(n, \tau, \Omega) > 0$, depending onthe dimension $n$, the domian $\Omega$,  and the parameter $\tau$, such that the  nearly-resonance set satisfies
\begin{equation}
\label{eqn: small-mes}
\mathrm{mes}(\Omega^{\tau}_{M}) \leq \kappa. 
\end{equation}
\end{lemma}

The proof relies only on elementary measure estimates and the convergence of a simple lattice sum, and is therefore deferred to~\Cref{sec: mes-dig}.  
\begin{remark}
\label{rem: num-irrational}

Beyond the analysis of the small boxes, the Melnikov conditions can be interpreted more deeply within the multi-scale analysis framework (see~\Cref{sec: multiscale}): 
\begin{itemize}
\item~\textbf{Single-box dynamics}:  The first Melnikov condition is viewed as a single-box eigenvalue problem. It excludes frequency measures where the restricted operator fails to be ``gap-broadening'' or invertible.
\item~\textbf{Inter-box interactions}:  The second Melnikov condition governs the interaction between two distinct boxes (or box translations). It ensures that resonance does not "propagate" or couple across different spatial locations.
\end{itemize}
A fundamental departure from classical KAM theory, where infinite-dimensional holomorphic equations require both Melnikov conditions to be satisfied simultaneously at every iteration, is the hierarchical application of these conditions. In the multi-scale analysis, based on a finer decomposition of Fourier modes, these conditions are imposed progressively scale by scale. This relaxation constitutes a cornerstone of~\citet{bourgain2005green}, allowing the treatment of higher-dimensional PDEs where frequency clusters would otherwise violate the standard (global) second Melnikov condition. 
\end{remark}

\subsection{Restriction to the original small box}
\label{subsec: original-small}

Excluding the set $\Omega_{M}^{\tau}$ of the ``bad ''frequencies, we now analyze the inverse of the restricted operators within the small-size regime specified in~\eqref{eqn: small-size}. We begin by focusing on the operator restricted to the original box $\Lambda_N$. 
\begin{theorem}
\label{thm: small-size-original}
Assume that $\widehat{\pmb{z}} \in \mathcal{K}(s)$. If the frequency satisfies $\pmb{\omega} \in \Omega \setminus \Omega_{M}^{\tau}$, then the restricted operator $T_N$ is invertible, and its inverse satisfies the operator norm bound
\begin{equation}
\label{eqn: small-size-original-inverse}
\| (T+\varepsilon B)_{N}^{-1} \|_2 \leq \frac{1}{\varepsilon_N}.
\end{equation}
In addition, for any $\pmb{k} \neq \pm  \pmb{k}' \in \Lambda_{N}$, the two-sided off-diagonal entries satisfy the Gevrey decay estimate:
\begin{equation}
\label{eqn: small-size-off-diagnol}
 \left| (T+\varepsilon B)_{N}^{-1} \left(  \pmb{k}, \pmb{k}' \right) \right|  \leq  \frac{1}{2} \left( \exp \left\{-| \pmb{k} - \pmb{k}'|^{s} \right\} +  \exp \left\{-| \pmb{k} + \pmb{k}'|^{s} \right\} \right).
\end{equation} 
\end{theorem}

Given that $\widehat{\pmb{z}} \in \mathcal{K}(s)$,~\Cref{prop: frequency-drift} implies that the drifted frequency $\pmb{\omega}'$ satisfies~\eqref{eqn: frequency-drift-bound}.  For the diagonal operator $D_{N}$, its eigenvalues admit the uniform lower bound: 
\begin{align}
 | D_{  N; j, \pmb{k}} | & \geq \left| - \langle \pmb{k}, \pmb{\omega} \rangle + \omega_j \right| - \left|  \langle \pmb{k}, \pmb{\omega}' - \pmb{\omega} \rangle  \right|  \nonumber  \\
                                & \geq \frac{1}{n^{\tau}(N+1)^{\tau }} - N \exp\left\{ - (\log M)^{20} \right\} \geq  \frac{1}{2n^{\tau}(N+1)^{\tau }},   \label{eqn: uniform-lower-original-small}
\end{align}
for any $j = 1, \ldots, n$ and $-\pmb{k} + \pmb{e}_j \in \Lambda_N \setminus \{0\}$. This lower bound ensures that the diagonal part of the operator remains sufficiently far from zero, preventing singular behavior. Consequently, the inverse of the restricted operator $T_{N}$ can be estimated via a Neumann series. Since the off-diagonal terms involve only convolution operations that interact predictably with the Gevrey decay, the detailed argument follows standard procedures and is therefore deferred to~\Cref{sec: small-size-original}.

\subsection{Restriction to a translated small box}
\label{subsec: translated-small}

In this section, we analyze the behavior of the operator $T + \varepsilon B$ when restricted to a translated small box. For any fixed $\pmb{k}_0 \in \mathbb{Z}^n \setminus \{0\}$, the translated box is defined as $\pmb{k}_0 + \Lambda_N$. 
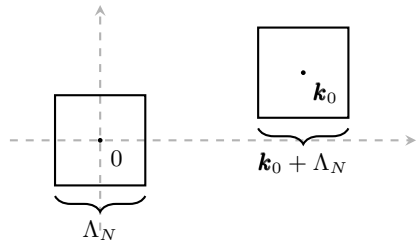
\begin{figure}[htb!]
\centering
\begin{tikzpicture}[scale=0.6, >=stealth]
    \draw[->, thick, dashed, gray!60] (-2,0) -- (7,0);
    \draw[->, thick, dashed, gray!60] (0,-2) -- (0,3);

    \coordinate (k1) at (0, 0);
    \fill (k1) circle (1.5pt) node[below right, font=\footnotesize] {$0$};       
    \draw[line width=0.8pt] ($(k1) - (1.0, 1.0)$) rectangle ($(k1) + (1.0, 1.0)$);
    \draw [decorate, decoration={brace, amplitude=6pt, mirror, raise=4pt}, thick] 
        ($(k1) + (-1.0, -1.0)$) -- ($(k1) + (1.0, -1.0)$) 
        node [midway, yshift=-0.6cm, font=\footnotesize] {$\Lambda_N$};    

    \coordinate (k0) at (4.5, 1.5);
    \fill (k0) circle (1.5pt) node[below right, font=\footnotesize] {$\pmb{k}_0$};
    \draw[line width=0.8pt] ($(k0) - (1.0, 1.0)$) rectangle ($(k0) + (1.0, 1.0)$);
    \draw [decorate, decoration={brace, amplitude=6pt, mirror, raise=4pt}, thick] 
        ($(k0) + (-1.0, -1.0)$) -- ($(k0) + (1.0, -1.0)$) 
        node [midway, yshift=-0.6cm, font=\footnotesize] {$\pmb{k}_0 + \Lambda_N$};
 
\end{tikzpicture}
\caption{The original small box $\Lambda_N$ and its translated counterpart $\pmb{k}_0 + \Lambda_N$.}
\label{fig: box-translation}
\end{figure}
This spatial translation shifts the index set of the operator, as illustrated in~\Cref{fig: box-translation}. We denote the restriction of the operator $T + \varepsilon B$ to this translated domain as 
\begin{equation}
\label{eqn: translated-box-small}
(T + \varepsilon B)_{\pmb{k}_0, N} := (T + \varepsilon B) \big |_{\pmb{k} \in \pmb{k}_0 + \Lambda_N}.
\end{equation} 
Our analysis focuses on cases where the box center lies outside the region $\Lambda_{2N}$ (i.e., $\pmb{k}_{0} \notin \Lambda_{2N}$). To evaluate this restricted operator, we recall the properties of the operator $B$ defined in~\eqref{eqn: rank-1}, which was constructed via a numerical algorithm to facilitate frequency updates. According to~\Cref{prop: B-operator}, the operator $B$ exhibits double-index Gevrey decay. This property ensures that the norm of the restricted operator satisfies a Gevrey-type bound: 
\[
\|B_{\pmb{k}_0, N}\| \leq (2N)^{n} \exp\left\{ -N^{s} \right\}
\]
for any $\pmb{k}_0  \notin \Lambda_{2N}$. In practical terms, this estimate implies that the contribution of $B_{\pmb{k}_0, N}$ is negligible compared to any scale involving exp-polylogarithmic decay, such as $\exp(-(\log N)^\nu)$ for $\nu > 0$. In other words, since Gevrey decay (with $s > 0$) is significantly faster than these scales, the additional operator $B$ do not exert any meaningful influence on the spectral analysis within this regime.

To bound the inverse of $T_{\pmb{k}_0, N}$, we introduce the shifted resonance parameter $\sigma_0 = - \langle \pmb{k}_0, \pmb{\omega}' \rangle$. This allows us to express the restricted operator as a spectral shift of $T_N$: 
\begin{equation}
\label{eqn: translation}
T_{\pmb{k}_0, N} = \sigma_0 + T_{N}, 
\end{equation}
which we denote more concisely as $T_{N}^{\sigma_0}$. Since the operator $S$ is translationally invariant, the translation acts exclusively on the diagonal component, such that $D_{N}^{\sigma_0} = D_{\pmb{k}_0, N}$. By invoking~\Cref{thm: small-size-original}, we ensure that for all $\pmb{\omega} \in \Omega/\Omega_{M}^{\tau}$ and $\widehat{\pmb{z}} \in \mathcal{K}(s)$, the entries of $D_{N; j, \pmb{k}}$ satisfy the lower bound established in~\eqref{eqn: uniform-lower-original-small}. Following~\Cref{defn: nearly-resonant} regarding the non-resonant set, we establish the following property. 
\begin{prop} 
\label{prop: small-eig-one}
Assume that $\widehat{\pmb{z}} \in \mathcal{K}(s)$. If the frequency satisfies $\pmb{\omega} \in \Omega \setminus \Omega_{N}^{\tau}$, the diagonal operator $ D_{N}^{\sigma_0}$ admits at most  one index $\pmb{k} \in \Lambda_{N}$ such that this diagonal entry $ D^{\sigma_0}_{N; j, \pmb{k}} = - \sigma_0 - \langle \pmb{k}, \pmb{\omega}' \rangle + \omega_j $ has an absolute value smaller than $(4n)^{-\tau}(N+1)^{-\tau}$.
\end{prop}
The proof follows the same argument presented in~\Cref{subsec: original-small} and is therefore deferred to~\Cref{sec: small-eig-one}. In light of~\Cref{prop: small-eig-one}, we define the unique index (if it exists) as the singular index $\pmb{k}_\star$, and its corresponding value $\sigma_\star = - \langle \pmb{k}_{\star}, \pmb{\omega}' \rangle$ as the \textit{minimal resonance gap} (or \textit{minimal shift}). Consequently, the singular site is defined as the singleton set $\Pi = \{ \pmb{k}_{\star} \}$. To maintain consistency with the subsequent multi-scale analysis, we re-index the small box as $\Xi_0 = \pmb{k}_0 + \Lambda_N$. By applying a Neumann series and proceeding analogously to~\Cref{thm: small-size-original}, we can reformulate~\Cref{prop: small-eig-one} into the following lemma for the non-resonant index set $\Xi_0 \setminus \Pi$, which serves as the inductive base for the multi-scale induction in~\Cref{subsec: inversion-localization}. 
\begin{lemma}
\label{lem: small-size-parpare}
Assume that $\widehat{\pmb{z}} \in \mathcal{K}(s)$ and $\pmb{\omega} \in \Omega \setminus \Omega_{M}^{\tau}$. For any center $\pmb{k}_0 \notin \Lambda_{2(N+1)}$, then the restricted operator $(T+\varepsilon B)|_{\Xi_0 \setminus \Pi}$ is invertible, and its inverse satisfies the operator norm bound
\begin{equation}
\label{eqn: small-size-inverse-prepare}
\| ((T+\varepsilon B)|_{\Xi_0 \setminus \Pi} )^{-1} \|_2 \leq (4n)^{\tau}(N+1)^{\tau}.
\end{equation}
In addition, the off-diagonal entries of $((T + \varepsilon B_{N})_{\Xi_0 \setminus \Pi})^{-1}$ satisfy the decay estimate
\begin{equation}
\label{eqn: small-size-off-diagnol-prepare}
 \left| ((T+\varepsilon B)_{\Xi_0 \setminus \Pi})^{-1} \left(  \pmb{k}, \pmb{k}' \right) \right|  \leq  \exp \left\{-| \pmb{k} - \pmb{k}'|^{s} \right\},
\end{equation}
for any $\pmb{k} \neq  \pmb{k}' \in \Xi_0 \setminus \Pi$.
\end{lemma}

For small translated boxes, the inversion condition~\eqref{eqn: inverse-grow} can be verified by examining the diagonal operator alone. Since the perturbation parameter is assumed to satisfy  $\varepsilon \ll \varepsilon_N$, the control of the diagonal operator's entries is already well-separated from the resonance. Specifically, if the minimal resonant gap satisfies $|\sigma_{\star} + \omega_j| \geq \varepsilon_{N}$ for $j=1, \ldots, n$, then we establish the restricted operator on a translated box satisfies the localization condition~\eqref{eqn: off-diagonal}
\begin{theorem}
\label{thm: small-scale-tran}
Assume that $\widehat{\pmb{z}} \in \mathcal{K}(s)$ and $\pmb{\omega} \in \Omega \setminus \Omega_{M}^{\tau}$. For any $\pmb{k}_0 \notin \Lambda_{2(N+1)}$, if the minimal resonant gap satisfies $|\sigma_{\star} + \omega_j| \geq \varepsilon_{N}$, then the restricted operator $(T+\varepsilon B)_{\pmb{k}_0,N}$ is invertible, and its inverse satisfies the operator norm bound
\begin{equation}
\label{eqn: trans-small-size-original-inverse}
\| (T+\varepsilon B)_{\pmb{k}_0,N}^{-1} \|_2 \leq \frac{1}{\varepsilon_N}.
\end{equation}
In addition, the off-diagonal entries of $T_{N}^{-1}$ satisfy the decay estimate
\begin{equation}
\label{eqn: trans-small-size-off-diagnol}
 \left| (T+\varepsilon B)_{\pmb{k}_0,N}^{-1} \left(  \pmb{k}, \pmb{k}' \right) \right|  \leq  \exp \left\{-| \pmb{k} - \pmb{k}'|^{s} \right\},
\end{equation}
for any $\pmb{k} \neq  \pmb{k}' \in \pmb{k}_0+ \Lambda_{N}$.

\end{theorem}

While~\Cref{lem: small-size-parpare} may seem redundant for the small-box regime specified in~\eqref{eqn: small-size}, it serves two primary purposes: first, it characterizes the behavior at the smallest scale of our multi-scale analysis; and second, it provides the necessary inductive basis for that analysis. The transitions from~\Cref{prop: small-eig-one} to~\Cref{lem: small-size-parpare}, and subsequently from~\Cref{lem: small-size-parpare} to~\Cref{thm: small-scale-tran}, are achieved via the resolvent identity, as demonstrated in~\Cref{subsec: inversion-localization}.

\section{Multi-scale analysis: extension to large boxes}
\label{sec: multiscale}

In this section, we extend our analysis to restricted large boxes of size $N > M$, centered at $\pmb{k}_0 \notin \Lambda_{2N}$. To control the propagation of resonances across these expanding scales, we employ a hierarchical exclusion process. This method, fundamentally established by~\citet{bourgain1994construction, bourgain1998quasi}, involves the iterative removal of ``bad'' frequencies that trigger small divisors. Below, we present a streamlined version of this technique, focusing on the systematic refinement of the non-resonant set.

\subsection{Size reduction: clustering of singular boxes}
\label{subsec: size-reduce-cluster}

To bridge the operator properties across disparate scales, we introduce a reduced scale $N'$ derived from the larger scale $N$ via the following polylogarithmic scaling:\begin{equation}
\label{eqn: 10-log-size}
N' = \exp\left\{ (\log N)^{\frac{1}{10}} \right\}. 
\end{equation}
This specific scaling allows us to isolate resonant interactions by identifying parameters where the small divisors, specifically those involving the frequency difference, fall below a critical threshold.
\begin{defn}
\label{defn: second-melnikov}
The \textit{difference nearly-resonant set} is defined as: 
 \begin{equation}
\label{eqn: nearly-resonant-multiscale}
\Omega_{2, N} = \left\{ \pmb{\omega} \in \Omega \bigg |  \left| - \langle \pmb{k}, \pmb{\omega}' \rangle + \omega_{j_1} - \omega_{j_2} \right| <4\varepsilon_{N'}, \; \forall \pmb{k} \in \Lambda_{2N} \setminus \Lambda_{2N'}, \; \forall j_1, j_2 = 1, \ldots, n \right\}.
\end{equation}
The \textit{total difference nearly-resonant set} is subsequently as the union over all scales $N > M$:
 \begin{equation}
\label{eqn: nearly-resonant-multiscale-total}
\Omega_{2} = \bigcup_{N > M } \Omega_{2, N}.\footnote{The subscript ``2'' identifies the nearly-resonant set associated with the second Melnikov condition in classical KAM theory. Its localized interpretation is briefly noted in~\Cref{rem: second-melnikov}, while its wider applications are explored in our forthcoming work. In contrast, the subscript ``$1$'' refers to the single-mode nearly-resonant set (first Melnikov condition).  The procedure for excluding ``bad'' frequencies and its multi-scale extension are detailed in~\Cref{subsec: small-large}. } 
\end{equation}
\end{defn}

Under the assumption that the frequency satisfies $\pmb{\omega} \notin \Omega_{2, N}$, we analyze the behavior of all smaller boxes of size $N'$ contained within the larger box of size $N$. We can ensure that if the operator $T + \varepsilon B$ restricted to these smaller boxes violates the inversion condition~\eqref{eqn: inverse-grow}, then the centers of these singular boxes must be clustered. Specifically, their centers are located within a distance of $4N'$, as illustrated in~\Cref{fig: second-melnikov}. 

\begin{figure}[htb!]
\centering

\begin{tikzpicture}[scale=0.6, >=stealth]

    \draw[line width=0.8pt] (-5, -4) rectangle (5, 4);
    \node at (3.5, 3) {\large $N$};
    
    \draw[line width=0.8pt] (-3, -2) rectangle (3, 2);
    
    \draw[line width=0.8pt] (-2, -0.5) rectangle (-0.5, 1);
    \coordinate (center1) at (-1.25, 0.25);
    \node at (center1) {\tiny $\bullet$};
    \node[below] at (center1) {\footnotesize $\pmb{k}_{0,1}$};
    
    \draw[line width=0.8pt] (0.5, -0.5) rectangle (2, 1);
    \coordinate (center2) at (1.25, 0.25);
    \node at (center2) {\tiny $\bullet$};
    \node[below] at (center2) {\footnotesize $\pmb{k}_{0,2}$};

    \node at (-0.9, 0.7) {\footnotesize $N'$};
    \node at (1.6, 0.7) {\footnotesize $N'$};
    
    \draw[line width=0.8pt] (center1) -- (center2);
    
    \draw [decorate, decoration={brace, amplitude=8pt, raise=4pt, mirror}, line width=0.8pt]
        (-3, -2) -- (3, -2);
    
    \node[below, align=center, font=\normalsize] 
        at (0, -2.8) 
        {$| \pmb{k}_{0,1} - \pmb{k}_{0,2} |_{\infty} < 4N'$};

\end{tikzpicture}

\caption{The clustering of singular box centers as the size reduction.}
\label{fig: second-melnikov}
\end{figure}
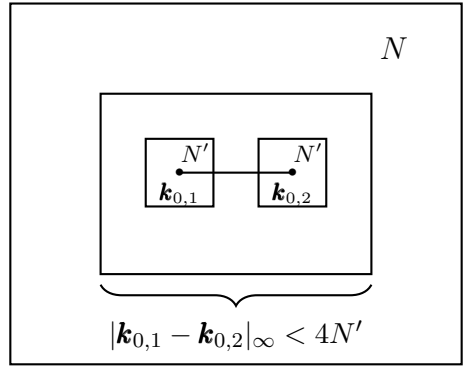

Next, we provide a formal justification for the clustering phenomenon describe above. Consider two small boxes with centers $\pmb{k}_{0,1}, \pmb{k}_{0,2} \in \pmb{k}_0 + \Lambda_{N} $. Given that the frequency satisfies $\pmb{\omega} \notin \Omega_{2,N}$, the entries of the diagonal operators, $D_{\pmb{k}_{0,1}, N'}$ and $D_{\pmb{k}_{0,2}, N'}$, must be sufficiently distinct for any $| \pmb{k}_{0,1} - \pmb{k}_{0,2} |_{\infty} \geq 4N'$. Owing to the smallness of the operator $B$ (introduced in the numerical update) and the translational invariance of the nonlinear perturbation operator $S$, we can bound the difference between the smallest singular values of the restricted operators $T_{\pmb{k}_{0,1}, N'}$ and $T_{\pmb{k}_{0,2}, N'}$. Specifically, we establish the following estimate: 
\begin{equation}
\label{eqn: signle-value-diff}
| \min s(T+\varepsilon B)_{\pmb{k}_{0, 1}, N} - \min s(T+\varepsilon B)_{\pmb{k}_{0, 2}, N} | \geq 3\varepsilon_{N'}
\end{equation}
which demonstrates that if both boxes are ``singular'' (i.e., their smallest singular values are small), they cannot be far apart. To formalize this spatial clustering, we define $\Sigma(\pmb{k_0}, N) $ as the set of box centers within the larger domain $\pmb{k}_0 + \Lambda_N$ that violates the inversion condition~\eqref{eqn: inverse-grow}
\begin{equation}
\label{eqn: eqn-violated-center}
\Sigma(\pmb{k_0}, N) = \left\{ \pmb{k}_0' \in \pmb{k}_0 + \Lambda_N, \; \pmb{k}_{0} \notin \Lambda_{2N} \bigg | \| (T + \varepsilon B)_{\pmb{k}_0', N'}^{-1} \|_2 > \frac{1}{\varepsilon_{N'}} \right\}.
\end{equation}
We then rigorously characterize the spatial clustering of these singular centers with the following lemma.

\begin{lemma}
\label{lem: second-melnikov}
Assume that $\widehat{\pmb{z}} \in \mathcal{K}(s)$ and $\pmb{\omega} \in \Omega\setminus (\Omega_{M}^{\tau} \cup  \Omega_{2, N})$. If a box center satisfies $\pmb{k}'_{0} \in \Sigma(\pmb{k_0}, N) $, then it must be located within a small box of radius $4N'$:
\[
\pmb{k}_0' \in \pmb{k}_0'' + \Lambda_{4N'}
\]
where $\pmb{k}_0'' \in \pmb{k}_0 + \Lambda_N$. 
\end{lemma}

\begin{remark}
\label{rem: second-melnikov}
As demonstrated above, the total difference nearly-resonant set corresponds to the second Melnikov condition, specifically regarding the persistence of low-dimensional tori. However, this construction marks a pivotal departure from classical KAM theory. Unlike classical theory, which relies on a uniform, ``all-at-once'' exclusion of resonant sets from the outset, the difference nearly-resonant set is excluded incrementally as the box size $N$ increases. This iterative exclusion necessitates the multi-scale analysis, distinguishing it from standard KAM techniques adapted for the PDE setting (see~\citep{kuksin2000analysis, liu2011kam, yuan2021kam}). When infinite-dimensional ``normal'' frequencies arise from spatial Fourier decomposition, their influence is inherently localized rather than global, being determined by their specific magnitudes and directions. As captured by~\citet{bourgain2005green}, accounting for such localized interactions is essential for the rigorous analysis of high-dimensional PDEs, particularly those involving multiple identical parameters or repeated eigenvalues (multiplicity), which is further exploited in our forthcoming work.
\end{remark}

\subsection{Multi-scale induction: inversion implies localization}
\label{subsec: inversion-localization}

Following the reduction scheme established in~\Cref{subsec: size-reduce-cluster}, the initial large box is iteratively scaled down to a size $4N'$, where $M_0 \leq N' \leq M$ and $M_0 \leq 4N' \leq M$. This process generates a nested sequence of boxes, an ``onion-like'' hierarchy,  where each layer is strictly contained within the previous one. The structure relies on the interplay between resonance and inversion across scales. By~\Cref{prop: small-eig-one}, each such small box in this sequence is guaranteed to contain at most one minimal resonance gap $\sigma_{\star} =  \langle \pmb{k}_{\star}, \pmb{\omega}'  \rangle$.  Simultaneously, for each box centered in the intermediate (annular) region,~\Cref{lem: second-melnikov} ensures that the restricted operator satisfies the inversion condition~\eqref{eqn: off-diagonal}. Importantly, this nested configuration is inherited by the sub-boxes within the intermediate regions, thereby preserving structural uniformity across all scales. This multiscale reduction is categorized into three distinct sets:
\begin{itemize}  
\item \textbf{Resonant boxes}: The sequence of boxes containing the singular sites $ \Pi$:  
\[
\pmb{k}_0 + \Lambda_N = \Xi_{J+1} \supset \Xi_{J} \supset \cdots \supset \Xi_0 \supset \Pi,
\]
where the scale is defined by $|\Xi_j| = 4N_j$ for $j = 0, \ldots, J$.
\item \textbf{Intermediate sub-boxes}: The boxes located within the region $\Xi_{j+1} \setminus \Xi_j$, denoted as $\Gamma_j$, each of fixed size $N_j$.
\item \textbf{Non-resonant index sets}: The sets of indices excluding the singular site are defined as $\Upsilon_j = \Xi_{j+1} \setminus \Pi$ for $t = -1,0, \ldots, J$.
\end{itemize}
The spatial relationship between the resonant box chain, the non-resonant intermediate sub-boxes, and the singular site is shown in~\Cref{fig: reduce-size}.
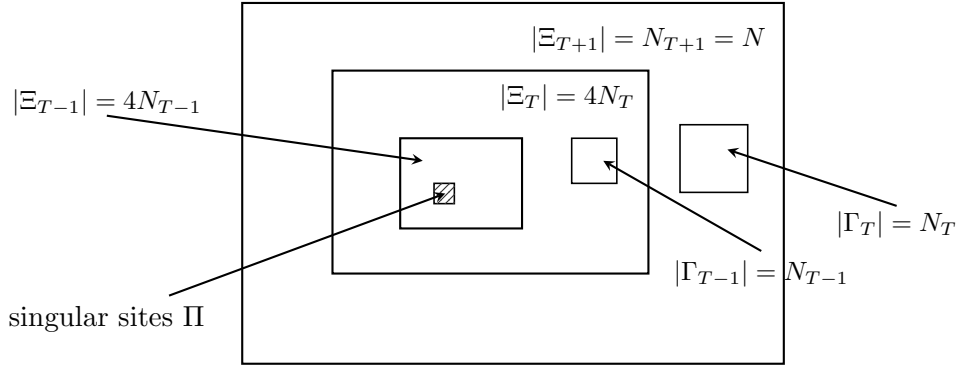
\begin{figure}[htb!]
\centering
\begin{tikzpicture}[scale=0.6, >=stealth]

    \draw[line width=0.8pt] (-7, -4) rectangle (5, 4);
    \node at (2.0, 3.2) {\small $|\Xi_{T+1}|=N_{T+1}=N$}; 
    
    \draw[line width=0.8pt] (-5, -2) rectangle (2, 2.5);
    \node at (0.2, 1.9) {\small $|\Xi_{T}| = 4N_{T}$}; 
    
    \draw[line width=0.8pt] (-3.5, -1) rectangle (-0.8, 1.0);
    \draw[->, thick, black!70!black] (-10,1.5) -- (-3,0.5);
    \node at (-10, 1.8) {\small$ |\Xi_{T-1}| = 4N_{T-1}$}; 
    
    \begin{scope}[shift={(-2.2,0.5)}]
        \draw[line width=0.6pt] (2.5,-0.5) rectangle (3.5,0.5);
    \end{scope}
    \draw[->, thick, black] (4.5,-1.5) -- (1,0.5);
    \node at (4.5,-2.0) [black] {\small  $|\Gamma_{T-1}| = N_{T-1}$};
    
    \begin{scope}[shift={(-1.8,-0.2)}]
        \draw[line width=0.6pt] (4.5,0.0) rectangle (6.0,1.5);
    \end{scope}
    \draw[->, thick, black!70!black] (7.5,-0.5) -- (3.75,0.75);
    \node at (7.5,-0.9) [black!70!black] {\small  $|\Gamma_{T}| = N_{T}$};
    
    \begin{scope}[shift={(-1.5,0)}]
        \draw[line width=0.6pt] (-1.25, -0.45) rectangle (-0.8, 0.00);
        \fill[pattern=north east lines, pattern color=black!80] 
             (-1.25, -0.45) rectangle (-0.8, 0.00);
    \end{scope}
    
    \node (sing) at (-10, -3) {singular sites $ \Pi$};
    
    \draw[->, thick] (sing) -- (-2.5, -0.225);

\end{tikzpicture}
\caption{The multiscale box reduction process. The large-scale box $\Lambda_N$ is organized into a resonant chain $\Xi_t$ (containing the singular site $\Pi$) coupled with a collection of non-resonant sub-boxes $\Gamma_t$. This hierarchical structure forms the basis for proving the localization of the inverse operator. }
\label{fig: reduce-size}
\end{figure}

With these preparations above, we now state the following theorem for the multi-scale analysis, which extends the results obtained for small boxes~(\Cref{thm: small-scale-tran}) to larger scales via induction.
\begin{theorem}
\label{thm: induction-multi-scale}
Assume that $\widehat{\pmb{z}} \in \mathcal{K}(s)$ and $\pmb{\omega} \in \Omega\setminus (\Omega_{M}^{\tau} \cup G)$. Suppose for any $M^0 \leq N\leq M^r$, the operator $(T + \varepsilon B)_{N}$ satisfies~\Cref{cond: implementation}. If the inversion condition~\eqref{eqn: inverse-grow} is satisfied for any $M^{r} < N \leq M^{r+1}$, then the localization condition
 \begin{equation}
 \label{eqn: ri-off-digonal-final}
          |(T + \varepsilon B)_{\Xi_{j+1}}^{-1}(\pmb{k}_1, \pmb{k}_2)|  \leq \exp\left\{- \frac{|\pmb{k}_1- \pmb{k}_2|^s}{2} \right\}, \quad \text{for}~|\pmb{k}_1 - \pmb{k}_2| > N_{j+1}^\frac{1}{2} 
\end{equation}
is also satisfied for the same range. 
\end{theorem}

Following the multi-scale reduction~(\Cref{lem: second-melnikov}) and the notations provided above, we now proceed with the proof. In direct correspondence with the small box analysis for the size range specified in~\eqref{eqn: small-size} (see~\Cref{subsec: translated-small}), the proof is organized into the following two steps.

\paragraph{Non-resonant gluing} For the base case ($i = 0$), we consider the inverse operator $(T+\varepsilon B)_{\Upsilon_{-1}}^{-1} =(T+\varepsilon B)_{\Xi_{0} \setminus \Pi}^{-1}$. The non-resonant index set $\Upsilon_{-1} = \Xi_{0} \setminus \Pi$ and its corresponding inverse were previously constructed in~\Cref{lem: small-size-parpare}. Given the small box size specified in the regime~\eqref{eqn: small-size}, the operator satisfies both the norm bound~\eqref{eqn: small-size-inverse-prepare} and the off-diagonal decay estimates~\eqref{eqn: small-size-off-diagnol-prepare}, thereby establishing the initial step.

We now invoke the inductive hypothesis for the step $i = j$. By construction, the non-resonant index sets $\Upsilon_{j-1}$ and the intermediate sub-boxes $\Gamma_j$ satisfy both the inversion condition~\eqref{eqn: trans-small-size-original-inverse} and the localization condition~\eqref{eqn: trans-small-size-off-diagnol}. To treat the larger set $\Upsilon_j$, we apply the resolvent identity in block form. By decomposing the set as $\Upsilon_j = \Gamma_j \cup (\Upsilon_j \setminus \Gamma_j)$, we obtain the following identity:
\begin{align}
(T + \varepsilon B)_{\Upsilon_{j}}^{-1} = & \begin{pmatrix} 
  (T + \varepsilon B)_{\Gamma_{j}}^{-1}    &       0                          \\
     0      &         (T + \varepsilon B)_{\Upsilon_{j}\setminus \Gamma_{j}}^{-1} 
\end{pmatrix} \nonumber \\ & -
\begin{pmatrix} 
 - (T + \varepsilon B)_{\Gamma_{j}}^{-1}    &       0                          \\
     0      &         (T + \varepsilon B)_{\Upsilon_{j}\setminus \Gamma_{j}}^{-1} 
\end{pmatrix}  \begin{pmatrix} 
                        0                          &          \varepsilon P^* \\
                         \varepsilon P      &        0
                         \end{pmatrix}   (T + \varepsilon B)_{\Upsilon_{j}}^{-1} \label{eqn: resolvent-identity}
\end{align}
where the operators $P$ and $P^*$ are coupling operators generated from $B + S$. In terms of matrix entries, this identity is expressed as:
\begin{align}
(T + \varepsilon B)_{\Upsilon_j }^{-1}(\pmb{k}_1, \pmb{k}_2) = & (T + \varepsilon B)_{\Gamma_j} ^{-1}(\pmb{k}_1, \pmb{k}_2) \nonumber \\
 &- \varepsilon \sum_{\pmb{k}_{3} \in \Gamma_j} \sum_{\pmb{k}_{4} \in \Upsilon_j \setminus \Gamma_j}  (T + \varepsilon B)_{\Gamma_j}^{-1}(\pmb{k}_1, \pmb{k}_3) P^*(\pmb{k}_3, \pmb{k}_4)(T + \varepsilon B)_{\Upsilon_j}^{-1} (\pmb{k}_4, \pmb{k}_2) \label{eqn: resolvent-identity-element}
\end{align}
where $\pmb{k}_1$ is the center of an intermediate sub-box $\Gamma_j$. Following~\Cref{prop: convolution} and~\Cref{prop: B-operator}, since the lattice box center under consideration lies outside the box $\Lambda_{2N}$, the coupling operators exhibit Gevrey decay:
\begin{equation}
\label{eqn: gevrey-p}
P(\pmb{k}, \pmb{k}') \leq \frac{2}{3} \exp \left\{ - |\pmb{k} - \pmb{k}'|^s \right\}, \quad \mathrm{and} \quad P^*(\pmb{k}, \pmb{k}') \leq \frac{2}{3} \exp \left\{- |\pmb{k} - \pmb{k}'|^s \right\}.
\end{equation}
By substituting the Gevrey decay~\eqref{eqn: gevrey-p} into the resolvent identity~\eqref{eqn: resolvent-identity-element}, we derive two iterative inequalities based on the distance between indices:
\begin{itemize}
\item For any $|\pmb{k}_1 - \pmb{k}_2| \geq 2N_j^{\frac{1}{2}}$, the resolvent identity~\eqref{eqn: resolvent-identity-element} reads:
         \begin{align}
         |(T + \varepsilon B)_{\Upsilon_j }^{-1}(\pmb{k}_1, \pmb{k}_2)|
         \leq &  \exp\left\{ - | \pmb{k}_1 - \pmb{k}_2|^s  \right\} \nonumber \\
          &+ \varepsilon \sum_{\pmb{k}_{4} \in \Upsilon_j \setminus \Gamma_j} \exp\left\{-\frac{|\pmb{k}_1 - \pmb{k}_4|^{s}}{2} \right\} |(T + \varepsilon B)_{\Upsilon_j}^{-1} (\pmb{k}_4, \pmb{k}_2)|   \label{eqn: ri-1-norm}
         \end{align}
\item For any $|\pmb{k}_1 - \pmb{k}_2| <2N_j^{\frac{1}{2}}$, , the resolvent identity~\eqref{eqn: resolvent-identity-element} reads:
         \begin{align}
         |(T + \varepsilon B)_{\Upsilon_j }^{-1}(\pmb{k}_1, \pmb{k}_2)|
         \leq & \exp\left\{ (\log 4N_j)^{15} \right\}  \nonumber \\
          &+ \varepsilon \sum_{\pmb{k}_{4} \in \Upsilon_j \setminus \Gamma_j} \exp\left\{-\frac{|\pmb{k}_1 - \pmb{k}_4|^{s}}{2} \right\} |(T + \varepsilon B)_{\Upsilon_j}^{-1} (\pmb{k}_4, \pmb{k}_2)|   \label{eqn: ri-2-norm}
         \end{align}
\end{itemize}         
By utilizing the iterative inequalities,~\eqref{eqn: ri-1-norm} and~\eqref{eqn: ri-2-norm}, we establish the following $L^1$ and $L^{\infty}$ bounds:
\begin{subequations}
\label{eqn: l1-linfty-ri}
\begin{empheq}[left=\empheqlbrace]{align} 
         &\sum_{\pmb{k}_1 \in \Upsilon_j} |(T + \varepsilon B)_{\Upsilon_j }^{-1}(\pmb{k}_1, \pmb{k}_2)|  \leq  \exp\left\{ (\log N_j)^{16} \right\} + \sqrt{\varepsilon}  \|(T + \varepsilon B)_{\Upsilon_j }^{-1} \|_{1}         \label{eqn: ri-1} \\
         &\sum_{\pmb{k}_2 \in \Upsilon_j} |(T + \varepsilon B)_{\Upsilon_j }^{-1}(\pmb{k}_1- \pmb{k}_2|  \leq  \exp\left\{ (\log N_j)^{16} \right\} + \sqrt{\varepsilon}  \|(T + \varepsilon B)_{\Upsilon_j }^{-1} \|_{\infty}      \label{eqn: ri-2}
 \end{empheq}        
 \end{subequations}
By taking the maximum of these bounds and applying the Schur test, we arrive at the operator norm estimate:
 \begin{equation}
 \label{eqn: ri-l2}
 \| (T + \varepsilon B)_{\Upsilon_j }^{-1}\|_2  \leq \sqrt{ \| (T + \varepsilon B)_{\Upsilon_j }^{-1} \|_1 \cdot \| (T + \varepsilon B)_{\Upsilon_j }^{-1}\|_{\infty}} \leq  \exp\left\{ (\log N_j)^{17} \right\}.
 \end{equation}     
 For any $|\pmb{k}_1 - \pmb{k}_2| > N_j$, the resolvent identity~\eqref{eqn: resolvent-identity-element} implies
 \[
 |(T + \varepsilon B)_{\Upsilon_j }^{-1}(\pmb{k}_1, \pmb{k}_2)| \leq  \frac{\varepsilon}{2} \sum_{\pmb{k}_{4} \in \Upsilon_j \setminus \Gamma_j} \exp\left\{-\frac{|\pmb{k}_1 - \pmb{k}_4|^{s}}{2} \right\} |(T + \varepsilon B)_{\Upsilon_j}^{-1} (\pmb{k}_4, \pmb{k}_2)|
 \]
 Given that the inversion bound~\eqref{eqn: ri-l2} holds, repeated application of this iteration yields the final localization bound:
 \begin{equation}
 \label{eqn: ri-off-digonal}
          |(T + \varepsilon B)_{\Upsilon_{j}}^{-1}(\pmb{k}_1, \pmb{k}_2)|  \leq  \frac12\exp\left\{- \frac{|\pmb{k}_1- \pmb{k}_2|^s}{2} \right\}. 
\end{equation}

 \paragraph{Inversion and localization} 
Finally, for the regime where $|\pmb{k}_1 - \pmb{k}_2| > N_j$, the resolvent identity leads to the following refinement:    
\begin{align*}
        |(T + \varepsilon B)_{\Xi_{j+1} }^{-1}(\pmb{k}_1, \pmb{k}_2)| 
        &\leq \frac12\exp\left\{- \frac{|\pmb{k}_1- \pmb{k}_2|^s}{2} \right\}   \\
        &+ \frac{\sqrt{\varepsilon}}{2} \sum_{\pmb{k}_4 \in \Xi_{j+1} \setminus \Upsilon_{j}} \exp \left\{- \frac{|\pmb{k}_1- \pmb{k}_4|^s}{2}\right\} |(T + \varepsilon B)_{\Xi_{j+1}}^{-1}(\pmb{k}_4, \pmb{k}_2)| 
\end{align*}
Incorporating the inversion bound \eqref{eqn: trans-small-size-original-inverse} to control the diagonal term and repeating application of the above iteration, we derive the final localization bound~\eqref{eqn: ri-off-digonal-final}, which completes the proof of~\Cref{thm: induction-multi-scale}.

\subsection{From small to large boxes: exclusion of resonant frequencies across scales}
\label{subsec: small-large}

In this section, we establish the mechanism for excluding the ``bad'' frequencies (the resonant set) to ensure the validity of the inverse estimates required for~\Cref{thm: small-scale-tran} and~\Cref{thm: induction-multi-scale}. The strategy for measure exclusion shifts depending on the box size regime.

\paragraph{Small box regime} For small boxes within regime~\eqref{eqn: small-size}, the perturbation parameter $\varepsilon$ remains smaller than $\varepsilon_N$. In this context, a Neumann series approach, consistent with~\Cref{thm: small-size-original}, is sufficient to maintain control over the operator. To ensure invertibility, we explicitly characterize the excluded frequencies via the single-mode resonant set:
\begin{equation}
\label{eqn: single-mode-resonance-small}
\Omega_{1, N} = \left\{ \pmb{\omega} \in \Omega \big | |\sigma_{\star} + \omega_j| = | - \langle \pmb{k}_{\star}, \pmb{\omega}' \rangle + \omega_j| < \varepsilon_{N}, \; \forall j=1,\ldots, n \right\},
\end{equation}
It follows by direct calculation that the measure of this set satisfies:
\begin{equation}
\label{eqn: single-mode-resonance-small-measure}
\mathrm{mes}(\Omega_{1, N, \pmb{k}_0} ) \leq n^2(2N+1)^{n-1}\varepsilon_{N}.
\end{equation}

\paragraph{Large box regime} For large boxes defined in regime~\eqref{eqn: large-size}, we analyze the norm of the restricted operator, specifically its smallest singular value. We represent the operator $(T + \varepsilon B)_{\pmb{k}_0, N}$ in a block structure partitioned by the singular site $\Pi$:
\begin{equation}
\label{eqn: large-box-operator-inverse-dig-1}
(T+\varepsilon B)_{\pmb{k}_0, N} = \begin{pmatrix} (T+\varepsilon B)_{\Xi_{J+1} \setminus\Pi} & P_1^* \\ P_1 & (T+\varepsilon B)_{\Pi} \end{pmatrix}.
\end{equation}
To analyze this structure, we introduce the Schur complement of $(T+\varepsilon B)_{\pmb{k}_0, N}$ with respect to the singular site $\Pi$ (see~\citet[Section 3.2.11]{golub2013matrix})
\begin{equation}
\label{eqn: schur-complement}
U = (T+\varepsilon B)_{\Pi}  - P_1(T+\varepsilon B)_{\Xi_{J+1} \setminus\Pi}^{-1} P_1^*
\end{equation}
where the operator restricted to the singular site is expressed as:
\begin{equation}
\label{eqn: operator-singular-set}
 (T+\varepsilon B)_{\Pi} = \sigma_{\star} + \pmb{\omega} + \varepsilon (S + B)(\pmb{k}_{\star}, \pmb{k}_{\star})
\end{equation}
Since $ \| P_1^* \| $ and $ \| P_1 \| $ are both bounded by $ \varepsilon $, we can control
the inverse of $ (T + \varepsilon B)_{\pmb{k}_0, N} $ as
\begin{equation}
    \label{eqn: inverse-control}
    \left\| (T + \varepsilon B)_{\pmb{k}_0, N}^{-1} \right\|_2 \leq 4 \left\| (T + \varepsilon B)_{\Xi_{J+1} \setminus\Pi}^{-1} \right\|_2^2 \cdot \left\| U^{-1} \right\|_2 + \left\| (T + \varepsilon B)_{\Xi_{J+1} \setminus\Pi}^{-1} \right\|_2.
\end{equation}
If we can control $ \| U^{-1} \|_2 $ by $ 1 / \sqrt{\varepsilon_N} $, then we have together
with \eqref{eqn: ri-l2} that 
\begin{equation*}
    \left\| (T + \varepsilon B)_{\pmb{k}_0, N}^{-1} \right\|_2 \leq \frac{1}{\varepsilon_N}.
\end{equation*}
Therefore, we only need to focus on the Schur complement $ U $. By substituting~\eqref{eqn: operator-singular-set} into~\eqref{eqn: schur-complement}, we reformulate the Schur complement as $U = \sigma_{\star} + \Phi$, where the perturbation term collects the remaining terms:
\begin{equation}
\label{eqn: phi}
\Phi =  \pmb{\omega} + \varepsilon (S + B)(\pmb{k}_{\star}, \pmb{k}_{\star}) -  P_1(T+\varepsilon B)_{\Xi_{J+1} \setminus\Pi}^{-1} P_1^*.
\end{equation}
Given the estimate $\| \Phi \|_2 \leq 2 \exp\left\{ (\log N)^{17/10} \right\}$ as shown in~\eqref{eqn: ri-l2}, and applying the triangle inequality $\|\sigma_{\star}  + \Phi\|_2 \geq \left| |\sigma_{\star}| - \| \Phi \|_2 \right|$, we define the single-mode resonant set for the large box as
\begin{equation}
\label{eqn: single-mode-resonance-small-origin}
\Omega_{1, N} = \left\{ \pmb{\omega} \in \Omega \big | \left| |\sigma_{\star}| - \| \Phi \|_2 \right| < \varepsilon_{N}, \; \forall j=1,\ldots, n \right\}.
\end{equation}

The exclusion of measure for the large box does not possess the simple structure found in the small box regime. Instead of approximating polynomials via the Malgrange Preparation Theorem and excluding measure through derivative estimates (as in~\citep{bourgain1998quasi}), we adopt a strategy based on the so-called Cartan estimates for analytic functions, as introduced by~\citet{bourgain2005green}. For our purposes, we utilize a simplified version from~\citet{levin1996lectures}.

\begin{lemma}[Theorem 4 of Lecture 11.3 in~\citet{levin1996lectures}]
\label{lem: cartan}
Let $f(x)$ be a function analytic in the disk $\{ z: |z| \leq 2eR \}$, $|f(0)| = 1$, and let $\eta$ be an arbitrary small positive number. Then the estimate
\[
\log |f(z)| > - H(\eta) \log M_f(2eR), \qquad H(\eta) = \log \frac{15e^3}{\eta},
\]
is valid everywhere in the disk $\{z: |z| \leq R \}$ except a set disks $(C_j)$ with sum of radii
\[
\sum r_{j} \leq \eta R.
\]
\end{lemma}

To complete the measure estimates, we consider the magnitude of the perturbation $\Phi$ in two separate regimes. Under the condition $\| \Phi \|_2 < \sqrt{ \varepsilon_{N}}$, the resonant set is contained within:
\begin{equation}
\label{eqn: small-exclude-large}
\Omega_{1, N} \subseteq \left\{ \pmb{\omega} \in \Omega \big | \left| \sigma_{\star} \right| <2 \sqrt{\varepsilon_{N}}, \; \forall j=1,\ldots, n \right\}. 
\end{equation}
In this regime, the measure estimate follows directly from the volume of the resonant slabs:
\begin{equation}
\label{eqn: single-mode-resonance-small-measure-large-1}
\mathrm{mes}(\Omega_{1, N, \pmb{k}_0} ) \leq 2n^2(2N+1)^{n-1} \sqrt{\varepsilon_{N}}.
\end{equation}
When $\| \Phi \|_2 \geq  \sqrt{\varepsilon_{N}}$, we define the analytic function as
\begin{equation}
\label{eqn: analytic-cartan}
f(\sigma_{\star}) = \frac{|\sigma_{\star}|}{\| \Phi \|_2} - 1
\end{equation}
with the disk parameter as $\| \Phi \|_2 = 2e(1+e)^{-1}R$. It follows that $|f(0)| = 1$ and $M_f(2eR) =e$. By setting $\eta = 15e^3\sqrt{ \varepsilon_N}$,~\Cref{lem: cartan} implies the following measure estimate as
\begin{equation}
\label{eqn: single-mode-resonance-small-measure-large-2}
\mathrm{mes}(\Omega_{1, N, \pmb{k}_0}) \leq \mathrm{mes} \left( \left\{ \pmb{\omega} \in \Omega \big | \left| f(\sigma_{\star}) \right| < \sqrt{\varepsilon_{N}} \right\} \right) \leq 8(1+e)e^2 \exp\left\{ (\log N)^{\frac{17}{10}} \right\} \sqrt{ \varepsilon_N }. 
\end{equation}
Combing the measure estimates~\eqref{eqn: single-mode-resonance-small-measure},~\eqref{eqn: single-mode-resonance-small-measure-large-1}, and~\eqref{eqn: single-mode-resonance-small-measure-large-2}, we establish the measure for the ``bad'' frequencies as below. 
\begin{theorem}
\label{thm: multi-scale-final}
For any $N \geq M_0$  and any box of size $N$ with center $\pmb{k}_0 \notin \Lambda_{2N}$, there exist a threshold $\varepsilon_0 > 0$ such that for all $\varepsilon \in (0, \varepsilon_0]$, if the operator $T + \varepsilon B$ restricted to the box $\pmb{k}_0 + \Lambda_{N}$ satisfy the inversion condition \eqref{eqn: trans-small-size-original-inverse}, then the measure of the ``bad'' frequencies set satisfies
\begin{equation}
\label{eqn: exclude-first-melnikov}
\mathrm{mes}(\Omega_{1,N, \pmb{k}_0}) \leq \varepsilon_{N}^{\frac14}.
\end{equation}
\end{theorem}

\section{Proof of the main theorem}
\label{sec: completion}

In this section, we present the formal proof of the main statement. We begin in~\Cref{cond: implementation} by verifying~\Cref{subsec: verification} via the pave method and extending these results to the derivatives. Subsequently,~\Cref{subsec: induction} provides the inductive estimate. Finally, we apply the iteration lemma to complete the proof.

\subsection{Verification of~\Cref{cond: implementation}}
\label{subsec: verification}

We proceed by induction on the size $N$. Following~\Cref{thm: small-size-original} as the base case, which satisfies~\Cref{cond: implementation}, we now consider the inverse operator $T_{N}^{-1}$ for $M^{r} < N \leq M^{r+1}$. Assume the result holds for $M_0 \leq N \leq M^r$, we set the parameters $K = N^{1/10}$ and consider a central box with size $10K$. Let $\Gamma_{\alpha} = \pmb{k}_{0, \alpha} + \Lambda_K$ to be a collection of boxes such that $\mathbf{k}_{0, \alpha} \notin \Lambda_{5K}$. The set $\Lambda_{N}$ can then be decomposed as: 
\begin{equation}
\label{eqn: box-seperation}
\Lambda_{N} = \Lambda_{10K} \bigcup \bigg( \bigcup _{\alpha} \Gamma_{\alpha} \bigg)
\end{equation}
Since the total number of the lattice points in $\Gamma_{\alpha} $ is at most $(2K^{10}+1)^n$, we apply~\Cref{prop: tensor} and~\Cref{thm: multi-scale-final} to estaimte the single-mode nearly-resonant set.  Specifically, by summing the measures of the ``bad'' frequencies associated with each localized box at scale $K$, we obtain: 
\begin{equation}
\label{eqn: 1-melnikov-k}
\mathrm{mes}(\Omega_{1, K}) \leq \sum_{\alpha} \mathrm{mes}(\Omega_{1, K, \pmb{k}_{0, \alpha}}) \leq (2K^{10}+1)^n \varepsilon_{K}^{\frac14}. 
\end{equation}
Thus, the total measure for the single-mode nearly-resonant set can be estimated by:
\begin{equation}
\label{eqn: 1-melnikov-total}
\mathrm{mes}(\Omega_1) \leq \sum_{K = M_0}^{\infty} (2K^{10}+1)^n \exp\left\{ \frac{(\log K)^{15}}{4} \right\}. 
\end{equation}
From~\Cref{defn: second-melnikov}, we can derive the total measure of the difference nearly-resonant set as:
\begin{equation}
\label{eqn: 2-melnikov-total}
\mathrm{mes}(\Omega_2) \leq \sum_{K=M_0}(2K^{10}+1)^n \exp\left\{ (\log K)^{\frac{3}{2}} \right\}.  
\end{equation}
Combining~\eqref{eqn: 1-melnikov-total} and~\eqref{eqn: 2-melnikov-total}, we can choose the threshold $\varepsilon_0 > 0$ such that the total measure of ``bad'' frequencies, $\mathrm{mes}(\Omega_1 \cup \Omega_{2})$, is sufficiently small to satisfy the desired requirement.

Finally, we employ the resolvent identity to ``glue" the inverse operators of the small boxes to the large box $\lambda_N$, thereby establishing the bounded inverse and the off-diagonal decay of its entries. For the central box $\Lambda_{10K}$, the resolent indentity reads
\begin{equation}
\label{eqn: 10k-box}
|T_{N}^{-1}(\pmb{k}_1, \pmb{k}_2)|  \leq |T_{ \Lambda_{10K}}^{-1}(\pmb{k}_1, \pmb{k}_2)| + \sum_{\pmb{k}_3 \in \Lambda_{10K}}\sum_{\pmb{k}_4 \in \Lambda_{N} \setminus  \Lambda_{10K}} |T_{ \Lambda_{10K}}^{-1}(\pmb{k}_1, \pmb{k}_3)| |T(\pmb{k}_3, \pmb{k}_4)| |T_{ \Lambda_{10K}}^{-1}(\pmb{k}_4, \pmb{k}_2)|,  
\end{equation}
where $\pmb{k}_4 \in \Lambda_{N}/\Lambda_{10K}$. Although $|T_{ \Lambda_{10K}}^{-1}(\pmb{k}_1, \pmb{k}_3)|$ must satisfy the two-side off-digonal decay,  $T(\pmb{k}_3, \pmb{k}_4)$ satisfies  the single-side off-diagnol decay as~\Cref{thm: small-scale-tran} and~\Cref{thm: induction-multi-scale}. For the intermediate boxs $\Gamma_{\alpha}$, the resolvent identity reads
\begin{equation}
\label{eqn: k-box}
|T_{N}^{-1}(\pmb{k}_1, \pmb{k}_2)|  \leq |T_{ \Gamma_{\alpha}}^{-1}(\pmb{k}_1, \pmb{k}_2)| + \sum_{\pmb{k}_3 \in \Gamma_{\alpha}}\sum_{\pmb{k}_4 \in \Lambda_{N} \setminus  \Gamma_{\alpha}}\ |T_{ \Gamma_{\alpha}}^{-1}(\pmb{k}_1, \pmb{k}_3)| |T(\pmb{k_3}, \pmb{k}_4)| |T_{ \Gamma_{\alpha}}^{-1}(\pmb{k}_4, \pmb{k}_2)|,  
\end{equation}
where $\pmb{k}_4 \in \Lambda_{10K}$. In this case, $T_{\Gamma_{\alpha}}^{-1}(\pmb{k_3}, \pmb{k}_4)$ satisfies the single-side off-diginal decay. The remaining steps for verifying~\Cref {cond: implementation} involve standard inequality manipulations, which are deferred to~\Cref{subsec: implement-left}.

\subsection{The iteration lemma}
\label{subsec: induction}

Before proceeding to the inductive estimates, we must characterize the derivative of the inverse operator $\partial(T + \epsilon B)_N^{-1}$, specifically regarding its norm bounds and off-diagonal decay properties. Using the standard operator identity $(T + \epsilon B)_N^{-1} (T + \epsilon B)_N = I$, the derivative of the inverse is given by:
\begin{equation}
\label{eqn: t-n-derivative-inverse}
\partial (T + \varepsilon B)_{N}^{-1} = - (T + \varepsilon B)_{N}^{-1}(\partial (T + \varepsilon B)_{N} ) (T + \varepsilon B)_{N}^{-1}.
\end{equation}
This relationship allows us to establish the following lemma regarding the operator's composite norm and spatial localization.
\begin{lemma}
\label{lem: inverse-derivative}
If the restricted operator $(T + \varepsilon B)_{N}$ satisfies~\Cref{cond: implementation}, then its inverse satisfies the norm bound:
\begin{equation}
\label{eqn: inverse-restricted-norm}
\| \partial (T + \varepsilon B)_{N}^{-1} \|_2 \leq \frac{4N}{\varepsilon_N^2}. 
\end{equation}
Furthermore, for sufficiently large spatial separations in the lattice $\Lambda_N$, specifically when $|\pmb{k} \pm \pmb{k}'| \geq N^{\frac34}$ the entries of the inverse matrix must exhibit Gevrey-type decay:
\begin{equation}
\label{eqn: inverse-restricted-entries}
| \partial (T + \varepsilon B)_{N}^{-1}\left( \pmb{k}, \pmb{k}' \right) | \leq \exp\left\{ - \frac{|\pmb{k} - \pmb{k}'|^{s}}{4}\right\} +  \exp\left\{ - \frac{|\pmb{k} + \pmb{k}'|^{s}}{4}\right\}.
\end{equation}
\end{lemma}
The inversion result~\eqref{eqn: inverse-restricted-norm} follows directly from the inversion condition~\eqref{eqn: inverse-grow} and the derivative identity~\eqref{eqn: t-n-derivative-inverse}. The localization result~\eqref{eqn: inverse-restricted-entries} involves basic inequality derivations regarding Gevrey classes; the full technical details are deferred to~\Cref{subsec: derivative-app}. With the verification of~\Cref{cond: implementation} and~\Cref{lem: inverse-derivative}, we establish the following rigorous iteration lemma.  For technical convenience, we adopt the convention that $N_{-1} = 0$ and $\Lambda_{N_{-1}} = \varnothing$.

\begin{theorem}[The Iteration Lemma]
\label{lem: iteration}
Assume that the frequencies satisfy $\pmb{\omega} \in \Omega \setminus (\Omega_{M}^{\tau} \cup \Omega_1 \cup \Omega_2)$. Let $\varepsilon_0 = \varepsilon_0(H_1, \Omega, \pmb{a}) > 0$ be the threshold provided in~\Cref{thm: main}. For any $0 < \varepsilon \leq \varepsilon_0$, there exists a Gevrey exponent $s(\varepsilon) > 0$ such that, at the $r$-th iteration,  the approximate solution $ \widehat{\pmb{z}}^{(r)}$ is supported on:  
\begin{equation}
\label{eqn: vector-support-r}
\mathrm{supp} \;  \widehat{\pmb{z}}^{(r)} \subseteq \Lambda_{N_r} = \Lambda_{M^r}. 
\end{equation}
 Moreover, for any index $m= \{ 0, 1, \ldots, r\}$ and any $\pmb{k}  \in \Lambda_{M^m} \setminus \Lambda_{M^{m-1}}$, assume the following Gevrey decay property holds: 
\begin{equation}
  \label{eqn: vector-decay-r}   
 \|  \widehat{\pmb{z}}^{(r)}(\pmb{k} ) \| \leq \sum_{\ell=m}^{r}\exp\left\{ - \frac32 N_{\ell}^{s} \right\},
\end{equation}
 the vector field $F$ satisfies: 
\begin{equation}
 \label{eqn: vector-field-decay-r}  
\big\| F\big( \widehat{\pmb{z}}^{(r)}_p; \;\pmb{\omega}^{(r+1)}, \pmb{a} \big) \big\| \leq  \exp\left\{-2N_{r}^{s}\right\},                                                                                                                                                                                                                                                                   
\end{equation}
Then, after updating the frequency via $Q$-equations~\eqref{eqn: q-eqn-numerical} and the coefficients via the dimension-enlarged Newton scheme~\eqref{eqn: p-eqn-numerical}, these properties are preserved at the $(r+1)$-th iteration. Specifically, the update iterate $ \widehat{\pmb{z}}^{(r+1)} $ satisfies: 
\begin{equation}
\label{eqn: vector-support-r+1}
\mathrm{supp} \;  \widehat{\pmb{z}}^{(r+1)} \subseteq [-N_{r+1}, N_{r+1}]^n. 
\end{equation}
Furthermore, for any $m \in \{ 0, 1, \ldots, r+1\}$ and any $\pmb{k}  \in \Lambda_{M^m} \setminus \Lambda_{M^{m-1}}$, the following decay property holds: 
\begin{equation}
  \label{eqn: vector-decay-r+1}   
   \|  \widehat{\pmb{z}}^{(r+1)}(\pmb{k} ) \| \leq \sum_{\ell=m}^{r+1}\exp\left\{ - \frac32 N_{\ell}^s\right\},
\end{equation}
 the vector field $F$ satisfies: 
\begin{equation}
 \label{eqn: vector-field-decay-r+1}  
 \big\| F\big( \widehat{\pmb{z}}_p^{(r+1)}; \;\pmb{\omega}^{(r+2)}, \pmb{a} \big) \big\| \leq  \exp\left\{-2N_{r+1}^{s}\right\}.                                                                                                                                                                                                                                                                 
\end{equation}
\end{theorem}

\begin{proof}[Proof of~\Cref{lem: iteration}]

For conciseness, we suppress the auxiliary parameters from the dimension-enlarged Newton scheme~\eqref{eqn: p-eqn-numerical} and define the iterative update as:
\begin{equation}
\label{eqn: nash-moser-single-simplification}
\Delta^{(r)}  = \widehat{\pmb{z}}^{(r+1)}_p - \widehat{\pmb{z}}^{(r)}_p  = - (T + \varepsilon B)_{N_{r+1}}^{-1}  \big(  \widehat{\pmb{z}}^{(r)}_p \big) F\big(  \widehat{\pmb{z}}^{(r)}_p \big).
\end{equation}
The proof is organized in the following three steps. 

\paragraph{Support enlargement} This step verifies the inclusion property~\eqref{eqn: vector-support-r+1}. Given the perturbation Hamiltonian $H_1$ is a polynomial with the total degree at most $d$, the vector field is located in the support as
\begin{equation}
\label{eqn: support-d-enlarge}
\mathrm{supp} \; F\big(  \widehat{\pmb{z}}^{(r)}_p \big) \subseteq [-dN_{r}, dN_{r}]^n \subseteq \left[ -\frac{N_{r+1}}{4}, \frac{N_{r+1}}{4} \right]^n,
\end{equation}
which confirms~\eqref{eqn: vector-support-r+1} in accordance with the update rule defined in~\eqref{eqn: nash-moser-single-simplification}. 

\paragraph{Lattice vector decay} This stepestablishes the bound in~\eqref{eqn: vector-decay-r+1}. Given the inversion condition~\eqref{eqn: inverse-grow} and the previous vector field estimate~\eqref{eqn: vector-field-decay-r}, we bound the iterative difference as:
\[
\| \Delta^{(r)} \|_2 \leq \big \| (T+ \varepsilon B)_{N_{r+1}}^{-1}  \big(  \widehat{\pmb{z}}^{(r)}_p  \big) \big\|_2  \big \| F\big(  \widehat{\pmb{z}}^{(r)}_p\big) \big\|_2 \leq \frac{1}{\varepsilon_{N_{r+1}}}  \exp\left\{-2N_{r}^{s}\right\}.
\]
Furthermore, by applying~\Cref{lem: inverse-derivative}, we estimate the derivative of the update. The derivative satisfies:
\begin{align*}
\| \partial \Delta^{(r)} \|_2 & \leq \big\| \partial (T+ \varepsilon B)_{N_{r+1}}^{-1}  \big(  \widehat{\pmb{z}}^{(r)}_p  \big) \big \|_2 \big\| F\big(  \widehat{\pmb{z}}^{(r)}_p\big) \big\|_2 + \big \| (T+ \varepsilon B)_{N_{r+1}}^{-1}  \big(  \widehat{\pmb{z}}^{(r)}_p  \big) \big \| _2\big\| \partial F\big(  \widehat{\pmb{z}}^{(r)}_p\big) \big \|_2                                                            \nonumber \\
                                        & \leq \left(  \frac{4n\sqrt{N_{r+1}} + \varepsilon_{N_{r+1}}}{\varepsilon_{N_{r+1}}^2}\right) \exp\left\{-2N_{r}^{s}\right\}.       
\end{align*}                                       
By combining the two bounds above, we can establish a threshold $\varepsilon_0 > 0$ such that 
\begin{equation}
\label{eqn: difference-derivative-estimate}
\| \Delta^{(r)}  \| \leq \exp\left\{-\frac{7N_{r}^{s}}{4}\right\}
\end{equation}
for any $0<\varepsilon \leq \varepsilon_0$. Now, choosing $s = s(\varepsilon)>0$ such that $6M^{s} \leq 7$. Then, at the $(r+1)$-th iteration, for any index $m= \{0, 1, \ldots, r+1\}$ and any $\pmb{k}  \in  \Lambda_{M^m} \setminus \Lambda_{M^{m-1}}$, the following holds:
\begin{align*}
 \| \widehat{\pmb{z}}^{(r+1)}(\pmb{k}) \|  \leq \ \| \widehat{\pmb{z}}^{(r)}(\pmb{k}) \|+  \|\Delta^{(r)}(\pmb{k})\|  \leq  \sum_{\ell=m}^{r+1}\exp\left\{ - \frac32 N_{\ell}^{s} \right\}
\end{align*}
which confirms the decay property~\eqref{eqn: vector-decay-r+1} across successive iterations.

\paragraph{Super-exponentical convergence} To establish the super-convergence property in~\eqref{eqn: vector-field-decay-r+1}, we analyze the vector field after the Newton update~\eqref{eqn: nash-moser-single-simplification}. By applying a Taylor expansion to the vector field and the frequency, and invoking~\Cref{prop: tensor}, the post-update residual can be bounded by the truncation error and the quadratic terms of the iteration: 
\begin{equation}
\label{eqn: vector-estimate-taylor-r+1}
\| F(\widehat{\pmb{z}}_p^{(r+1)}; \pmb{\omega}^{(r+2)}) \|_2 \leq \| [ (T + \varepsilon B) - (T + \varepsilon B)_{N_{r+1}}] \Delta^{(r)}  \|_2 + \exp\left\{- 3N_{r+1}^{s} \right\},
\end{equation}
while the derivative of the residual satisfies the following bound:
\begin{align}
\| \partial F(\widehat{\pmb{z}}_p^{(r+1)}; \pmb{\omega}^{(r+2)}) \|_2 \leq &  \| \partial [ (T + \varepsilon B) - (T + \varepsilon B)_{N_{r+1}}] \Delta^{(r)}  \|_2  \nonumber \\ &+  \| [ (T + \varepsilon B) - (T + \varepsilon B)_{N_{r+1}} ]\partial \Delta^{(r)}  \|_2+  3\exp\left\{- 3N_{r+1}^{s} \right\}.     \label{eqn: vector-estimate-derivative-taylor-r+1}
\end{align}
To estimate the Gevery norms in~\eqref{eqn: vector-estimate-taylor-r+1} and~\eqref{eqn: vector-estimate-derivative-taylor-r+1}, we exploit the off-diagonal decay of the operators $T+ \varepsilon B $ and $(T+ \varepsilon B)_{N_{r+1}}^{-1}$ established in~\Cref{prop: convolution} and~\Cref{lem: inverse-derivative}. Since the estimation procedures for these terms are analogous, we focus on $[(T+ \varepsilon B)- (T+ \varepsilon B)_{N_{r+1}}] \Delta^{(r)}$ as the representative example. We utilize the following decomposition:
\begin{align}
[(T+ \varepsilon B)- (T+ \varepsilon B)_{N_{r+1}}] \Delta^{(r)} = & \left[  (I - P_{N_{r+1}}) (T+ \varepsilon B) P_{N_{r+1}}  \right]\left[ (I - P_{N_{r+1}/3})  \Delta^{(r)} \right] \nonumber \\ &+ \left[ (I - P_{N_{r+1}}) (T+ \varepsilon B)  P_{N_{r+1}/3} \right] \Delta^{(r)}.  \label{eqn: decompostion-t-tr}
\end{align}
Given the confirmed off-diagonal decay, there exists a threshold $\varepsilon_0 > 0$ such that for $0 < \varepsilon \leq \varepsilon_0$, the following estimates holds:
\begin{equation}
\label{eqn: final-super-convergence-bound-1}
\|(I - P_{N_{r+1}}) (T+ \varepsilon B)  P_{N_{r+1}/3} \| \leq \frac13\exp\left\{ - \frac{7}{12} N_{r+1}^{s} \right\}
\end{equation}
and
\begin{equation}
\label{eqn: final-super-convergence-bound-2}
\|(I - P_{N_{r+1}/3}) (T+ \varepsilon B) _{N_{r+1}}^{-1} P_{N_{r+1}/4} \| \leq \frac13\exp\left\{ - \frac{1}{48} N_{r+1}^{s} \right\}
\end{equation}
Noting that $\|(I - P_{N_{r+1}}) T P_{N_{r+1}} \| \leq \|S\| \leq 1$, we combine the bounds~\eqref{eqn: final-super-convergence-bound-1} and~\eqref{eqn: final-super-convergence-bound-2} with the decomposition in~\eqref{eqn: decompostion-t-tr} to account for the truncation error from~\eqref{eqn: vector-estimate-taylor-r+1} and~\eqref{eqn: vector-estimate-derivative-taylor-r+1}. For a suitably chosen threshold $\varepsilon_0 > 0$, we obtain
\begin{equation}
\label{eqn: final-super-convergence-1}
\| F(\widehat{\pmb{z}}_p^{(r+1)}; \pmb{\omega}^{(r+2)}) \| \leq \frac{2}{3} \exp\left\{ -\left( \frac{2}{M^{s}}  + \frac{1}{48}\right) N_{r+1}^{s} \right\}  +  4\exp\left\{- 3N_{r+1}^{s} \right\}
\end{equation}
By choosing $s = s(\varepsilon)>0$ such that $95 M^{s} \leq 96$, we arrive at the desired super-convergence bound:
\begin{equation}
\label{eqn: final-super-convergence-2}
\| F(\widehat{\pmb{z}}_p^{(r+1)}; \pmb{\omega}^{(r+2)}) \|_s \leq \ \exp\left\{ - 2  N_{r+1}^{s} \right\}, 
\end{equation}
which concludes the proof of the Iteration Lemma~(\Cref{lem: iteration}). 
\end{proof}

\begin{remark}
\label{rem: sum-invert}
Regarding the decay estimates for lattice vectors specified in \eqref{eqn: vector-decay-r} and \eqref{eqn: vector-decay-r+1} of Lemma \ref{lem: iteration}: while the individual estimates might appear to grow intermittently upon initial inspection, the cumulative decay remains well-controlled. Specifically, for any multi-index $\pmb{k} \in  \Lambda_{N_r} \setminus \Lambda_{N_{r-1}}$, the following holds: 
\[
\sum_{\ell = r}^{\infty} \exp\left\{-\frac{3N_{\ell}^{s}}{2}\right\} \leq \exp\left\{-N_{r}^{s} \right\} \leq \exp\left\{ - |\pmb{k}|^{s} \right\},
\]
which demonstrates that at each iteration $r$, the sequence $\widehat{\pmb{z}}^{(r)}$ remains within the space $\mathcal{K}(s)$.  Consequently, the solution preserves the $s$-Gevrey decay property throughout the iterative process, ensuring that the spatial localization does not degrade as the lattice size $N$ increases.\end{remark}

\paragraph{Proof of~\Cref{thm: main}} First, we observe that \eqref{eqn: difference-derivative-estimate} holds when $ \widehat{\pmb{z}}_p $ is replaced by $ \widehat{\pmb{z}} $
as $ \widehat{\pmb{z}}_q^{(r)} $ remains constant at each step. The displacement between successive iterations is bounded as:

\begin{equation*}
  \| \widehat{\pmb{z}}^{(r+1)} - \widehat{\pmb{z}}^{(r)} \| = \| \Delta^{(r)} \| \leq \exp \left\{ - \frac{7}{4} N_{r}^s \right\} = \exp \left\{ - \frac{7}{4} (M^s)^r \right\}.
\end{equation*}

\noindent This implies the convergence of the sequences $ \left\{ \widehat{\pmb{z}}^{(r)} \right\} $. Furthermore, we can bound the distance to the limit $\mathbf{z}^\star$ by estimating the tail of the sequence:

\begin{equation*}  
  \begin{aligned}
    \| \widehat{\pmb{z}}^{(r)} - \widehat{\pmb{z}}^\star \| 
    &\leq \sum_{l = 0}^{\infty} \| \widehat{\pmb{z}}^{(r+l+1)} - \widehat{\pmb{z}}^{(r+l)} \| \leq \sum_{l = 0}^{\infty} \exp \left\{ - \frac{7}{4} (M^s)^{r+l} \right\}  \\
    &\leq \exp \left\{ - \frac{3}{2} (M^s)^{r} \right\},
  \end{aligned}
\end{equation*}

\noindent which confirms \eqref{eqn: z-hat-converge}. Moreover, from $Q$-equations \eqref{eqn: q-eqn-numerical} and the uniform bound \eqref{eqn: hessian-bound}, we have:

\begin{equation*}
  | \pmb{\omega}^{(r+1)} - \pmb{\omega}^{(r)} | \leq \varepsilon \gamma \| \widehat{\pmb{z}}^{(r+1)} - \widehat{\pmb{z}}^{(r)} \| \leq \varepsilon \gamma \exp \left\{ - \frac{7}{4} (M^s)^r \right\}.
\end{equation*}

\noindent Thus, the convergence of the frequency sequences $ \left\{ \pmb{\omega}^{(r)} \right\} $ follows, and we also have

\begin{equation*}  
  \begin{aligned}
    | \pmb{\omega}^{(r)} - \pmb{\omega}^\star | 
    &\leq \sum_{l = 0}^{\infty} | \pmb{\omega}^{(r+l+1)} - \pmb{\omega}^{(r+l)} | \leq \sum_{l = 0}^{\infty} \varepsilon \gamma \exp \left\{ - \frac{7}{4} (M^s)^{r+l} \right\}  \\
    &\leq \exp \left\{ - \frac{3}{2} (M^s)^{r} \right\},
  \end{aligned}
\end{equation*}

\noindent which confirms \eqref{eqn: frequency-hat-converge}.

Finally, we write the exact solution in a Fourier expansion 

\begin{equation*}
  \pmb{z}^\star(t) = \sum_{\pmb{k} \in \mathbb{Z}^n} \widehat{\pmb{z}}^\star(\pmb{k}) e^{i \langle \pmb{\omega}^\star, \pmb{k} \rangle t},
\end{equation*}

\noindent From the support property \eqref{eqn: vector-support-r}, the error can be split into two terms:

\begin{equation*}
  \begin{aligned}
    \| \pmb{z}^{(r)}(t) - \pmb{z}^\star(t) \| &\leq \left\| \sum_{\pmb{k} \in \mathbb{Z}^n} \left( \widehat{\pmb{z}}^\star(\pmb{k}) - \widehat{\pmb{z}}^{(r)}(\pmb{k}) \right) e^{i \langle \pmb{\omega}^\star, \pmb{k} \rangle t} \right\|
      + \left\| \sum_{\pmb{k} \in \Lambda_{M^{r+1}}} \widehat{\pmb{z}}^{(r)}(\pmb{k}) \left( e^{i \langle \pmb{\omega}^\star, \pmb{k} \rangle t} - e^{i \langle \pmb{\omega}^{(r)}, \pmb{k} \rangle t} \right) \right\|  \\
      & \defi I_1 + I_2.
  \end{aligned}
\end{equation*}

\noindent By \eqref{eqn: vector-support-r}, Gevrey decay property \eqref{eqn: vector-decay-r}, the Cauchy-Schwarz inequality, 
and \eqref{eqn: z-hat-converge} just proved above, we estimate the term $ I_1 $ as

\begin{equation*}
  \begin{aligned}
      I_1 &\leq \sum_{\pmb{k} \in \Lambda_{M^{r+1}}} \| \widehat{\pmb{z}}^\star(\pmb{k}) - \widehat{\pmb{z}}^{(r)}(\pmb{k}) \| + \sum_{\pmb{k} \notin \Lambda_{M^{r+1}}} \| \widehat{\pmb{z}}^\star(\pmb{k}) \|  \\
          &\leq \| \widehat{\pmb{z}}^{(r)} - \widehat{\pmb{z}}^\star \| \cdot (2M^{r+1} + 1)^{n/2} + \sum_{\pmb{k} \notin \Lambda_{M^{r+1}}} e^{-\frac{5}{4} |k|^s}  \\
          &\leq \frac{1}{2} \exp \left\{ - (M^s)^{r} \right\}.
  \end{aligned}
\end{equation*}

From the uniform boundedness of $ \| \widehat{\pmb{z}}^{(r)} \| $, a Taylor expansion of the exponential term, and the Cauchy-Schwarz inequality, we estimate the term $ I_2 $ as

\begin{equation*}
  \begin{aligned}
      I_2 &\leq \| \widehat{\pmb{z}}^{(r)} \|  \left( \sum_{\pmb{k} \in \Lambda_{M^{r+1}}} \left( e^{i \langle \pmb{\omega}^\star, \pmb{k} \rangle t} - e^{i \langle \pmb{\omega}^{(r)}, \pmb{k} \rangle t} \right)^2 \right)^\frac{1}{2}   \\
          &\leq \left( \| \widehat{\pmb{z}}^\star \| + \exp \left\{ - \frac{3}{2} (M^s)^{r} \right\} \right) \cdot | \pmb{\omega}^{(r)} - \pmb{\omega}^\star | \cdot \left( \sum_{\pmb{k} \in \Lambda_{M^{r+1}}} \|\pmb{k} t \|^2 \right)^\frac{1}{2}  \\
          &\leq \frac{1}{2} \exp \left\{ - (M^s)^{r} \right\},
  \end{aligned}
\end{equation*}
which holds for any $ t \in [0, N_r] = [0, M^{r}] $. By combing the estimate for terms $ I_1 $ and $ I_2 $, we obtain the final time-domain error bound \eqref{eqn: convergence-soln}, which completes the proof of \Cref{thm: main}.

\section{Numerical Experiments}
\label{sec: numer}

In this section, we demonstrate the computational efficiency of the proposed alternating numerical procedure (see~\Cref{fig: algorithm-nm}). The key feature of this iterative scheme lies in its decoupled refinement strategy: the $Q$-equations~\eqref{eqn: q-eqn-vector} are employed to update the frequency via the iterative scheme~\eqref{eqn: q-eqn-numerical}, while the dimension-enlarged Newton scheme~\eqref{eqn: p-eqn-numerical} is implement to solve the $P$-equations~\eqref{eqn: p-eqn-vector} and derive the non-resonant vectors.  To demonstrate the versatility and robustness of the method, we apply it to two classical benchmark problems in nonlinear dynamics: the undamped Duffing oscillator and the Hénon–Heiles model.

\subsection{Undamped Duffing oscillator}
\label{subsec: duffing}

The undamped Duffing oscillator serves as a classical model for Hamilton systems featuring cubic nonlinearities; see, for instance,~\citet{Kanamaru:2008}. The Hamiltonian $H = H(x, y)$, defined on the phase space $(x, y) \in \mathbb{R}^2$, takes the form
\begin{equation}
\label{eqn: h-duffing}
  H =  H(x, y) = \frac{1}{2} (x^2 + y^2) + \frac{\varepsilon}{4} y^4,
\end{equation}
where the system is equipped with the standard symplectic two-form $\omega = dx \wedge dy$. To facilitate the application of our alternating numerical procedure, we introduce the complex coordinates:
\begin{subequations}
\label{eqn: coordinate-change}
\begin{empheq}[left=\empheqlbrace]{align} 
         z       &= \frac{1}{\sqrt{2}} (y - ix),       \label{eqn: coordinate-change-z} \\
        \bar{z} &= \frac{1}{\sqrt{2}} (y + ix)     \label{eqn: coordinate-change-z-bar}                                                     
\end{empheq}    
\end{subequations}
 Substuting the coordinate transformation~\eqref{eqn: coordinate-change} into the Hamiltonian~\eqref{eqn: h-duffing}, the Hamiltonian $H= H(z, \bar{z}) $ can be rewritten in complex form as
\begin{equation}
   H = H(z, \bar{z}) = |z|^2 + \frac{\varepsilon}{16} (z + \bar{z})^4,
\end{equation}
with the symplectic structures preserved as $dz \wedge d\overline{z} = -idx\wedge dy$. Earlier mathematical studies of this model have focused on quasi-periodic solutions through the lens of the Poincar\'e map and the Moser twist theorem. These approaches establish the existence of invariant tori and Lagrange stability (e.g.,~\citet{yuan1998invariant, yuan2000lagrange}), which ensure the long-term boundedness of the system's trajectories.

We demonstrate the efficiency of our alternating numerical procedure using a perturbation parameter of $\varepsilon = 1$. The evolution of the phase trajectories for iterations $r = 1$ to $r =5$ is illustrated in~\Cref{fig: traj-duff}. Starting from the linear solution serving as the initial approximation, the method rapidly captures the nonlinear dynamics. 
\begin{figure}[htb!]
\centering
\begin{subfigure}[t]{0.46\linewidth}
\centering
 \includegraphics[scale=0.16]{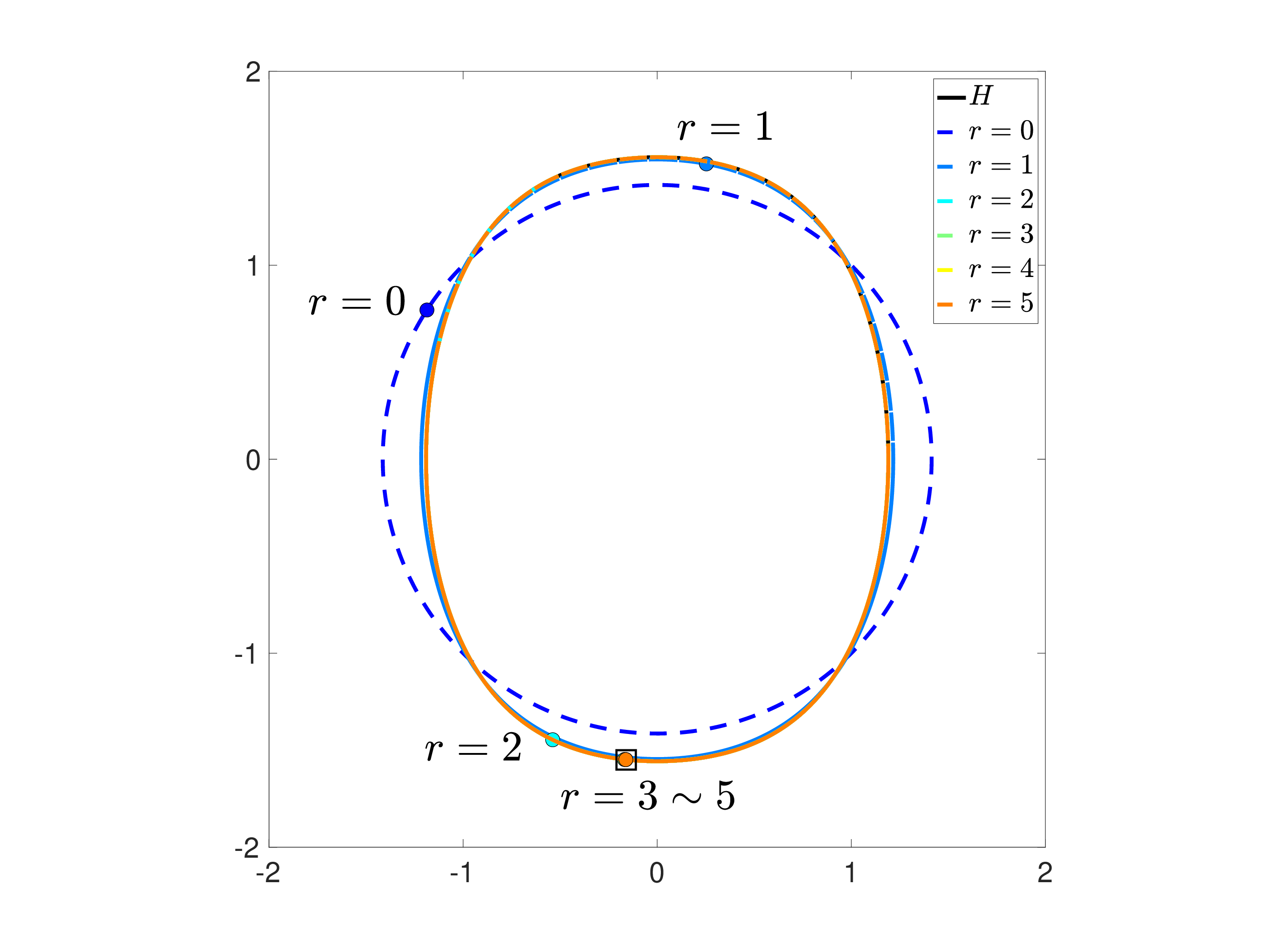}
\caption{Global phase portrait}
\end{subfigure}
\begin{subfigure}[t]{0.46\linewidth}
\centering
 \includegraphics[scale=0.16]{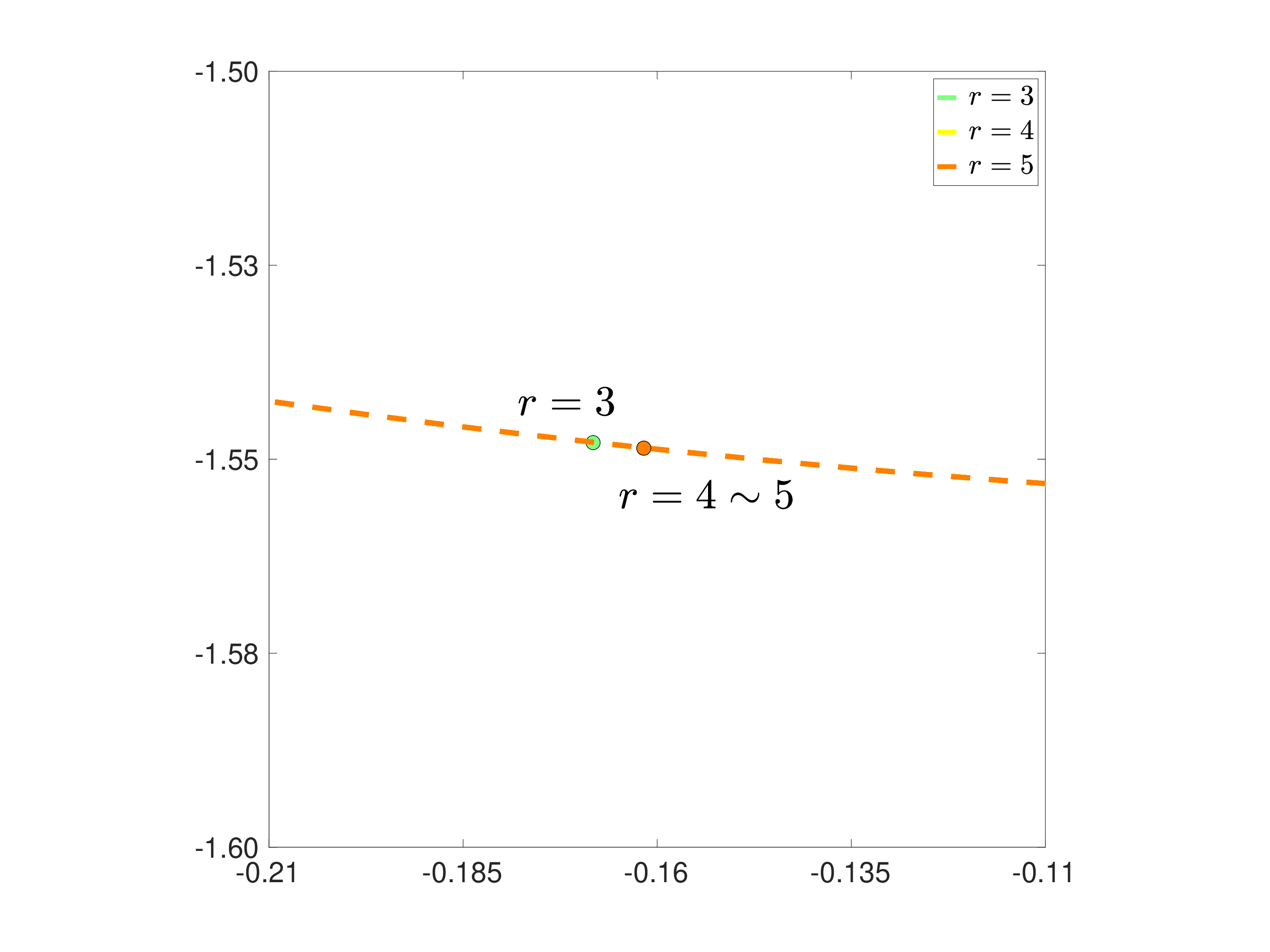}
\caption{Zoom-in view ($r = 3 \sim 5$)}
\end{subfigure}
\caption{Phase space trajectories of the Duffing oscillator with states labeled at $t=10$. The global view (a) illustrates the overall flows, while the magnification (b) highlights the rapid convergence toward the target state at $t=10$.}
\label{fig: traj-duff}
\end{figure}
By the second iteration, the solution at $t=10$ already reaches a position significantly close to the accurate trajectory.  At the third iteration, the trajectory becomes almost indistinguishable from the exact solution, with an error on the scale $O(0.01)$. For $r=4$ and $r=5$, the results become so refined that they are practically indistinguishable, even under a magnified zoom-in view.  To further assess the accuracy of the method, we examine the iteration errors,  which exhibit super-exponential convergence, as illustrated in~\Cref{fig: con-duffing}. Specifically, we track the following quantities: (i) the Fourier coefficient error, $\| \widehat{\pmb{z}}^{(r+1)} -  \widehat{\pmb{z}}^{(r+1)} \|$, measuring the convergence of the coefficients, (ii) the frequency error, $|\pmb{\omega}^{r+1} -\pmb{\omega}^{r} |$, quantifying the refinement of the computed frequency, and (iii) the pointwise solution error at $t=10$, $|\pmb{z}^{(r+1)}(10) - \pmb{z}^{(r)}(10)|$, evaluating the discrepancy of the solution at a fixed time. All three metrics consistently confirm the rapid convergence behavior of the proposed alternating scheme.

\begin{figure}[htb!]
\centering
\begin{subfigure}[t]{0.325\linewidth}
\centering
 \includegraphics[scale=0.115]{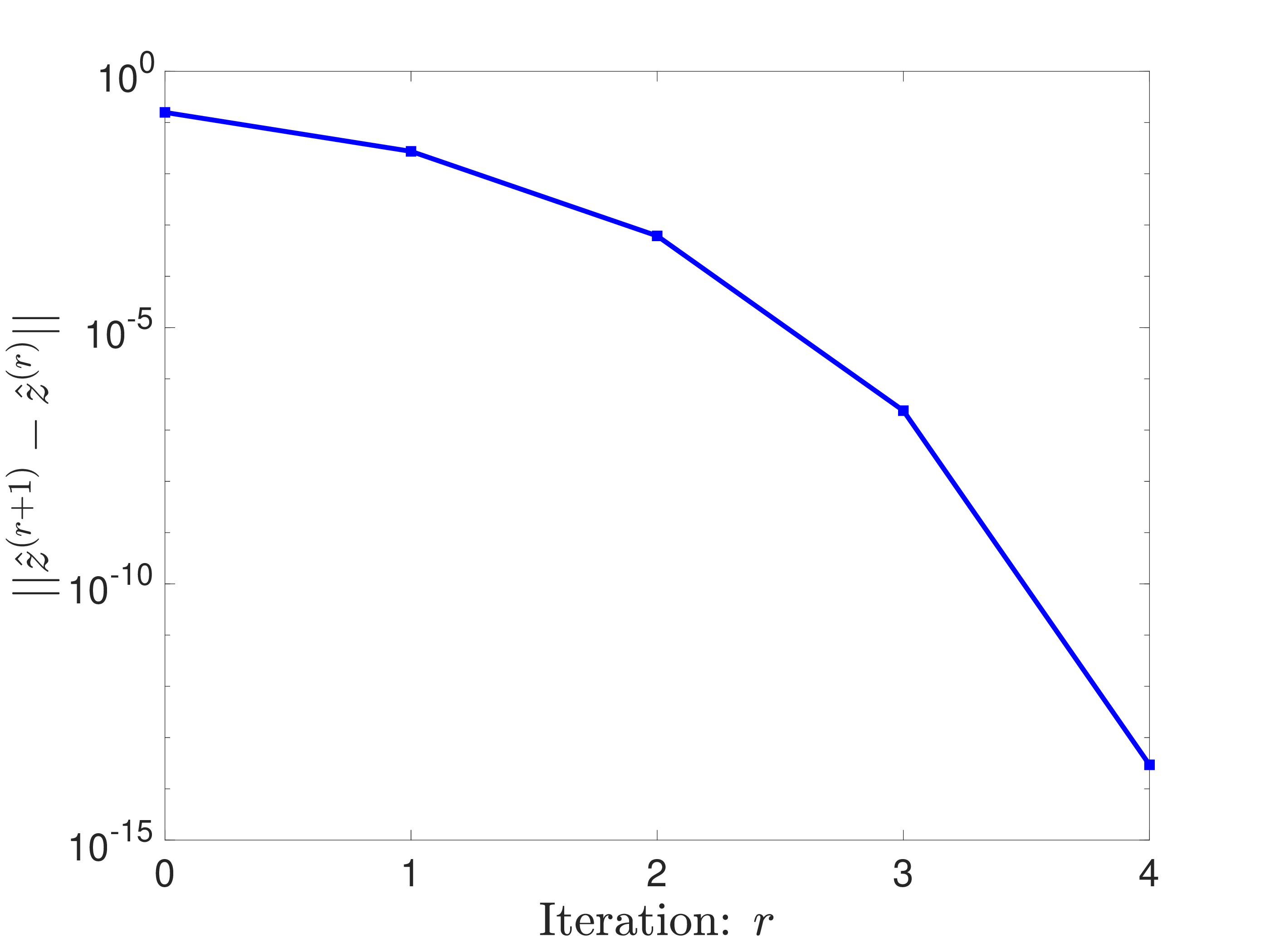}
\caption{Fourier coefficient vectors}
\end{subfigure}
\begin{subfigure}[t]{0.325\linewidth}
\centering
 \includegraphics[scale=0.115]{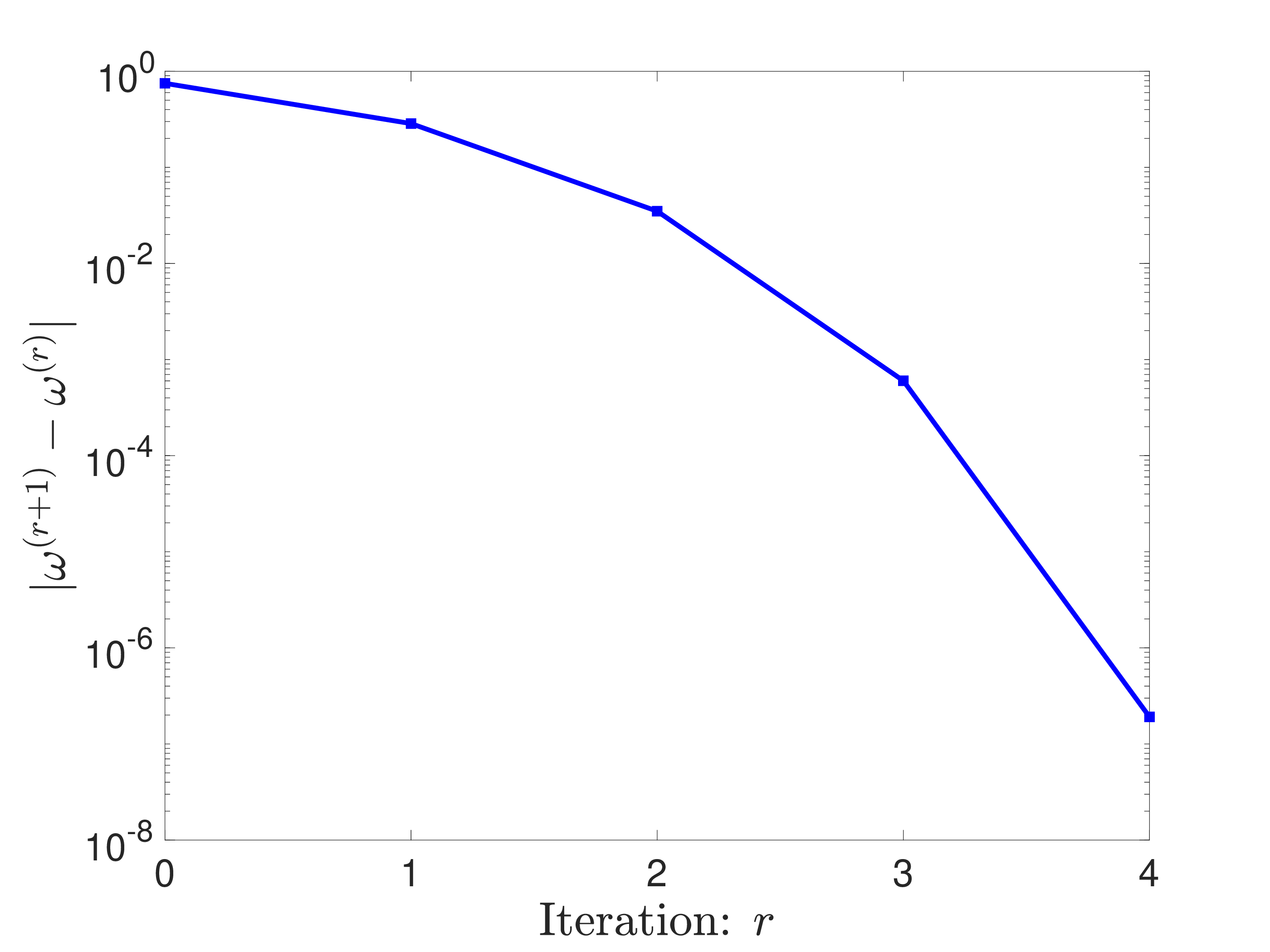}
\caption{Frequencies}
\end{subfigure}
\begin{subfigure}[t]{0.325\linewidth}
\centering
 \includegraphics[scale=0.115]{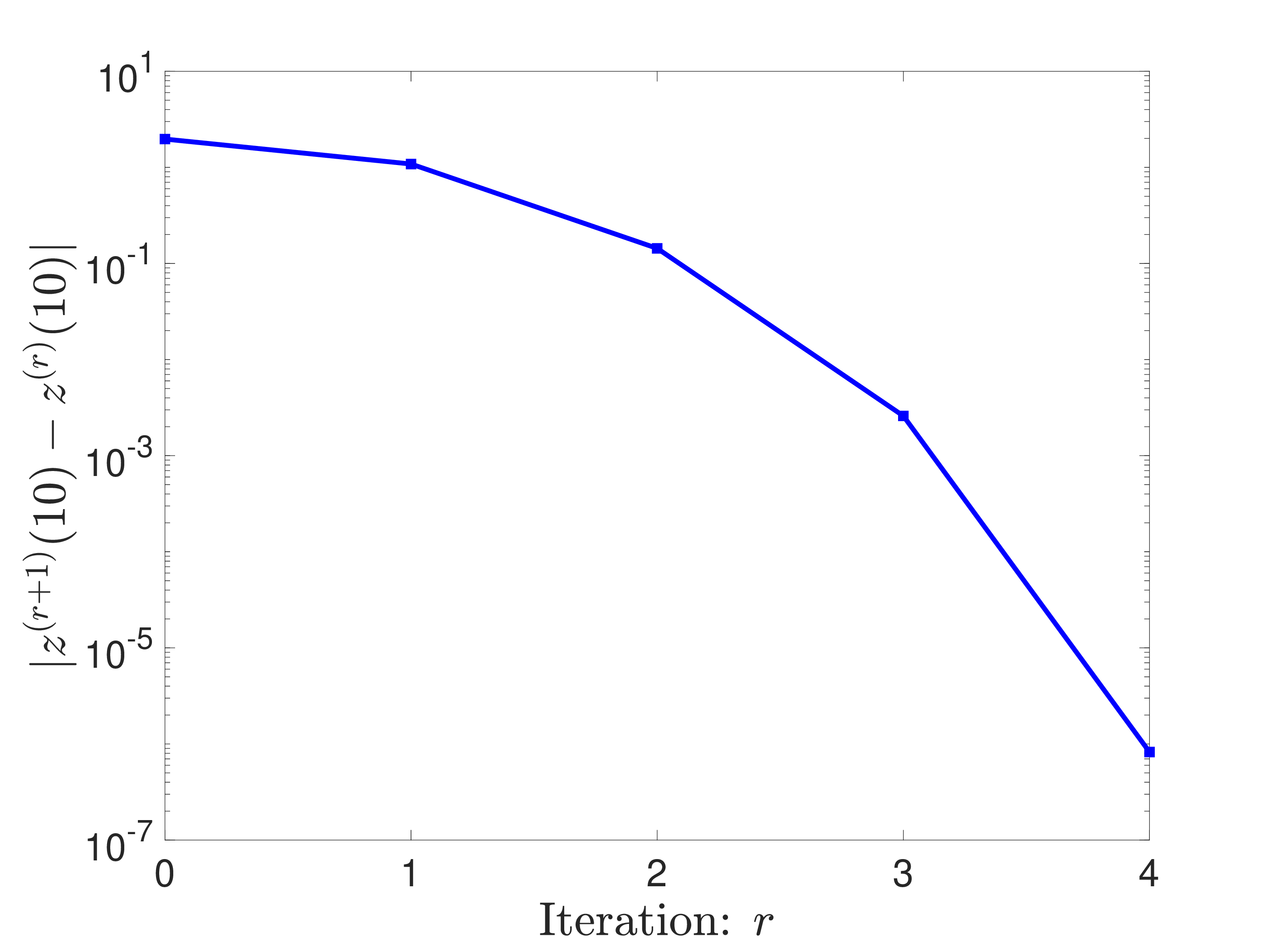}
\caption{State  at $t=10$}
\end{subfigure}
\caption{Convergence profiles across iterations for the Duffing system. The plots demonstrate the super-exponential decay of residuals for (a) the Fourier coefficient vector, (b) the identified frequencies, and (c) the solution state at $t=10$.}
\label{fig: con-duffing}
\end{figure}

\subsection{The H\'{e}non-Heiles model}
\label{subsec: henon-heiles}

While a one-dimensional Duffing system is sufficient to demonstrate the basic efficiency of our alternating numerical procedure, it is inherently limited. To fully capture the complexity of quasi-periodic solutions and the structural evolution of phase space, a higher-dimensional system is required. We therefore consider the Hénon-Heiles model, a two-dimensional cornerstone of celestial mechanics originally developed to describe stellar motion within an axisymmetric galactic potential~\citep{henon1964applicability}. The Hamiltonian $H = H(x_1, x_2; y_1, y_2)$, defined on $\mathbb{R}^4$, takes the form
\begin{equation}
\label{eqn: henon-helies-h}
    H = H(x_1, x_2, y_1, y_2) = \frac{1}{2} \omega_1 (x_1^2 + y_1^2) + \frac{1}{2} \omega_2 (x_2^2 + y_2^2) + \varepsilon (y_1^2 y_2 - \frac{1}{3} y_2^3),
\end{equation}
with the symplectic two-form $\omega = dx_1 \wedge dy_1 + dx_2 \wedge dy_2$. Similarly, to implement our alternating numerical procedure, we apply the complex coordinate transformation~\eqref{eqn: coordinate-change} into~the Hamiltonian~\eqref{eqn: henon-helies-h} as
\begin{equation}
\label{eqn: henon-helies-h-complex}
    H = H(z, \bar{z}) = \omega_1 |z_1|^2 + \omega_2 |z_2|^2 + \varepsilon \left[ \frac{1}{2\sqrt{2}} (z_1 + \bar{z}_1)^2 (z_2 + \bar{z}_2) - \frac{1}{6\sqrt{2}}(z_2 + \bar{z}_2)^3 \right],
\end{equation}
where the symplectic form is expressed as $w = dz_1 \wedge d\overline{z}_1 + dz_2 \wedge d\overline{z}_2 = -i \left(dx_1\wedge dy_1 + dx_2\wedge dy_2 \right)$. As a nearly-integrable system, this model serves as a rigorous testing ground for both perturbation theory and numerical integration. It comprises a harmonic oscillator base coupled with third-order nonlinear terms that trigger a transition from ordered to chaotic motion as energy increases.  The mathematical complexity of its nearly-integrable coupling was further formalized by~\citet{channell1990symplectic}, who utilized the model to demonstrate how symplectic integrators maintain the structural integrity of the phase space over long temporal scales.

\begin{figure}[htb!]
\centering
\begin{subfigure}[t]{0.46\linewidth}
\centering
 \includegraphics[scale=0.16]{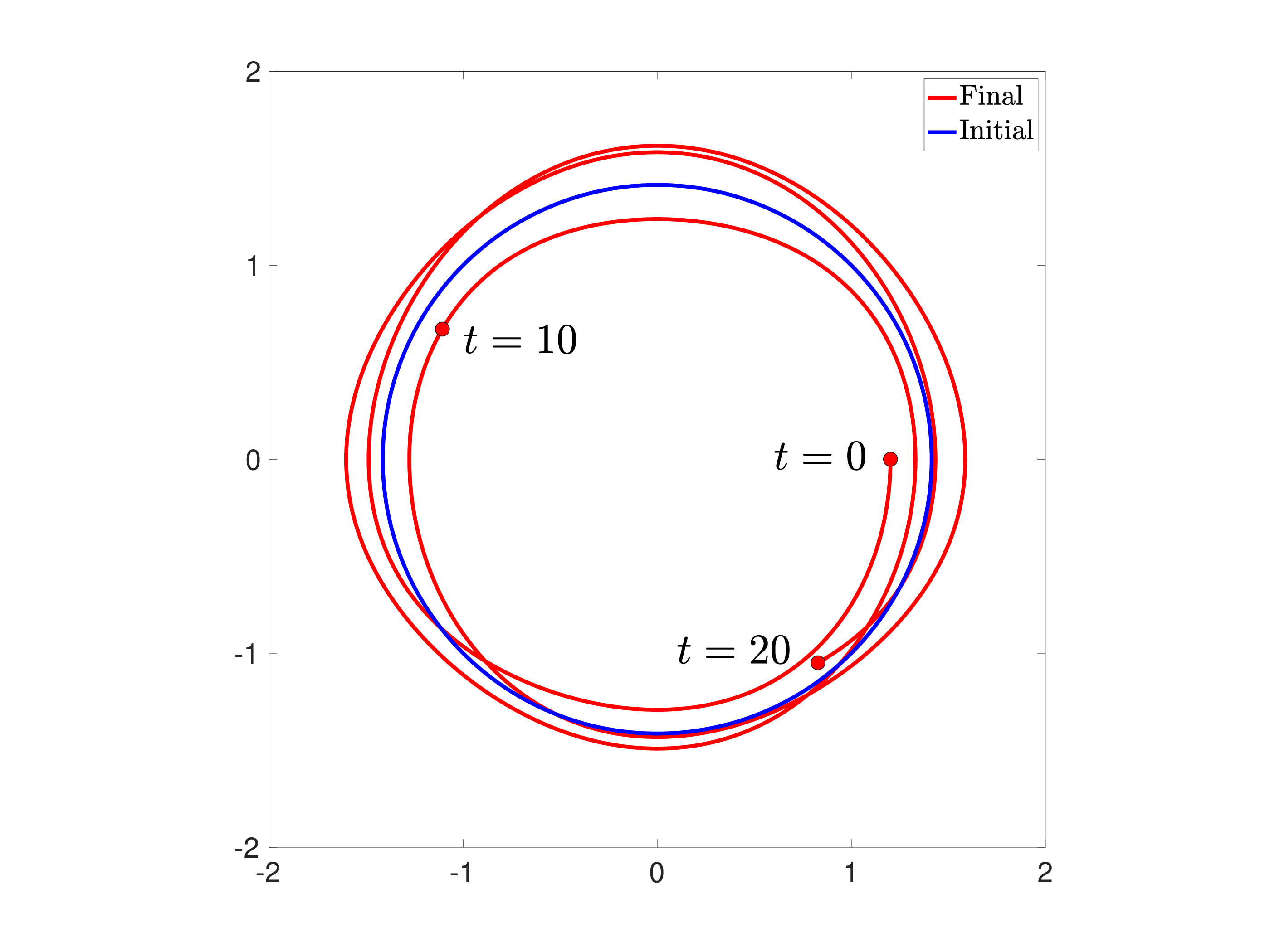}
\caption{$(x_1, y_1)$}
\end{subfigure}
\begin{subfigure}[t]{0.46\linewidth}
\centering
 \includegraphics[scale=0.16]{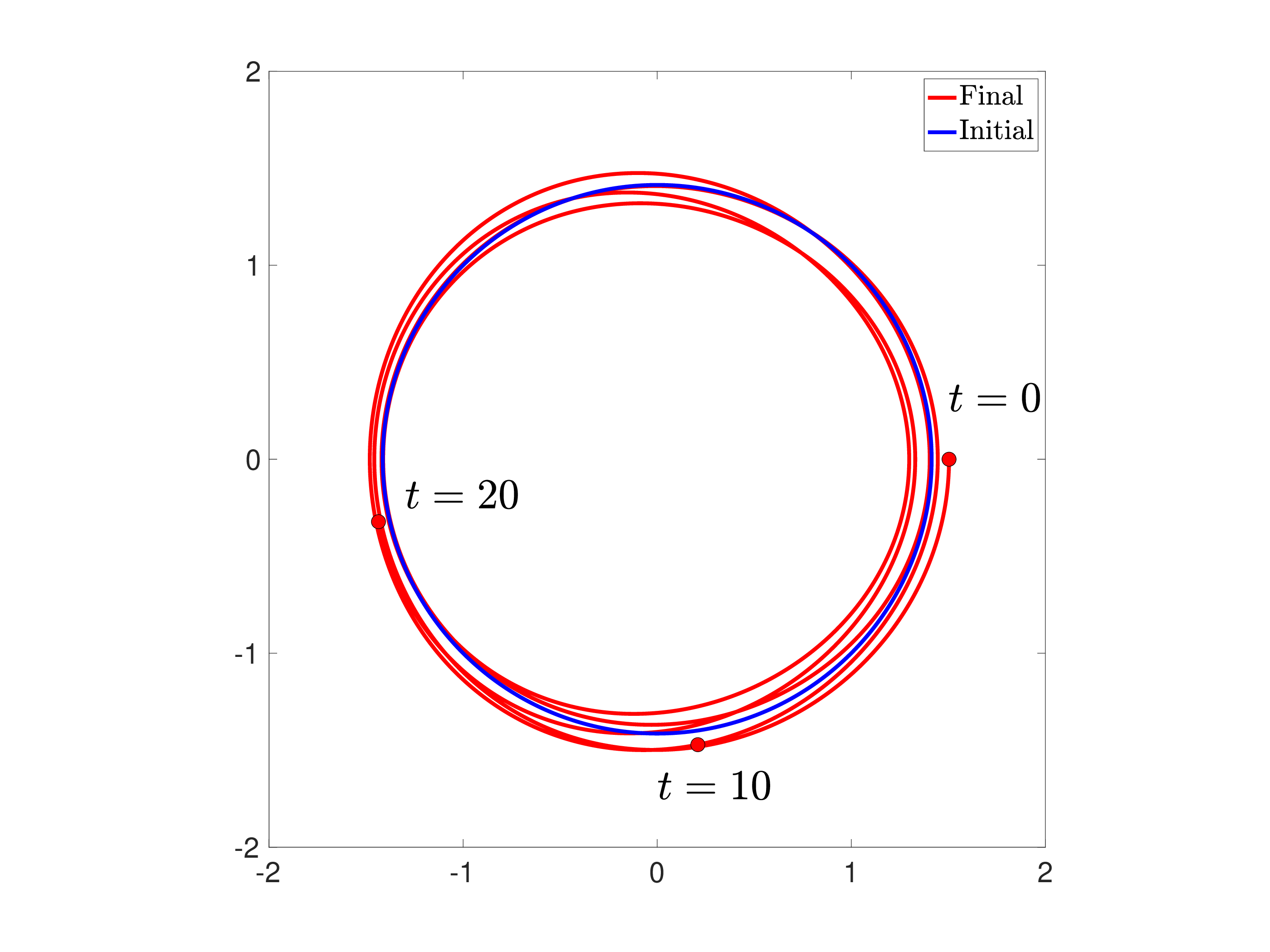}
\caption{$(x_2, y_2)$}
\end{subfigure}
\caption{Phase space trajectories of the H\'enon–Heiles model. The evolution of system states is highlighted at discrete time intervals ($t = 0$, $t = 10$, and $t = 20$), illustrating the dynamic progression of the trajectory. }
\label{fig: henon-heiles-1-dim}
\end{figure}
While symplectic methods are renowned for preserving the ``shape'' of the orbit, specifically the invariant tori, they often suffer from the accumulation of phase-lag errors over long integration periods. As a result, even when the qualitative shape of the orbit is maintained, the numerical trajectory may gradually drift along the torus. Our alternating numerical procedure effectively addresses this limitation by precisely predicting the instantaneous positions and phase angles along these trajectories. With the perturbation parameter set to $\varepsilon = 0.1$,~\Cref{fig: henon-heiles-1-dim} demonstrates that our method captures the exact state of the system within the two phase planes $(x_1, y_1)$ and $(x_2, y_2)$.From a theoretical sense, independently of computational resource constraints, the proposed procedure allows phase-lag errors to be reduced arbitrarily. Consequently, the particle's position at any prescribed time can be computed with high precision, thereby eliminating the temporal drift that typically compromises the long-term reliability of conventional Hamiltonian simulations.
\begin{figure}[htb!]
\centering
\begin{subfigure}[t]{0.42\linewidth}
\centering
 \includegraphics[scale=0.14]{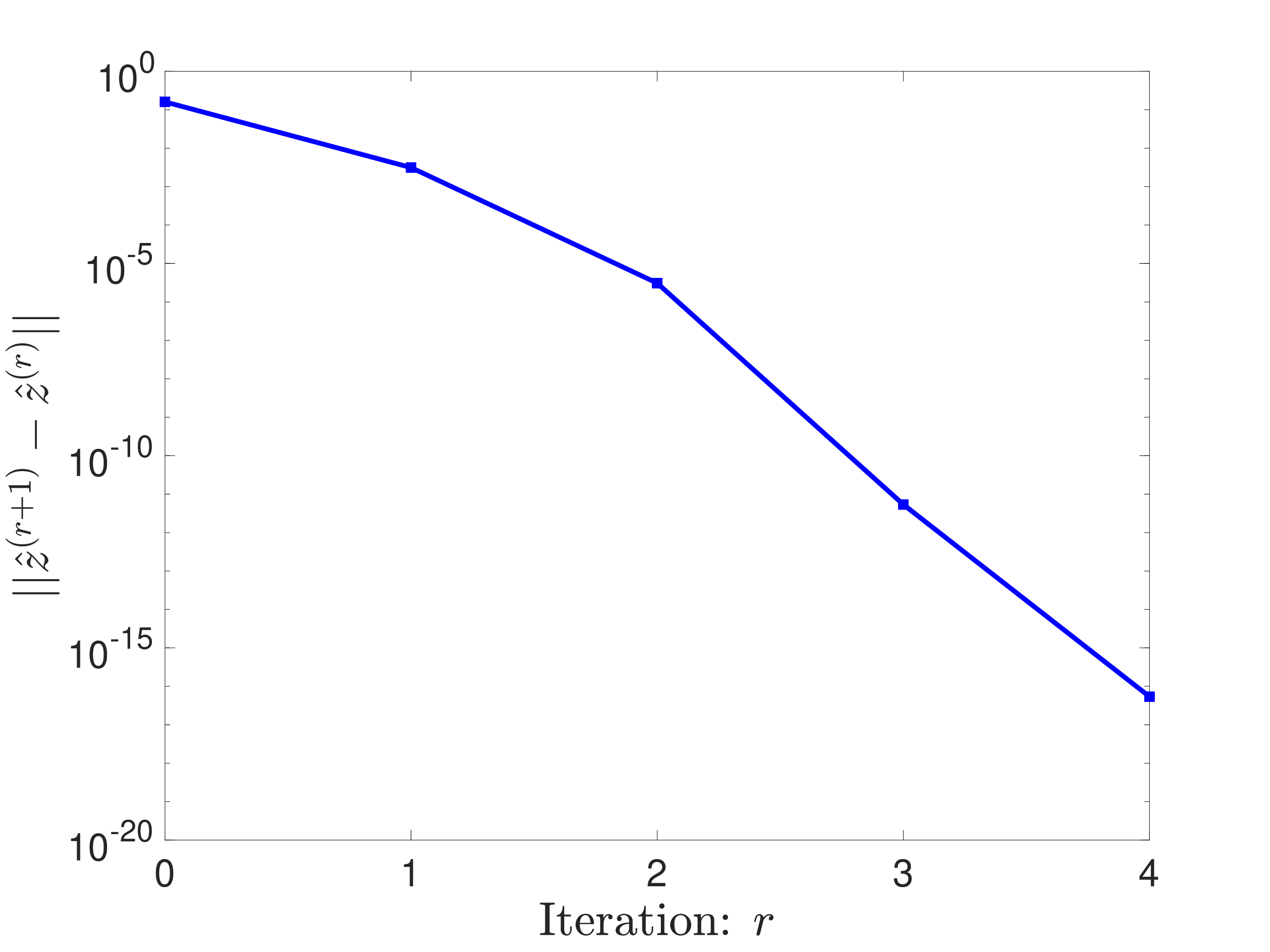}
\caption{Fourier coefficient vectors}
\end{subfigure}
\begin{subfigure}[t]{0.42\linewidth}
\centering
 \includegraphics[scale=0.14]{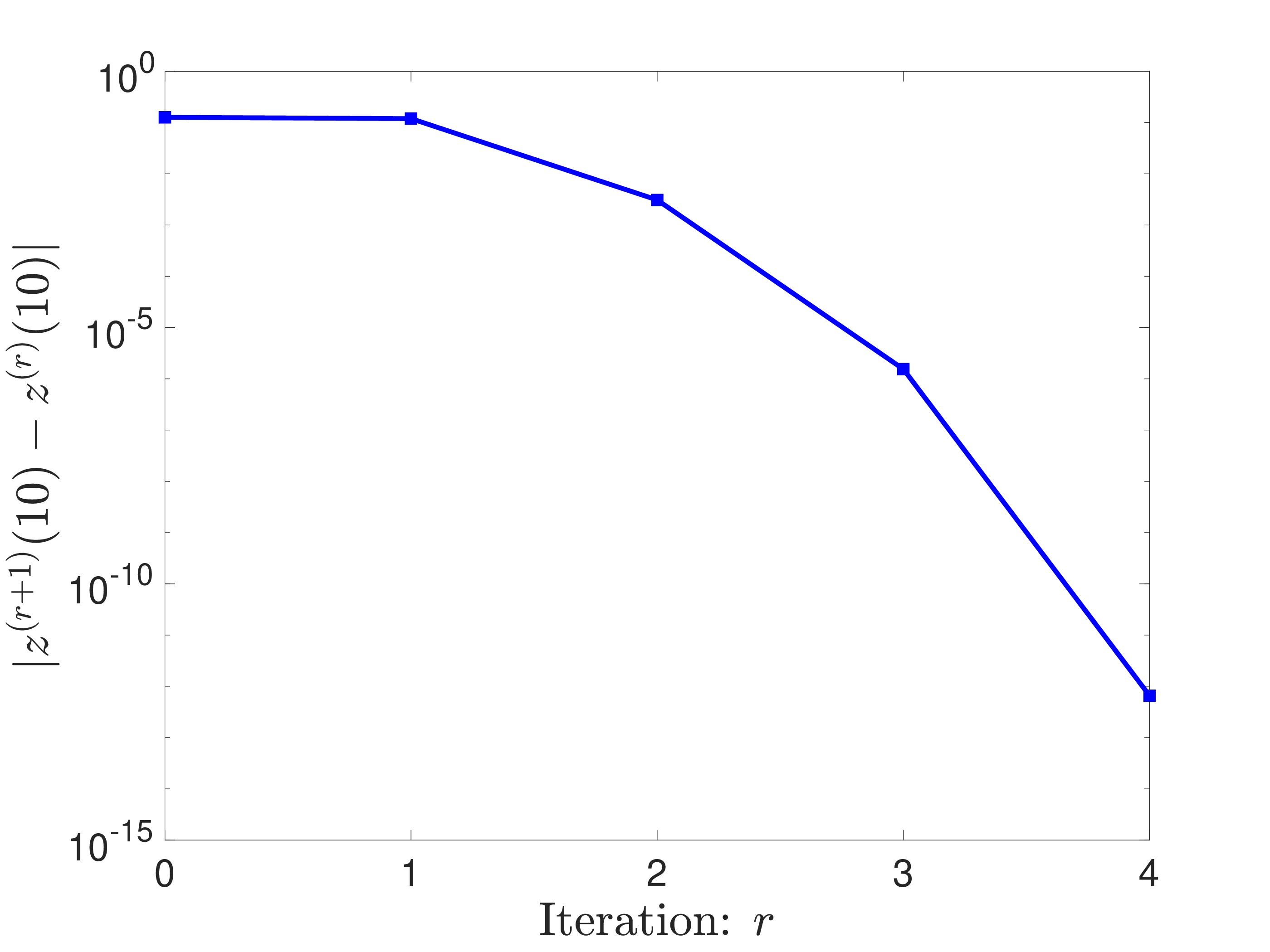}
\caption{State at $t=10$}
\end{subfigure}
\caption{Convergence profiles across iterations for the H\'enon–Heiles model.  (a) Norm decay of the Fourier coefficient vector. (b) Numerical solution state evaluated at $t=10$. }
\label{fig: henon-heiles-con}
\end{figure}
Furthermore, we exhibit super-exponential convergence in~\Cref{fig: henon-heiles-con-freq} for both the Fourier coefficient error, $\| \widehat{\pmb{z}}^{(r+1)} -  \widehat{\pmb{z}}^{(r+1)} \|$ and the pointwise solution error at $t=10$, $|\pmb{z}^{(r+1)}(10) - \pmb{z}^{(r)}(10)|$. A subtle point requires attention here: the support for the complex Hamilton vector field $\pmb{X} = (\partial H/ \partial \overline{z}_1, \partial H/ \partial z_2)$ does not include $\pmb{e}_1$ and $\pmb{e}_2$ as given by:
\[
\left\{ 
\begin{aligned}
& \mathrm{supp} \widehat{\frac{\partial H}{\partial \overline{z}_1}} = \left\{ \pm \pmb{e}_1 \pm \pmb{e}_2 \right\} \\
& \mathrm{supp} \widehat{\frac{\partial H}{\partial \overline{z}_2}} = \left\{ 0,  \pm 2\pmb{e}_1, \pm 2\pmb{e}_2 \right\}
\end{aligned} \right.
\]
\begin{figure}[htb!]
\centering
\begin{subfigure}[t]{0.42\linewidth}
\centering
 \includegraphics[scale=0.14]{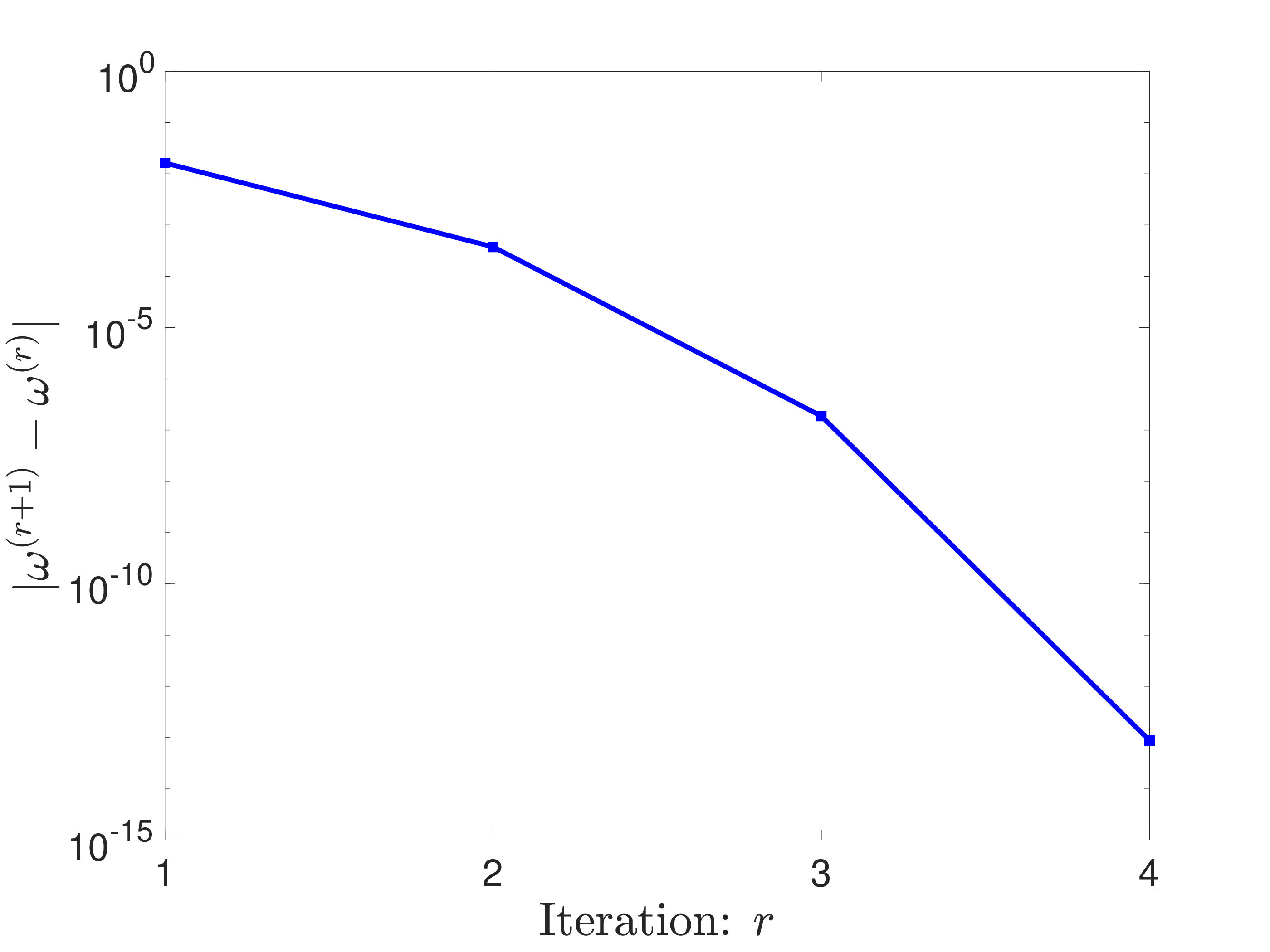}
\caption{Frequency error}
\end{subfigure}
\begin{subfigure}[t]{0.42\linewidth}
\centering
 \includegraphics[scale=0.14]{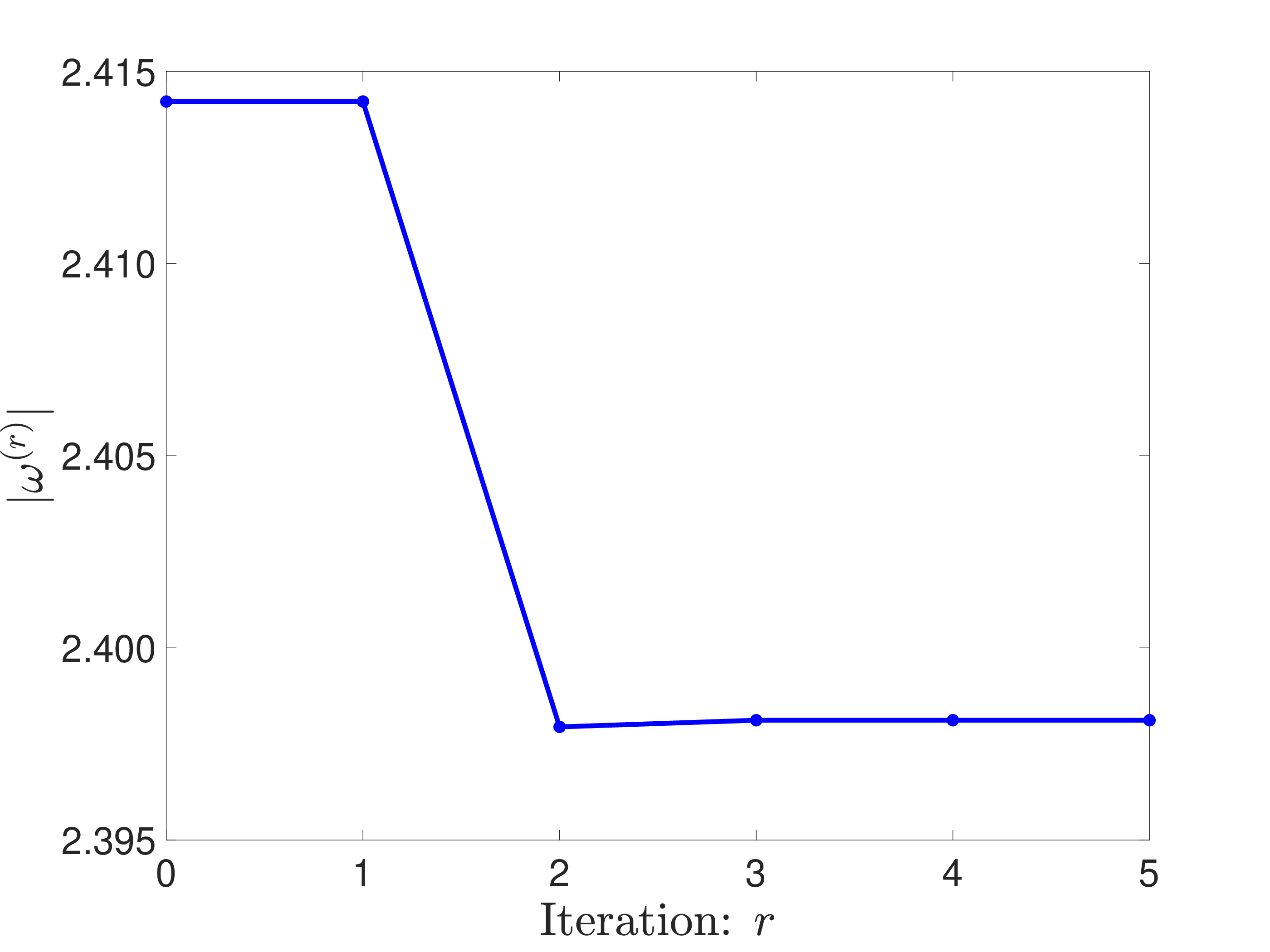}
\caption{Frequencies}
\end{subfigure}
\caption{The behavior of frequencies across iterations for the H\'enon–Heiles model. (a) Reduction of frequency error. (b) The frequency iteration.}
\label{fig: henon-heiles-con-freq}
\end{figure}
Thus, the vector field associated with the $Q$-equations in~\eqref{eqn: q-eqn-vector} is zero, meaning that the associated frequency update~\eqref{eqn: q-eqn-numerical} remains invariant during the first step. The resulting numerical performance for the frequency iteration, which clearly illustrates this super-exponential convergence, is depicted in~\Cref{fig: henon-heiles-con-freq}.

\section{Conclusion and further work}
\label{sec: conclu}

In this study, we developed an alternating numerical procedure that bridges the gap between theoretical infinite-dimensional analysis and practical computation. By introducing a numerical update for frequency derived from the Q-equations, we successfully integrated an iterative framework with the dimension-enlarged Newton scheme (Nash-Moser scheme) for P-equations. This work establishes a novel computational pathway for the CWB framework, enabling the practical application of previously abstract theoretical constructs. The scheme achieves arbitrary numerical accuracy in capturing the intricate dynamics of quasi-periodic solutions within nearly-integrable Hamiltonian systems, providing a functional realization of the CWB scheme. We have significantly streamlined the multi-scale analysis developed by~\citet{bourgain1998quasi} to establish the convergence of this algorithm. This simplification ensures that the numerical results maintain rigorous mathematical integrity while remaining computationally accessible.


%
%
%
%
%
%
%

The methodology presented here lays the groundwork for several extensions. Our immediate focus involves the computation of lower-dimensional quasi-periodic solutions for finite-dimensional systems. In forthcoming work, we will investigate the numerical acquisition of quasi-periodic solutions for nonlinear partial differential equations (PDEs) using this alternating algorithm. Further research will examine numerical convergence and stability, specifically focusing on the influence of small divisors on the robustness and efficiency of the expanded numerical scheme.

\section*{Acknowledgements}
This work was partially supported by the NSFC (Grant No. 12241105) and by SIMIS (startup fund and cross-disciplinary research projects).


\bibliographystyle{abbrvnat}
\bibliography{sigproc}

\appendix
\section{Proofs of results in~\Cref{sec: priori-error}}
\label{sec: priori-error-app}

In this appendix, we provide the detailed proofs for the results presented in~\Cref{sec: priori-error}, specifically addressing~\Cref{prop: vector-field} (\Cref{subsec: vector-field}),~\Cref{prop: convolution} (\Cref{subsec: convolution}),~\Cref{prop: B-operator} (\Cref{subsec: b-operator}), and~\Cref{prop: tensor} (\Cref{subsec: tensor}).  To facilitate these proofs, we first establish an elementary inequality and demonstrate that the Gevrey decay property is preserved under convolution. These preliminary results are rigorously stated in the following lemmas. 
\subsection{Preliminary lemmas}
\label{subsec: preliminary}

\begin{lemma}
    \label{lemma: App-1}
    Let $a, b \ge 0$ and $0 < s < 1$. Then the following inequality holds:
    \begin{equation}
    \label{eqn: ineq_power}
        a^s + b^s - (a+b)^s \ge (2-2^s) \min\{a,b\}^s 
    \end{equation}
\end{lemma}

\begin{proof}
    Assume without loss of generality that $a \ge b$. 
    If $b=0$, the inequality is trivial. 
    For $b > 0$, we divide both sides of \eqref{eqn: ineq_power} by $b^s$ to obtain:
    \[ \left(\frac{a}{b}\right)^s + 1 - \left(1 + \frac{a}{b}\right)^s \ge 2 - 2^s \]
    Let $x = a/b \in [1, \infty)$ and define the function $f(x) = x^s + 1 - (1+x)^s$. 
    Its derivative is
    \[ f'(x)  = s \left[ \frac{1}{x^{1-s}} - \frac{1}{(1+x)^{1-s}} \right]. \]
    Since $1-s > 0$ and $x < 1+x$, it follows that $x^{1-s} < (1+x)^{1-s}$, which implies $f'(x) > 0$. 
    Therefore, $f$ is monotonically increasing on $[1, \infty)$, and its minimum occurs at $x=1$, where $f(1) = 2 - 2^s$. This completes the proof.
\end{proof}

Next, we use \Cref{lemma: App-1} to show that the Gevrey decay property is preserved under
convolution.

\begin{lemma}
\label{lemma: App-2}
Let the Gevrey decay set $\mathcal{K}(s)$ be defined in~\eqref{eqn: gevrey-l2}. For any two Fourier vectors $\widehat{a}, \widehat{b}  \in  \mathcal{K}(s)$, there exists a constant $C_2(n, s) > 0$ such that their convolution satisfies 
\begin{equation}    
\label{eqn: convolution}   
\sup_{\pmb{k} \in \mathbb{Z}^n} \left( |(\widehat{a}*\widehat{b})(\pmb{k})| \exp\left\{| \pmb{k} |^{s} \right\} \right) \leq C_2(n,s)
\end{equation}  
 More generally, if $\widehat{a}_j(\pmb{k}) \in \mathcal{K}(s)$ for $j=1, \ldots,m$, then there exists a constant $C_m(n, s) > 0$ such that the $m$-fold convolution satisfies:
 \begin{equation}    
\label{eqn: m-convolution}   
\sup_{\pmb{k} \in \mathbb{Z}^n} \left( \big| \big(\widehat{a}_1 * \widehat{a}_2 * \cdots * \widehat{a}_m \big)(\pmb{k}) \big| \exp\left\{| \pmb{k} |^{s} \right\} \right) \leq C_m(n,s)
\end{equation}   

\end{lemma}

\begin{proof}
Given that the Fourier vectors $\widehat{a}, \widehat{b}  \in  \mathcal{K}(s)$, we calculate their convolution as
 \begin{align*}
     |(\widehat{a}*\widehat{b})(\pmb{k})| & \le \sum_{\pmb{k}' \in \mathbb{Z}^n} |\widehat{a}(\pmb{k}-\pmb{k}')| |\widehat{b}(\pmb{k}')|                            \nonumber          \\
                              & \leq \sum_{\pmb{k}' \in \mathbb{Z}^n} \exp\left\{-|\pmb{k}-\pmb{k}'|^{s}  \right\} \exp\left\{-|\pmb{k}'|^{s} \right\}                       \nonumber          \\
                              & \le \sum_{\pmb{k}' \in \mathbb{Z}^n} \exp\left\{-| \pmb{k} |^{s} \right\} \exp\left\{- \left( |\pmb{k}-\pmb{k}'|^{s}  + | \pmb{k}' |^{s}  - \left( |\pmb{k}-\pmb{k}'| + | \pmb{k}' | \right)^{s}  \right) \right\},                                                                                                                                                                   
 \end{align*}
where the last step uses the triangle inequality. Applying~\Cref{lemma: App-1} to bound the sum involving the remaining exponential terms, we have
\[
|(\widehat{a}*\widehat{b})(\pmb{k})|  \leq  \exp\left\{-| \pmb{k} |^{s} \right\} \left( \sum_{\pmb{k}' \in \mathbb{Z}^n} \exp\left\{ - (2 - 2^{s}) \min\left\{ | \pmb{k} - \pmb{k}'|^{s}, \; | \pmb{k}' |^{s} \right\} \right\}  \right).     
\]
Since the sum of the Gevrey decay series converges in an $n$-dimensional lattice $\mathbb{Z}^n$, we establish the inequality~\eqref{eqn: convolution}.  By induction, this same argument extends to any finite $m$-fold convolutional products,  thereby establishing the inequalities~\eqref{eqn: m-convolution}.
\end{proof}

\subsection{Proof of~\Cref{prop: vector-field}}
\label{subsec: vector-field}

Since the perturbation Hamiltonian $H_1$ is a real-coefficient polynomial, each components of its associated vector field, $\partial H/ \partial \overline{\pmb{z}}_{j}$ for $j=1,\ldots, n$, is likewise a real-coefficient polynomial. Given that $\widehat{\pmb{z}} \in \mathcal{K}(s)$,~\Cref{lemma: App-2} implies that the coefficient $\widehat{\pmb{X}}(\pmb{k})$ exhibit Gevrey decay. By applying the Leibniz rule, it follows that the derivatives $\partial_{\pmb{\omega}}\widehat{\pmb{X}}(\pmb{k})$ also exhibits Gevrey decay, which leads directly to the estimate in~\eqref{eqn: vector-field-Gevrey}. By summing these decaying terms (leveraging the convergence of the Gevrey series) or applying the supremum bound across the domain, it follows that the total vector field is bounded. This establishes the uniform bound in~\eqref{eqn: vector-field-bound}.

\subsection{Proof of~\Cref{prop: convolution}}
\label{subsec: convolution}

Following a similar argument as used for the vector field, we apply~\Cref{lemma: App-2} to establish Gevrey decay for the entries of the Hessian operator and their derivative with respect to the frequency $\pmb{\omega}$. Since the two operators  $\mathcal{H}_1$ and $\mathcal{H}_2$ exhibit Toeplitz-like and Hankel-like structures respectively, we can establish the necessary bound over the full lattice $\mathbb{Z}^n$. This results naturally extends to the non-resonant set in~\eqref{eqn: hessian-gevrey}. To bound the norm for the operator $\mathcal{H}_1$, we consider the norm $\mathcal{H}_{1} \widehat{\pmb{z}}$ for any $ \widehat{\pmb{z}}  \in \ell_2$ as
\begin{align*}
\| \mathcal{H}_{1} \widehat{\pmb{z}} \|^2 &=  \sum_{\pmb{k} \in \mathbb{Z}^n}\bigg\| \sum_{\pmb{k}' \in \mathbb{Z}^n}\mathcal{H}_{1}(\pmb{k'}) \widehat{\pmb{z}}(\pmb{k} - \pmb{k'})  \bigg\|^2 \\
                                                                                                        & \leq \sum_{\pmb{k} \in \mathbb{Z}^n} \left( \sum_{\pmb{k}' \in \mathbb{Z}^n} \left\| \mathcal{H}_1(\pmb{k'})   \right\|^{\frac12} \cdot \left\| \mathcal{H}_1(\pmb{k'})   \right\|^{\frac12}  \|\widehat{\pmb{z}}(\pmb{k}  - \pmb{k'}) \| \right)^2                                                                                               
 \end{align*}
Since the operator $\mathcal{H}_1$ exhibits Gevrey decay, we apply the Cauchy-Schwarz inequality to derive the following bound as
\begin{equation}
\label{eqn: h1-bound}
 \| \mathcal{H}_1 \widehat{\pmb{z}} \|^2 \leq \bigg[ \sum_{\pmb{k} \in \mathbb{Z}^n} \|\mathcal{H}_1(\pmb{k} )\|   \bigg]^2  \bigg[ \sum_{\pmb{k} \in \mathbb{Z}^n} \|\widehat{\pmb{z}}(\pmb{k})\|^2    \bigg] <\infty.
\end{equation}
Similarly, we can establish a corresponding bound for the operator $\mathcal{H}_2$ as
\begin{equation}
\label{eqn: h2-bound}
 \| \mathcal{H}_2 \widehat{\pmb{z}} \|^2 \leq \bigg[ \sum_{\pmb{k} \in \mathbb{Z}^n} \|\mathcal{H}_2(\pmb{k} )\|   \bigg]^2  \bigg[ \sum_{\pmb{k} \in \mathbb{Z}^n} \|\widehat{\pmb{z}}(\pmb{k})\|^2    \bigg] <\infty.
\end{equation}
By combining these estimates, we successfully establish the bound presented in~\eqref{eqn: hessian-bound}.

\subsection{Proof of~\Cref{prop: B-operator}}
\label{subsec: b-operator}

Recalling that the frequency iteration operator $B$ from~\eqref{eqn: rank-1},  we observe that for any $\pmb{k}, \pmb{k}' \in \mathcal{S}$, its entries satisfy the following inequality: 
\[
\| B(\pmb{k}, \pmb{k}') \| \leq \frac{1}{e} \left\| \frac{ \partial \langle \pmb{k}, \widehat{\pmb{X}}_q \rangle  }{ \partial \widehat{\pmb{z}}_p(\pmb{k}') }~\widehat{\pmb{z}}_p (\pmb{k}) \right\|.
\]
For any $\widehat{\pmb{z}} \in \mathcal{K}(s)$, we apply the decay properties established in~\Cref{prop: vector-field} to verify the Gevrey bound in~\eqref{eqn: b-gevrey}. Given that   $B$ is a rank-one operator,  we apply the Cauchy-Schwarz inequality to bound its norm as follows:
\[
\| B(\pmb{k}, \pmb{k}') \| \leq \frac{1}{e} \| \pmb{k} \widehat{\pmb{z}}_p(\pmb{k}) \| \cdot \left\| \frac{ \partial \widehat{\pmb{X}}_q  }{ \partial \widehat{\pmb{z}}_p(\pmb{k}') } \right\| 
\]
By utilizing the Gevrey decay of the Hessian operators derived in~\Cref{subsec: convolution}, we extend the bound over the lattice $\mathbb{Z}^n \times \mathbb{Z}^n$ as
\begin{equation*}
\| B \| \leq \sum_{\pmb{k} \in \mathbb{Z}^n} \left( \frac{\| \pmb{k} \widehat{\pmb{z}}_p(\pmb{k}) \|}{e} \right) \cdot \sum_{\pmb{k}' \in \mathbb{Z}^n} \left( \left\|  \frac{\partial  \widehat{\pmb{X}}_q}{\partial \widehat{\pmb{z}}_p} (\pmb{k}') \right\| \right).
\end{equation*}
Finally, since $\widehat{\pmb{z}} \in \mathcal{K}(s)$, we invoke~\Cref{prop: vector-field} to ensure the convergence of the infinite sum over the lattice, which confirms the uniform bound as stated in~\eqref{eqn: b-bound}.

\subsection{Proof of~\Cref{prop: tensor}}
\label{subsec: tensor}

As the same procedure used for vector fields and tangent linear operators, we analyze the three-order tensor $\partial^2 \widehat{\pmb{X}} / \partial \widehat{\pmb{z}}^{2}$. For any $\pmb{k}, \pmb{k}', \pmb{k}'' \in \mathbb{Z}^n$, we derive the kernels as
\begin{align*}
\frac{\partial^2 \widehat{\pmb{X}}(\pmb{k})}{  \partial \widehat{\pmb{z}}(\pmb{k}') \partial \widehat{\pmb{z}}(\pmb{k}'') }   = &  \mathcal{T}_{11}(\pmb{k} - \pmb{k}' - \pmb{k}'') +  \mathcal{T}_{21}(\pmb{k} + \pmb{k}' - \pmb{k}'') +  \mathcal{T}_{12}(\pmb{k} - \pmb{k}' + \pmb{k}'')  +  \mathcal{T}_{22}(\pmb{k} + \pmb{k}' + \pmb{k}'')  \nonumber\\
= &  \frac{1}{(2\pi)^{n}} \int_{\mathbb{T}^{n}} \frac{\partial^3 H_1}{\partial \pmb{z}^2 \partial \overline{\pmb{z}}} e^{-i \langle \pmb{k} - \pmb{k}' - \pmb{k}'', \pmb{\theta} \rangle } d\pmb{\theta} + \frac{1}{(2\pi)^{n}} \int_{\mathbb{T}^{n}} \frac{\partial^3 H_1}{\partial \pmb{z} \partial \overline{\pmb{z}}^2} e^{-i \langle \pmb{k} + \pmb{k}' - \pmb{k}'', \pmb{\theta} \rangle } d\pmb{\theta}   \nonumber   \\
                                                                                                                                                                                            + & \frac{1}{(2\pi)^{n}} \int_{\mathbb{T}^{n}} \frac{\partial^3 H_1}{\partial \pmb{z} \partial \overline{\pmb{z}}^2} e^{-i \langle \pmb{k} - \pmb{k}' + \pmb{k}'', \pmb{\theta} \rangle } d\pmb{\theta} + \frac{1}{(2\pi)^{n}} \int_{\mathbb{T}^{n}} \frac{\partial^3 H_1}{\partial \pmb{z}^2 \partial \overline{\pmb{z}}} e^{-i \langle \pmb{k} + \pmb{k}' + \pmb{k}'', \pmb{\theta} \rangle } d\pmb{\theta}.  
\end{align*}
Given that $\widehat{\pmb{z}} \in \mathcal{K}(s)$ and the perturbation Hamiltonian $H_1$ is a real-coefficient polynomial,  there exists some constant $\gamma_9 = \gamma_9(H_1, s) > 0$ such that 
\[
\sup_{\pmb{k}, \pmb{k}', \pmb{k}'' \in \mathbb{Z}^n}\bigg(  \| \mathcal{T}_{ij}(\pmb{k} + (-1)^i\pmb{k}' + (-1)^j\pmb{k}'') \| \exp\left\{ |\pmb{k} + (-1)^i\pmb{k}' + (-1)^j\pmb{k}'''|^s \right\} \bigg) \leq \gamma_9,
\] 
which demonstrates that the tensors $\mathcal{T}_{ij}$ for $i, j \in \{ 1, 2 \}$ exhibits Gevrey decay. To establish the boundedness of the tensor, we apply the Cauchy-Schwarz inequality twice, first to decouple the $\mathbf{k}''$ summation and then to decouple the $\mathbf{k}'$ convolution. Taking $\mathcal{T}_{11}$ as a representative example, we derive the following bound:
\begin{align*}
\| \mathcal{T}_{11} (\widehat{\pmb{z}}_1(\pmb{k}'), \widehat{\pmb{z}}_2(\pmb{k}'' ) \|^2 &=  \sum_{\pmb{k} \in \mathbb{Z}^n}\bigg\| \sum_{\pmb{k}'' \in \mathbb{Z}^n} \bigg(\sum_{\pmb{k}' \in \mathbb{Z}^n} \mathcal{T}_{11} (\pmb{k} - \pmb{k}' - \pmb{k}'') \widehat{\pmb{z}}_1(\pmb{k}') \bigg) \widehat{\pmb{z}}_2(\pmb{k}'' ) \bigg\|^2 \\
                                                                                                        & \leq \left[ \sum_{\pmb{k} \in \mathbb{Z}^n} \left\|\sum_{\pmb{k}' \in \mathbb{Z}^n} \mathcal{T}_{11} (\pmb{k} - \pmb{k}') \widehat{\pmb{z}}_1(\pmb{k}') \right \|   \right]^2  \left[ \sum_{\pmb{k} \in \mathbb{Z}^n} \|\widehat{\pmb{z}}_2(\pmb{k})\|^2    \right]        \\
                                                                                                        & \leq   \left[ \sum_{\pmb{k} \in \mathbb{Z}^n} \|\mathcal{T}_{11}(\pmb{k} )\|   \right]^2   \left[ \sum_{\pmb{k} \in \mathbb{Z}^n} \|\widehat{\pmb{z}}_1(\pmb{k})\|^2    \right]  \left[ \sum_{\pmb{k} \in \mathbb{Z}^n} \|\widehat{\pmb{z}}_2(\pmb{k})\|^2    \right]  < \infty                                                              
 \end{align*}
 Similarly, we can establish the corresponding bounds for the tensors $\mathcal{T}_{21}$, $\mathcal{T}_{12}$, and $\mathcal{T}_{22}$. By summing these components, we establish the overall bound~\eqref{eqn: tensor}.

\section{Proofs of results in~\Cref{sec: small}}
\label{sec: small-app}

In this appendix, we provide the detailed proofs for the results presented in~\Cref{sec: small}, specifically addressing~\Cref{lem: mes-dig} and~\Cref{thm: small-size-original}.

\subsection{Proof of~\Cref{lem: mes-dig}}
\label{sec: mes-dig}

Let $\Omega \in \mathbb{R}^n$ be a bounded domain, and define its diameter as $d = \max_{x,y \in \Omega} \| x- y \|_2$. For each fixed vector $\pmb{k} \in \mathbb{Z}^{n} \setminus \{ 0\}$, the nearly-resonance condition $ \left| \langle \pmb{k}, \pmb{\omega} \rangle\right|  <  | \pmb{k} |^{-\tau} $ defines a region in $\Omega$ bounded by two parallel hyperplanes perpendicular to $\pmb{k}$. These hyperplanes enclose a ``resonant strip'' or slab. The distance between these hyperplanes, representing the width of the resonant strip, is at most $2  \| \pmb{k} \|_2^{ - (\tau+1)}$. Therefore,  we can estimate the measure of  the resonant slab within the bounded domain $\Omega$ as:
\begin{equation}
\label{eqn: mes-single}
\mathrm{mes}\left( \left\{ \pmb{\omega} \in \Omega \bigg | \left| \langle \pmb{k}, \pmb{\omega} \rangle\right| < \frac{1}{ | \pmb{k} |^{\tau}} \right\} \right) \leq \frac{2  d^{n-1}}{ \| \pmb{k} \|_2^{\tau + 1}}.
\end{equation}
Since the exponent satisfies $\tau > n - 1$, the lattice series $ \sum_{\pmb{k} \in \mathbb{Z}^{n} \setminus \{ 0 \} } \| \pmb{k} \|_2^{-\tau - 1}$ converges.   Summing the estimates in~\eqref{eqn: mes-single}  yields the desired mesure bound~\eqref{eqn: small-mes}, which completes the proof.

\subsection{Proof of~\Cref{thm: small-size-original}}
\label{sec: small-size-original}

Before proceeding to the proof of~\Cref{thm: small-size-original}, we first establish a lemma regarding the preservation of the Gevrey decay property under operator multiplication. We say that an operator $A : \ell_2(\mathbb{Z}^{n}) \mapsto  \ell_2(\mathbb{Z}^{n})$ has \textit{proper decay} if its entries satisfy:
\[
\sup_{\pmb{k}, \pmb{k}' \in \mathbb{Z}^n} \left( 2|A(\pmb{k}, \pmb{k}')| \left( \exp\left\{ - |\pmb{k} - \pmb{k}'|^s \right\} + \exp\left\{ - |\pmb{k} + \pmb{k}'|^s \right\} \right)^{-1} \right) \leq 1.
\]

\begin{lemma}
\label{lemma: App-3}

Let $A_1$ and $A_2$ be two operators with proper decay. Then there exists a constant $C'_2(n, s) > 0$ such that their product $\mathcal{A} = A_1A_2$ satisfies 
\begin{equation}    
\label{eqn: convolution-mix}   
\sup_{\pmb{k}, \pmb{k}' \in \mathbb{Z}^n} \left( 2|\mathcal{A}(\pmb{k}, \pmb{k}')| \left( \exp\left\{ - |\pmb{k} - \pmb{k}'|^s \right\} + \exp\left\{ - |\pmb{k} + \pmb{k}'|^s \right\} \right)^{-1} \right) \leq C'_2(n,s)
\end{equation}  
 More generally, let $A_j$ for $j=1, \ldots,m$ be operators with proper decay. Then there exists a constant $C'_m(n, s) > 0$ such that the $m$-fold product $\mathcal{A}=\prod_{j=1}^{m}A_j $ satisfies:
 \begin{equation}    
\label{eqn: m-convolution-mix}   
\sup_{\pmb{k},\pmb{k}'  \in \mathbb{Z}^n} \left(  2|\mathcal{A}(\pmb{k}, \pmb{k}')|  \left( \exp\left\{ - |\pmb{k} - \pmb{k}'|^s \right\} + \exp\left\{ - |\pmb{k} + \pmb{k}'|^s \right\} \right)^{-1}  \right)  \leq C'_m(n,s)
\end{equation}   
\end{lemma}

\begin{proof}
Given that $A_1$ and $A_2$ are two operators with proper decay, the entry of their product is bounded by:
 \begin{align*}
     |(\mathcal{A})(\pmb{k}, \pmb{k}')| & \le \sum_{\pmb{k}'' \in \mathbb{Z}^n} |A_1(\pmb{k},\pmb{k}'')| |A_2(\pmb{k}'', \pmb{k}')|                             \nonumber          \\
                              & \leq \frac12\sum_{\pmb{k}'' \in \mathbb{Z}^n}  \left( \exp\left\{ - |\pmb{k} - \pmb{k}''|^s \right\} + \exp\left\{ - |\pmb{k} + \pmb{k}''|^s \right\} \right) \left(  \exp\left\{ - |\pmb{k} ''- \pmb{k}'|^s \right\} + \exp\left\{ - |\pmb{k}'' + \pmb{k}'|^s \right\}     \right)                  \nonumber          \\
                              & \le \frac12  \left( \exp\left\{-| \pmb{k} - \pmb{k}' |^{s} \right\} + \exp\left\{ - |\pmb{k} + \pmb{k}'|^s \right\}  \right) \\
                              & \mathrel{\phantom{\le}} \cdot \sum_{\pmb{k}'' \in \mathbb{Z}^n} \exp\left\{ - (2 - 2^{s}) \min\left\{ | \pmb{k} \pm \pmb{k}''|^{s}, \; | \pmb{k}'' \pm \pmb{k}' |^{s} \right\} \right\}.       
 \end{align*}
Since the sum of the Gevrey decay series converges in an $n$-dimensional lattice $\mathbb{Z}^n$, we establish the inequality~\eqref{eqn: convolution-mix}.  By induction, this same argument extends to any finite $m$-fold products,  thereby establishing the inequality~\eqref{eqn: m-convolution-mix}.
\end{proof}

We now complete the proof of~\Cref{thm: small-size-original}.  Since the diagonal operator $D_{N}$ admits the uniform lower bound~\eqref{eqn: uniform-lower-original-small}, it is invertible. This allows us to factor the restricted operator $(T + \varepsilon B)_{N}$ as follows:
\[
(T + \varepsilon B)_{N}= D_{N} + \varepsilon (B_{N} + S_{N}) = D_{N} \left[ I + \varepsilon D_{N}^{-1}(B_N+S_{N}) \right].
\]
Given the smallness of the perturbation parameter $\varepsilon$, the operator $(T + \varepsilon B)_{N}$ is invertible, and its inverse can then be expanded as a Neumann series: 
\begin{equation}
\label{eqn: T-nu-expansion}
(T+\varepsilon B)_{N}^{-1} = \left[ I + \varepsilon D_{N}^{-1}(B_N + S_{N}) \right]^{-1} D_{N}^{-1} = \sum_{\ell=0}^{\infty} \left[ - \varepsilon D_{N}^{-1}(B_N+S_{N})  \right]^{\ell} D_{N}^{-1}.
\end{equation}
By applying the bound for $\| D_{N}^{-1} \|_2$ from~\eqref{eqn: uniform-lower-original-small} alongside the estimates from~\Cref{prop: convolution} and~\Cref{prop: B-operator}, we derive the following estimate as
\[
\| T_{N}^{-1} \|_2 \leq \| D_{N}^{-1} \|_2  \left( \sum_{\ell =0}^{\infty} \| \varepsilon D_{N}^{-1}(S_{N} + B_{N})\|_2^{\ell} \right) \leq \frac{1}{\varepsilon_{N}},
\] 
which establishes the norm bound~\eqref{eqn: small-size-original-inverse}. To derive the off-diagonal decay, we rewrite the Neumann series~\eqref{eqn: T-nu-expansion} to isolate the higher-order terms:
\[
(T+\varepsilon B)_{N}^{-1} = D_{N}^{-1} -  \left[ \sum_{\ell=0}^{\infty} \left( - \varepsilon D_{N}^{-1}(S_{N} +B_N) \right)^{\ell} \right] \left(\varepsilon D_{N}^{-1}(S_{N}+B_{N}) D_{N}^{-1}\right).
\]
Utilizing the Gevrey decay properties for the entries of for operators $S$ and $B$ (\Cref{prop: convolution} and~\Cref{prop: B-operator}),~\Cref{lemma: App-3} implies that for any $\pmb{k} \neq \pm \pmb{k}'$:
\[
\left| \left( - \varepsilon D_{N}^{-1} (S_{N}+B_{N}) \right)^{\ell} \left(  \pmb{k}, \pmb{k}' \right) \right| \leq  \left(\frac{1}{2}\right)^{\ell+1} \left( \exp\left\{- | \pmb{k} - \pmb{k}' |^{s} \right\} +\exp\left\{- | \pmb{k} +\pmb{k}' |^{s} \right\}\right),
\]
By summing the corresponding Neumann series, we establish the inequality for the off-diognal entries~\eqref{eqn: small-size-off-diagnol}. The proof of~\Cref{thm: small-size-original} is now complete.

\subsection{Proof of~\Cref{prop: small-eig-one}}
\label{sec: small-eig-one}

We proceed by contradiction. Suppose that the diagonal operator $D^{\sigma_0 }_{N}$ admits two distinct diagonal entries whose absolute values are smaller than $(4n)^{-\tau}(N+1)^{-\tau}$.  Specifically, for any two distinct lattice points $\pmb{k}_1, \pmb{k}_2 \in \pmb{k}_0 + \Lambda_N$, assume
\begin{align*}
         & |\sigma_{0} + \langle \pmb{k}_1, \pmb{\omega}' \rangle -  \omega_{j_1} | <  \frac{1}{(4n)^{\tau}(N+1)^{\tau}}, \\
         & |\sigma_{0} + \langle \pmb{k}_2, \pmb{\omega}' \rangle -  \omega_{j_2} | <   \frac{1}{(4n)^{\tau}(N+1)^{\tau}},  
\end{align*} 
for some indices $j_1, j_2, = 1, \ldots, n$. According to~\Cref{prop: frequency-drift}, we apply the triangle inequality to the difference to obtain:
\[
|\langle \pmb{k}_1 - \pmb{k}_2, \pmb{\omega} \rangle - \omega_{j_1} +  \omega_{j_2} | < \frac{1}{(4n)^{\tau}(N+1)^{\tau}}  + 2 N \sqrt{\varepsilon}
\]
On the other hand, since the drifted frequency vector satisfies the non-resonance condition~\eqref{eqn: nearly-resonance}, we have the following lower bound:
\[
|\langle \pmb{k}_1 - \pmb{k}_2, \pmb{\omega} \rangle -  \omega_{j_1} +  \omega_{j_2} | \geq \frac{1}{| \pmb{k}_1 - \pmb{k}_2  - \pmb{e}_{j_1} + \pmb{e}_{j_2}|^{\tau}} \geq \frac{1}{(2n)^{\tau}(N+1)^{\tau}}
\]
By choosing the threshold $\varepsilon _0 > 0$ to be sufficiently small, the upper bound derived from our assumption directly contradicts this non-resonance lower bound. This contradiction ensures that at most one such index $\pmb{k}$ exists, thereby completing the proof of~\Cref{prop: small-eig-one}.

\section{Proofs of results in~\Cref{sec: completion}}
\label{sec: completion-app}

In this appendix, we provide the detailed proofs for the results presented in~\Cref{sec: completion}. Specifically, we complete the verification of~\Cref{cond: implementation} and provide the proof for~\Cref{lem: inverse-derivative}. 

\subsection{Technical estimates for the resolvent expansion}
\label{subsec: implement-left}

We proceed with the verification by applying the resolvent identity to two cases defined in \eqref{eqn: 10k-box} and \eqref{eqn: k-box}.  Our objective is to establish a general bound for every entry of $ T_N^{-1} $.

We first consider $ \pmb{k}_1 \in \Lambda_{5K} $ (and consequently, $ -\pmb{k}_1\in \Lambda_{5K} $), as described in \eqref{eqn: 10k-box}. 
The first term $ |T_{\Lambda_{10K}}^{-1}(\pmb{k}_1, \pmb{k}_2)| $ is controlled by $ \exp\left\{ (\log 10K)^{15} \right\} $.
To estimate the second term of in the RHS of \eqref{eqn: 10k-box}, we note that since $ \pmb{k}_4 \notin \Lambda_{10K} $, we have $ |\pmb{k}_1 - \pmb{k}_4| > 5K > (10K)^\frac{1}{2} $. By the inductive hypotheses of \Cref{cond: implementation}, we estimate $ |T_{\Lambda_{10K}}^{-1}(\pmb{k}_1, \pmb{k}_3)| $ by using \eqref{eqn: off-diagonal} if $ |\pmb{k}_1 \pm \pmb{k}_3| > (10K)^\frac{1}{2} $ and
by \eqref{eqn: inverse-grow} otherwise. Combining these with the off-diagonal decay of $ T(\pmb{k}_3, \pmb{k}_4) $, we arrive at the following estimate for \eqref{eqn: 10k-box} as
\begin{equation}
    \label{eqn: 10k-est-general}
    \begin{aligned}
            |T_{N}^{-1}(\pmb{k}_1, \pmb{k}_2)| 
            &\leq \exp\left\{ (\log 10K)^{15} \right\} \\
            &+ \frac{\varepsilon}{2} \sum_{\pmb{k}_4 \in \Lambda_{N} \setminus \Lambda_{10K}} \left( \exp\left\{ -\frac{|\pmb{k}_1 - \pmb{k}_4|^s}{2} \right\} + \exp\left\{ -\frac{|\pmb{k}_1 + \pmb{k}_4|^s}{2} \right\} \right) |T_{ N }^{-1}(\pmb{k}_4, \pmb{k}_2)|.
    \end{aligned}
\end{equation}

\noindent We now consider \eqref{eqn: k-box}, where $ \pmb{k}_1 \notin \Lambda_{5K} $ and the box $ \Gamma_\alpha $ is centered at $ \pmb{k}_1 $. Here, the first term $ |T_{\Gamma_\alpha}^{-1}(\pmb{k}_1, \pmb{k}_2)| $ is controlled by $ \exp\left\{ (\log K)^{15} \right\} $. Similarly, we distinguish two cases $ |\pmb{k}_1 \pm \pmb{k}_3| > K^\frac{1}{2} $ and $ |\pmb{k}_1 \pm \pmb{k}_3| \leq K^\frac{1}{2} $, estimating $ |T_{\Gamma_\alpha}^{-1}(\pmb{k}_1, \pmb{k}_3)| $ via~\Cref{thm: induction-multi-scale} to obtain

\begin{equation}
    \label{eqn: k-est-general}
    \begin{aligned}
            |T_{N}^{-1}(\pmb{k}_1, \pmb{k}_2)| 
            &\leq \exp\left\{ (\log K)^{15} \right\} \\
            &+ \frac{\varepsilon}{2} \sum_{\pmb{k}_4 \in \Lambda_{N} \setminus \Gamma_\alpha} \exp\left\{ -\frac{|\pmb{k}_1 - \pmb{k}_4|^s}{2} \right\} |T_{ N }^{-1}(\pmb{k}_4, \pmb{k}_2)|.
    \end{aligned}
\end{equation}

\noindent By taking the maximum of $ \pmb{k}_4 \in \Lambda_N $ for the term $ |T_{N}^{-1}(\pmb{k}_4, \pmb{k}_2)| $ in
both \eqref{eqn: 10k-est-general} and \eqref{eqn: k-est-general}, we obtain:

\begin{equation*}
    \begin{aligned}
        |T_{N}^{-1}(\pmb{k}_1, \pmb{k}_2)| 
        &\leq \exp\left\{ (\log 10K)^{15} \right\}  \\
        &\qquad + \left( \max_{\pmb{k}_4} |T_{N}^{-1}(\pmb{k}_4, \pmb{k}_2)| \right) \cdot \frac{\varepsilon}{2} \sum_{\pmb{k}_4 \in \Lambda_{N}} \left( \exp\left\{ -\frac{|\pmb{k}_1 - \pmb{k}_4|^s}{2} \right\} + \exp\left\{ -\frac{|\pmb{k}_1 + \pmb{k}_4|^s}{2} \right\} \right)  \\
        &\leq \exp\left\{ (\log 10K)^{15} \right\} + \frac{1}{2} \max_{\pmb{k}_4} |T_{N}^{-1}(\pmb{k}_4, \pmb{k}_2)|.
    \end{aligned}
\end{equation*}

\noindent This implies the following uniform bound:

\begin{equation}
    \label{eqn: est-general}
    |T_{N}^{-1}(\pmb{k}_1, \pmb{k}_2)| \leq 2\exp\left\{ (\log 10K)^{15} \right\}.
\end{equation}

To verify the off-diagonal decay \eqref{eqn: off-diagonal}, we assume
$ |\pmb{k}_1 \pm \pmb{k}_2| > N^\frac{1}{2} = K^5 $.
Since those boxes $ \Lambda_{10K} $ and $ \Gamma_\alpha $ are of size at most $ 10K < K^5 $, the first term in the RHS of both \eqref{eqn: 10k-box} and \eqref{eqn: k-box} vanish. 
Following the same argument above, we obtain

\begin{equation}
    \label{eqn: 10k-est}
    |T_{N}^{-1}(\pmb{k}_1, \pmb{k}_2)| \leq \frac{\varepsilon}{2} \sum_{\pmb{k}_4 \in \Lambda_{N} \setminus \Lambda_{10K}} \left( \exp\left\{ -\frac{|\pmb{k}_1 - \pmb{k}_4|^s}{2} \right\} + \exp\left\{ -\frac{|\pmb{k}_1 + \pmb{k}_4|^s}{2} \right\} \right) |T_{ N }^{-1}(\pmb{k}_4, \pmb{k}_2)|.
\end{equation}

\noindent and 

\begin{equation}
    \label{eqn: k-est}
    |T_{N}^{-1}(\pmb{k}_1, \pmb{k}_2)| \leq \frac{\varepsilon}{2} \sum_{\pmb{k}_4 \in \Lambda_{N} \setminus \Gamma_\alpha} \exp\left\{ -\frac{|\pmb{k}_1 - \pmb{k}_4|^s}{2} \right\} |T_{ N }^{-1}(\pmb{k}_4, \pmb{k}_2)|.
\end{equation}

\noindent Repeated application of the iterations \eqref{eqn: 10k-est} and \eqref{eqn: k-est} yields
the localization bound:

\begin{equation}
    \label{eqn: N-off-diagonal}
    | (T + \varepsilon B)_{N}^{-1}\left( \pmb{k}_1, \pmb{k}_2 \right) | \leq \exp\left\{ - \frac{|\pmb{k}_1 - \pmb{k}_2|^{s}}{2}\right\} +  \exp\left\{ - \frac{|\pmb{k}_1 + \pmb{k}_2|^{s}}{2}\right\}, \quad \text{for}~|\pmb{k}_1 \pm \pmb{k}_2| \geq N^{\frac12},
\end{equation}

\noindent which confirms~\eqref{eqn: off-diagonal}.

Finally, using \eqref{eqn: est-general} and \eqref{eqn: N-off-diagonal}, we calculate the $ l_1 $-
and $ l_\infty $-norm of $ T_N^{-1} $ to bound the $ l_2 $-norm of $ T_N^{-1} $ as 

\begin{equation}
    \label{eqn: N-norm}
    \| T_N^{-1} \|_2 \leq 4\exp\left\{ (\log 10K)^{15} \right\} \leq \exp\left\{ (\log N)^{15} \right\},
\end{equation}
which matches~\eqref{eqn: inverse-grow}, thereby completing the verification of~\Cref{cond: implementation}.

 \subsection{Proof of~\Cref{lem: inverse-derivative}}
 \label{subsec: derivative-app}

From~\Cref{prop: convolution} and~\Cref{prop: B-operator}, it follows that for any $\pmb{k} \neq \pmb{k}'$, the off-diagonal entries of $\partial T_N$ satisfy
\[
|\partial (T + \varepsilon B)_N(\pmb{k}, \pmb{k}')| \leq \left( \exp\left\{ - | \pmb{k} - \pmb{k}' |^{s} \right\} + \exp\left\{ - | \pmb{k} + \pmb{k}' |^{s} \right\}\right).
\]
Following the operator identity in~\eqref{eqn: t-n-derivative-inverse}, we expand $\partial (T + \varepsilon B)_{N}^{-1}(\pmb{k}, \pmb{k}')$ as a double sum as:
\[
\partial (T + \varepsilon B)_{N}^{-1}(\pmb{k}, \pmb{k}') = - \sum_{\pmb{k}_1 \in  \Lambda_N} \sum_{\pmb{k}_2 \in  \Lambda_N}(T + \varepsilon B)_{N}^{-1}(\pmb{k}, \pmb{k}_1) (\partial (T + \varepsilon B)_{N} )(\pmb{k}_1, \pmb{k}_2) (T + \varepsilon B)_{N}^{-1}(\pmb{k}_2, \pmb{k}')
\]
By applying the triangle inequality, we obtain: 
\begin{align*}
&|\partial (T + \varepsilon B)_{N}^{-1}(\pmb{k}, \pmb{k}')| \\  
&\leq \sum_{\pmb{k}_1 \in  \Lambda_N} \sum_{\pmb{k}_2 \in  \Lambda_N } |(T + \varepsilon B)_{N}^{-1}(\pmb{k}, \pmb{k}_1)| \cdot |\partial (T + \varepsilon B)_{N} (\pmb{k}_1, \pmb{k}_2)| \cdot |(T + \varepsilon B)_{N}^{-1}(\pmb{k}_2, \pmb{k}')| \\
&\leq \frac{4N^2 (2N+1)^{2n}}{\varepsilon_N^2}  \left(\exp\left\{ - \frac{\left( |\pmb{k} - \pmb{k}'| - 2N^{\frac12} \right)^{s}}{2} \right\} + \exp\left\{ - \frac{\left( |\pmb{k} + \pmb{k}'| - 2N^{\frac12} \right)^{s}}{2} \right\} \right),
\end{align*}
where the last inequality follows~\Cref{lemma: App-3}.  Given that the spatial separation satisfies $|\pmb{k} \pm \pmb{k}'| \geq N^{\frac34}$, the off-diagonal entries of $\partial T_{N}^{-1}$ satisfies~\eqref{eqn: inverse-restricted-entries}. The proof thus is complete.

\end{document}